\documentclass[a4paper,oneside,11pt]{amsart}
\usepackage[utf8]{inputenc}
\allowdisplaybreaks
\usepackage{amsfonts}
\usepackage{amsmath}
\usepackage{amssymb}
\usepackage{amsthm}
\usepackage{geometry}
 \usepackage{lineno}
\usepackage{mathrsfs} 
\usepackage[dvipsnames]{xcolor}

\usepackage{bbm}
\usepackage{mathtools}

\usepackage{dsfont}
 
\usepackage[pagebackref=true,colorlinks=true, linkcolor=Blue, citecolor=Blue]{hyperref}
\renewcommand*\backref[1]{\ifx#1\relax \else (Cited on #1) \fi}

\usepackage{appendix}
\usepackage{enumerate}
\usepackage{float}

\usepackage{tikz}
\usepackage{tikz-cd} 
\usetikzlibrary{arrows, arrows.meta}

\usepackage[nameinlink]{cleveref} 
\usepackage{scalerel}[2016/12/29]

\theoremstyle{plain}
\newtheorem{definition}{Definition}
\newtheorem{proposition}[definition]{Proposition}
\newtheorem{lemma}[definition]{Lemma}
\newtheorem{corollary}[definition]{Corollary}
\newtheorem{theorem}[definition]{Theorem}
\newtheorem{remark}[definition]{Remark}

\theoremstyle{definition}

\numberwithin{definition}{section}
\numberwithin{equation}{section}

\DeclareMathOperator{\sgn}{sgn}
\DeclareMathOperator{\dist}{dist}

\DeclareMathOperator{\supp}{supp}

\newcommand*{\E}{\mathbb{E}}
\newcommand*{\R}{\mathbb{R}}

\newcommand*{\Z}{\mathbb{Z}}

\newcommand*{\N}{\mathbb{N}}

\newcommand*{\Fcal}{\mathcal{F}}

\newcommand*{\Pcal}{\mathcal{P}}

\newcommand*{\Lcal}{\mathcal{L}}

\newcommand*{\Ncal}{\mathcal{N}}

\newcommand{\abs}[1]{\left\lvert #1 \right\rvert}
\newcommand*{\Tl}{T\Ssup{\Lambda}}
\newcommand*{\Xl}{X\Ssup{\Lambda}}

\newcommand*{\GG}{\mathscr{G}}

\newcommand*{\RR}{\mathscr{R}}
\newcommand*{\LL}{\mathscr{L}}

\newcommand*{\1}{\mathds{1}}

\renewcommand*{\L}{\Lambda}

\renewcommand*{\d}{\mathrm{d}}
\newcommand*{\e}{\mathrm{e}}

\newcommand*{\range}{R}

\newcommand*{\SpecEnt}{\mathscr{I}}
\newcommand*{\RelEnt}{I}

\newcommand{\Ssup}[1]{^{\scriptscriptstyle{({#1}})}} 

\newcommand{\sfbf}[1]{{\boldsymbol{\mathsf{#1}}}}

\newcommand*{\birth}{{\sfbf{b}}}

\newcommand{\birthsinL}[1]{B_\L(#1)}
\newcommand{\deathsinL}[1]{D_\L(#1)}
\newcommand{\deathsinLfromomega}[1]{D^{\omega}_\L(#1)}
\newcommand{\deathsinLfrombirths}[1]{D^{2}_\L(#1)}

\newcommand{\rightDerivative}[1]{\frac{\d}{\d #1^+}}

\makeatletter
\newcommand{\subalign}[1]{%
  \vcenter{%
    \Let@ \restore@math@cr \default@tag
    \baselineskip\fontdimen10 \scriptfont\tw@
    \advance\baselineskip\fontdimen12 \scriptfont\tw@
    \lineskip\thr@@\fontdimen8 \scriptfont\thr@@
    \lineskiplimit\lineskip
    \ialign{\hfil$\m@th\scriptstyle##$&$\m@th\scriptstyle{}##$\hfil\crcr
      #1\crcr
    }%
  }%
}
\makeatother

\crefname{equation}{}{}

\title[Reversible birth-and-death dynamics in continuum]{Reversible birth-and-death dynamics in continuum: \\ a de Bruijn-type identity for free-energy dissipation}
\author[B. Jahnel]{Benedikt Jahnel}
\author[J. K\"oppl]{Jonas K\"oppl}
\author[Y. Steenbeck]{Yannic Steenbeck}
\author[A. Zass]{Alexander Zass}

\address[Benedikt Jahnel]{TU Braunschweig, Institut für Mathematische Stochastik,
Germany, and Weierstrass Institute, Berlin, Germany.}
\email{benedikt.jahnel@tu-braunschweig.de}
\address[Jonas K\"oppl]{Weierstrass Institute, Berlin,
Germany.}
\email{koeppl@wias-berlin.de}
\address[Yannic Steenbeck]{TU Braunschweig, Institut für Mathematische Stochastik, 
Germany.}
\email{yannic.steenbeck@tu-braunschweig.de}
\address[Alexander Zass]{Weierstrass Institute, Berlin,
Germany.}
\email{zass@wias-berlin.de}
\date{\today}


\date{\today}

\definecolor{amethyst}{rgb}{0.6, 0.4, 0.8}

\usepackage[textwidth=20mm,textsize=small]{todonotes}

\usepackage{soul}

\begin{document}

\begin{abstract}
We investigate free-energy dissipation in a continuous-time birth-and-death dynamics in $\R^d$. For these Markov processes, the class of reversible measures coincides with the infinite-volume Gibbs point processes for some sufficiently nice Hamiltonian. For a wide class of initial distributions, we derive a de~Bruijn-type identity that relates the time evolution of the specific relative entropy along trajectories to the Fisher information, in particular establishing the thermodynamic limit of the latter. Along the way, we analyze some fine properties of the considered dynamics, such as the existence and regularity of local densities, obtain a spatial ergodic theorem for the entropy production per unit volume, and derive a small-time exponential series expansion of the dynamics. 
\end{abstract}

\maketitle

\setcounter{tocdepth}{1}
\hypersetup{bookmarksdepth=2}

{
  \hypersetup{linkcolor=black}
  \tableofcontents
}

\section{Introduction}
Gibbs point processes on \(\mathbb{R}^d\) serve as the continuum analogues to Gibbs measures on the lattice \(\mathbb{Z}^d\). These processes provide a fundamental framework for modeling thermodynamic equilibrium, with several intersections between the disciplines of mathematical statistical mechanics, stochastic geometry, and general probability theory. 
Extensive research has been conducted on various non-trivial models, answering fundamental questions of equilibrium statistical physics, including: existence of infinite-volume Gibbsian point processes, e.g.,~\cite{ruelle_1969,Dereudre_Vasseur_2020,Roelly_Zass_2020}, uniqueness, e.g.~\cite{ruelle_1969,Fernndez2007,Houdebert2022,Betsch2023}, phase transitions or non-uniqueness, e.g.,~\cite{Ruelle1971WRPT,Giacomin1995Agreement, CCK95,Georgii_1996,Dereudre2025,dereudre2026}, the equivalence of various ensembles,~\cite{georgii_equivalence_1995}, and validity of the Gibbs variational principle, e.g.,~\cite{georgii_equivalence_1995,dereudre2009variational,JKSZ24}.
The main goal of this work is to extend the existing equilibrium framework by establishing initial results for an associated stochastic dynamics. In fact, relative to the extensive literature on interacting particle systems on the lattice (cf.~\cite{liggett_interacting_2005}), the study of out-of-equilibrium behavior for Gibbsian configurations in the continuum setting remains comparatively underdeveloped. We address this gap by analyzing the time evolution and convergence-to-equilibrium properties of a continuum-space birth-and-death dynamics.

\subsection{Infinite-volume dynamics and convergence to equilibrium}

At least on a formal level, the {\em infinite-volume Gibbs measures} associated with the {\em energy functional} \(H\) are expected to be reversible with respect to the {\em infinite-volume birth-and-death process} on the space of point configurations in \(\mathbb{R}^d\) characterized by the generator $\LL$, which acts on a suitable domain of test functions as:
\begin{equation*}
    (\LL f)(\eta) := \int_{\R^d} \d x \ b(x,\eta)\big( f(\eta+\delta_x) - f(\eta) \big) + \sum_{x\in\eta} \big( f(\eta-\delta_x) - f(\eta)\big),
\end{equation*} 
if the {\em birth-rates} \(b\) are given by 
\begin{align*}
    b(x,\eta) := \e^{-h(x,\eta)}
\end{align*} 
and the {\em conditional energy} is defined as
\begin{align*}
    h(x,\eta) := H(\eta + \delta_x) - H(\eta).
\end{align*} 
This is a direct consequence of the {\em GNZ equations}~\cite{xanh1979integral}, which uniquely characterize Gibbs point processes.
However, a rigorous construction of the associated dynamics poses significant challenges, primarily due to the unbounded local density of particles, which is inherent to point processes in the continuum. Consequently, existence of well-defined dynamics cannot be directly deduced from Liggett's classical results (see~\cite{penrose2008spatial}).
For a more exhaustive survey of the developments regarding the construction of point process dynamics, we refer the reader to~\cite[Section 1.1]{JKSZ25}. We only mention that they include works both on the operator-theoretic and stochastic aspects of the theory (e.g.,~\cite{Preston1975, HS78,KonSko06Contact,Kond2008,Finkelshtein2014Dynamical, Kurtz1980Representations,Garcia1995Birth,garcia2006spatial, Etheridge2019Genealogical}). In the present work, our approach is based on the graphical construction framework of \cite{garcia2006spatial}. This method, allowing for a path-wise definition of the process, is explained in more detail in \Cref{section:setting_dynamics} below.

While the aforementioned literature provides a robust framework for the construction of infinite-volume dynamics and the description of their (infinitesimal) invariant measures, they offer limited insight into the out-of-equilibrium properties of the system, i.e., when the initial configuration is not sampled from, or globally absolutely continuous with respect to, the equilibrium measure. Specifically, aspects relating to free-energy dissipation and asymptotic stability of the system (weak limit points) remain largely unexplored.

These two lines of investigation are deeply intertwined. In the finite-state-space setting, the relation can be illustrated by the classical theory of a continuous-time irreducible Markov chains with generator \(\mathfrak{L}\) as follows. For a reversible measure \(\nu\), the {\em relative entropy} of a probability measure \(\mu\) w.r.t.\ \(\nu\) is defined as
\begin{align*}
    I(\mu \lvert \nu) = \sum_{x} \mu(x) \log \frac{\mu(x)}{\nu(x)},
\end{align*} 
where \(\mu\mapsto I(\mu \lvert \nu)\) is non-negative and vanishing if and only if \(\mu = \nu\). Under the dynamics, the map \(t \mapsto I(\mu T_t \lvert \nu)\) is strictly decreasing if $\mu \neq \nu$.
Consequently, $\mu \mapsto I(\mu \lvert \nu)$ is a strict Lyapunov function for the unique fixed point $\nu$ of the measure-valued evolution
 \begin{align*}
     \partial \mu_t = \mu_t \mathfrak{L}.
 \end{align*} 
 This yields the convergence of \(\mu_t = \mu T_t\) to the equilibrium measure \(\nu\) as \(t \rightarrow \infty\).

This relative-entropy strategy has been successfully adapted to the more complex setting of interacting particle systems on the lattice \(\mathbb{Z}^d\), both for reversible and non-reversible dynamics, cf.~\cite{holley_free_1971,higuchi_results_1975, Sullivan1976,kunsch_time_1984,jahnel_attractor_2019, jahnel_dynamical_2023,JK25}.

\subsection{Contributions of this work}

In \cite{JKSZ25} we started a program focused on the statistical mechanics and long-time behavior of birth-and-death dynamics in the continuum. A key result of that work is a free-energy dissipation inequality, previously observed only for non-interacting systems, cf.~\cite{Dello2024Wasserstein,huesmann2025}. In the present work, we significantly refine this result by showing that the aforementioned dissipation inequality is, in fact, an identity. To the best of our knowledge, this constitutes the first derivation of a \emph{de Bruijn-type} identity for continuum-space birth-and-death processes in the interacting case.

From a technical perspective, the improvement to an equality is achieved by working directly with the original \textit{global} Markov process rather than employing local approximations. Central here is a semi-explicit construction of the local densities for the evolved process. Such a construction ensures that the densities exist and have the regularity required for taking time derivatives. We believe this may be of independent interest for future studies of the dynamics of point processes. Furthermore, we establish the convergence of a small-time series expansion for the dynamics, which allows us to prove that the system cannot reach a Gibbs state in finite time unless the initial distribution is already Gibbs. Although our analysis focuses on a birth-and-death process with rates that come from interactions that satisfy precise assumptions (namely continuity, boundedness from above and below, and finite range), like the area interaction, we believe that the methodology developed here could serve as a framework for broader investigations of interacting particles systems in the continuum.

\subsection{Outline of the paper}
The remainder of the paper is structured as follows.
In \Cref{section:setting_and_main_result}, we introduce our setting and present the main results, with an outlook towards future research directions presented in \Cref{section:outlook}. Before we dive into the technicalities of the proof, we give a brief outline of the proof strategy in \Cref{section:strategy}. 
In \Cref{section:domain_of_generator}, we then go on to check that the (formal) forward Kolmogorov equation, involving the formal generator \(\mathscr{L}\) and the semigroup \((T_t)_{t \geq 0}\) of our birth-and-death-dynamics, holds pointwise for local functions.
\Cref{section:density_properties_of_bd_process} provides an analysis of the regularity and continuity properties of (the distribution of) the birth-and-death process \((Y_t^{\underline\omega})_{t \geq 0}\) given a fixed initial configuration. The main objective here is to construct suitable versions of the local densities to justify their time-differentiation in subsequent proofs.
The core of our results is contained in \Cref{section:entropy_dissipation}, which addresses entropy dissipation in two distinct steps:
In \Cref{section:entropy_dissipation_local_equality}, we derive a formula for the time derivative of the local entropies of the evolved measures; in \Cref{section:entropy_dissipation_thermodynamic_limit}, we establish the thermodynamic limit of these formulae. By piecing together these two parts, we are able to show a de~Bruijn-type identity for the free-energy dissipation in our setting.
In \Cref{section:gibbs_and_reversible_measures}, we conclude the proof of our main result with the observation that our dynamics cannot attain a Gibbsian state in finite time, unless the initial distribution itself is already a Gibbs measure. On the technical side, this is shown with an analyticity argument, using a small-time exponential series expansion. This, together with our de~Bruijn-type identity, implies that the free energy is strictly decreasing in time if the initial distribution is not already an equilibrium distribution.
Finally, in \Cref{section:proof-quantitative-decay} we comment on how to apply the result from \cite{DaiPraPosta2013} in our situation to show that for high temperatures the decay of the specific relative entropy is exponentially fast.

\section{Setting and main results}\label{section:setting_and_main_result}

Throughout the article, where it helps readability, we will use the symbol $f\lesssim g$ if $f\le c g$ for a constant $c \in (0, \infty)$ independent of some properties of $f$ and $g$.
We write $f\lesssim_A g$ if we want to make clear that the implicit constant $c=c(A)$ depends on some parameter $A$, for example a subset $\L\subset \R^d$ or a time \(t > 0\).

\subsection{Point-process formalism}
Our analysis deals with probability measures on and Markov processes with values in the space
\begin{align*}
    \Omega:=\Big\{\omega= \sum_{i\ge 0}\delta_{x_i} \colon x_i\in\R^d \Big\}
\end{align*} of simple, locally finite {\em point configurations} (or measures) in $\R^d$.
We may identify an element \(\omega \in \Omega\) with its support,
\begin{align*}
    \omega = \sum_{i\ge 1} \delta_{x_i}
    \equiv \{x_1,x_2, \dots\}.
\end{align*}
For any subset $\L\subseteq\R^d$, we denote by $\omega_\L = \omega \cap \L$ the restriction of the configuration $\omega\in\Omega$ to $\L$. The space of all such point configurations in $\L$ is denoted by $\Omega_\L$. We will use the notation $\L\Subset\R^d$, or $\L\in\mathfrak{B}_b(\mathbb{R}^d)$, to say that \(\L\) is a bounded Borel subset of $\R^d$. We also fix a sequence of increasing boxes in positional space \(\mathbb{R}^d\), denoted throughout by
\begin{align*}
    \Lambda_n := \left[-n/2, n/2\right)^d, \qquad n \in \mathbb{N}:=\{1,2,\ldots\}.
\end{align*} 
In particular, it is important to note that, when taking infinite-volume (thermodynamic) limits, we will write $\Lambda\uparrow\R^d$ to mean that such a limit is taken as $n\to\infty$ along the sequence $(\Lambda_n)_{n\in\N}$.
For any $\L\Subset\R^d$, the {\em number of points} of $\eta$ in $\L$,
\begin{align*}
    \Omega \ni \eta \mapsto N_\L(\eta)\equiv |\eta_\L| \in [0,\infty),
\end{align*} 
is well-defined; in particular, $N_\Lambda(\eta)$ is a finite number for every $\eta \in \Omega$.
We endow $\Omega$ (and analogously \(\Omega_\Lambda\)) with the canonical $\sigma$-algebra $\mathcal{F}$ generated by the family of all these counting variables, i.e.,
\begin{align*}
    \mathcal{F} = \sigma\big(N_{\L} \colon \L \Subset \mathbb{R}^d \big).
\end{align*} 
A probability measure on the set of point configurations \(\Omega\) is also called \emph{point process} and the set of all these distributions is denoted by $\Pcal(\Omega)$. The restriction (or marginal) $\mu_\L\in\Pcal(\Omega_\L)$ of any $\mu\in\Pcal(\Omega)$ to $\L\Subset\R^d$ is simply given by
\begin{align*}
    \int f(\eta) \mu_\L(\d \eta)
    = \int  f(\eta_\Lambda)  \mu(\d \eta), \qquad f \text{ measurable}.
\end{align*}
A function \(f \colon \Omega \to \mathbb{R}\) is called {\em \(\Lambda\)-local} or {\em \(\Lambda\)-measurable}, for \(\Lambda \Subset \mathbb{R}^d\), if \(f\) is measurable w.r.t.\ \(\mathcal{F}_\Lambda = \sigma\big(N_{\Delta} \colon \Delta \Subset \L \big)\) or equivalently if \(f\) is measurable and \(f = f(\cdot_\Lambda)\). Consequently, a measurable function \(f \colon \Omega \to \mathbb{R}\) is {\em local} if there exists \(\Lambda \Subset \mathbb{R}^d\) such that \(f\) is \(\Lambda\)-measurable. We let $\Lcal$ denote the set of all bounded, local, measurable functions $f\colon\Omega \to \R$.
Finally, we denote by \(\theta_x\) the translation by \(x \in \mathbb{R}^d\), acting either on points in \(\mathbb{R}^d\) via \(y \mapsto y -x\) or on point configurations via \(\sum_{i\ge 1} \delta_{x_i} \mapsto \sum_{i\ge 1} \delta_{x_i - x}\).
A probability measure \(\mu \in \Pcal(\Omega)\) is said to be {\em translation invariant} if \(\mu \circ \theta_x^{-1} = \mu\) for all \(x \in \mathbb{R}^d\) and we write $\Pcal_\theta$ for the set of translation-invariant measures $\mu\in \Pcal(\Omega)$.

\subsection{Gibbs point processes}

Our investigation starts from the so-called {\em Gibbs point processes} that represent thermodynamic equilibrium states in the form of probability measures on the space of point configurations $\Omega$. 
Denoting by $\Omega_f$ the space of finite configurations, consider an {\em energy functional} $H\colon \Omega_f\to \R\cup\{+ \infty\}$ such that for \(\L\Subset\R^d\) the {\em(conditional) energy in \(\Lambda\)}  
\begin{equation*}
H_\L(\eta):=\lim_{n\uparrow\infty}\big(H(\eta_{\L_n})-H(\eta_{\L_n\setminus\L})\big)
\end{equation*} 
is well-defined.
We restrict our study here to interactions that are (i) continuous, (ii) bounded from above and below, and (iii) have finite range, meaning that there exists some $\range>0$, the \emph{interaction range}, such that
\begin{equation*}  H_\L(\eta)=H_\L(\eta_{\L\oplus B(0,\range)}),\quad\text{where } \L\oplus B(0,\range):=\{x\in\R^d\colon \dist(x,\L)\leq \range\}.
\end{equation*}
Our guiding example is the {\em area interaction},
which is given, for any $\range>0$, by the Hamiltonian
\begin{equation*}
H(\eta)=\alpha\abs{B_{\range/2}(\eta)} := \alpha \Big\lvert \bigcup_{x\in \eta}B_{\range/2}(x) \Big\rvert,
\end{equation*} 
with $\alpha\in\R\setminus\{0\}$, and where $B_r(x)$ denotes the closed ball with radius $r>0$ centered at $x\in \R^d$.
Clearly, the interaction range of the area interaction is $\range$.
Note that the interaction is attractive for $\alpha>0$, meaning that points tend to cluster, while it is repulsive for $\alpha<0$, thus exhibiting a more ordered pattern.

Now, {\em finite-volume Gibbs point processes} in $\L\Subset\R^d$ with {\em boundary condition} $\omega\in \Omega$ are defined as
\begin{equation*}
\nu\Ssup{\omega}_\L(\d\eta_\L):=Z_\L(\omega)^{-1}\e^{-H_\L(\eta_\L\omega_{\L^c})}\pi_\L(\d \eta_\L),
\end{equation*}
where the partition function $Z_\L(\omega)$ is the usual normalization constant that makes \(\nu\Ssup{\omega}_\L(\d\eta_\L)\) a probability measure on $\Omega_\L$ and the reference measure $\pi$ is the homogeneous {\em Poisson point process} with intensity \(1\) on \(\Omega\).
A probability measure $\nu$ on $\Omega$ is called a (grand-canonical) {\em infinite-volume Gibbs point process} if it satisfies the {\em DLR equations}  
\begin{equation*}
    \int\nu(\d \omega)f(\omega)
    =
    \int \nu(\d \omega) \int\nu_\L\Ssup{\omega}(\d \eta_\L)\ f(\eta_\L\omega_{\L^c}),\quad \Lambda \Subset \R^d, \, f\ge0 \text{ measurable}.
\end{equation*}
Furthermore, we denote the set of all translation-invariant infinite-volume Gibbs point processes as $\GG_\theta$.
By standard results for Gibbs point processes (cf.~\cite{dereudre_2019}), the set $\GG_\theta$ is never empty, and it is well-known, see e.g.~\cite{Ruelle1971WRPT,Giacomin1995Agreement,CCK95}, that notably the area-interaction model exhibits a {\em phase transition} in the sense that $\GG_\theta$ contains more than one element for sufficiently large radii $R$.
Let us once again stress that our main results hold in regimes of uniqueness as well as non-uniqueness.

\subsection{Dynamics}\label{section:setting_dynamics}

We consider Markovian dynamics that are reversible with respect to elements of $\GG_\theta$. For this, let 
\begin{equation*}
    h(x,\eta):= \lim_{n\uparrow\infty}\big(H(\eta_{\L_n}+\delta_x) - H(\eta_{\L_n})\big)
\end{equation*} 
denote the {\em conditional energy} of a point $x\in\R^d$ in a configuration $\eta\in\Omega$, and define the {\em birth rate}, based on $h$, at $x$ given $\eta$, as
\begin{equation*}
    b(x,\eta) := \e^{-h(x,\eta)}.
\end{equation*}
\begin{remark}
    In the terminology of the GNZ equations, see e.g.~\cite{Georgii1976Canonical,xanh1979integral}, the birth rate is the \emph{Papangelou intensity} associated to the {\em Hamiltonian} $H$. Our assumptions on $H$ mean that the birth rate is also continuous and bounded from above and below; moreover, points at distance larger than $\range$ from $x\in\R^d$ obviously do not contribute to the birth rate at $x$. Observe that, for example, due to the attractive nature of the area interaction for $\alpha>0$, the birth rate is higher close to the points of $\eta$.
\end{remark}

Now, we consider the {\em birth-and-death process} associated to the (formal) generator
\begin{equation*}\label{eq:genBD}
    (\LL f)(\eta) := \int_{\R^d} \d x \ b(x,\eta)\big( f(\eta+\delta_x) - f(\eta) \big) + \sum_{x\in\eta} \big( f(\eta-\delta_x) - f(\eta)\big).
\end{equation*}
We can construct the associated Markov process $X\Ssup{\omega}=(X\Ssup\omega_t)_{t\ge 0}$, with starting configuration $X_0\Ssup{\omega}=\omega$, as a thinning of a space-time Poisson random measure, see~\cite[Theorem~2.13]{garcia2006spatial}. In the same vein (compare with the proof of ~\cite[Theorem~2.13]{garcia2006spatial}) we can first construct a Markov process $X\Ssup{\underline\omega}=(X\Ssup{\underline\omega}_t)_{t\ge 0}$, where the lifespans of the initial configuration are deterministic and already encoded in the marked version \(\underline{\omega}\) of \(\omega\).
\begin{proposition}[Graphical representation]\label{Definition:Markov_Process}
    Let $\Ncal$ be a Poisson random measure on $\R^d\times[0,\infty)^3$ with intensity measure $\d x\otimes \d u \otimes\e^{-r}\d r\otimes\d s$, $\omega = \sum_{i\ge 1}\delta_{x_i} \in \Omega$, and $\underline\omega = \sum_{i\ge 1}\delta_{(x_i,\tau_i)}$ be the point process on $\R^d\times[0,\infty)$ obtained by associating to each $x_i\in\omega$ a lifespan $\tau_i>0$.
    Suppose $\Fcal = (\Fcal_t)_{t\ge 0}$ is a filtration such that $\Ncal$ is $\Fcal$-compatible, i.e., such that for each bounded Borel set $A\subset\R^d\times [0,\infty)^2$, $\Ncal(A,\cdot)$ is $\Fcal$-adapted, and $\Ncal(A,t+s)-\Ncal(A,t)$ is independent of $\Fcal_t$, for any $s,t\geq 0$.
    Then, for any Borel measurable set $B \subset \R^d$, there exists a unique solution of
    \begin{equation}\label{eq:sdeBD}
    \begin{split}
        X\Ssup{\underline\omega}_t(B) = \int_{B\times [0,\infty)^2 \times [0,t]}\Ncal(\d x,\d u,\d r,\d s)\ &\1_{[0,b(x,X\Ssup{\underline{\omega}}_{s-})]}(u)\1_{(t-s,\infty)}(r) \\
        &+ \int_{B\times [0,\infty)}\underline\omega(\d x,\d r)\ \1_{(t,\infty)}(r).
    \end{split}
    \end{equation}
\end{proposition}

By randomizing the initial lifespans $\tau_i$ according to independent (of $\mathcal{N}$) unit-exponential random variables, we obtain the Markov process $X\Ssup{\omega}$ that corresponds to the unique solution of the martingale problem for $\LL$, see also~\cite[Theorem~4.4.2]{EK86}. Note that such a process will take values in the space of c\`adl\`ag functions on counting measures on $\R^d$, equipped with the Skorokhod $J_1$-topology. We will show how the semigroup of such a Markov process acts on a large class of functions.

\begin{remark}
    It is worth noting that, in our case, the Markov process could also be directly constructed as a limit of finite-volume approximations. Indeed, a ``finite speed of propagation'' result like \Cref{le_comparison_finite_infinite} below would yield that the coupling between the evolutions of two birth-and-death processes living in large but bounded nested volumes that incorporate the observation window has a sub-exponential (in the distance between the observation window and the boundary of the smaller one of the large volumes) failure probability in a fixed observation window.
\end{remark}

We denote the {\em semigroup} \((T_t)_{t \geq 0}\) associated to $X\Ssup{\omega}$ by
\begin{align*}
    (T_t f)(\omega):= \mathbb{E}[f(X_t\Ssup{\omega})],\qquad t\geq 0,
\end{align*}
for all $f$ for which the right-hand side is well-defined, 
and note that, as shown below in \Cref{lem:domain},
\begin{align*}
    \frac{\d}{\d s}T_s f\Big\vert_{s = t} = T_t \LL f , \qquad f\in \Lcal.
\end{align*}
This semigroup of course gives a dual semigroup that describes the evolution of a starting measure \(\mu\) under the birth-and-death-dynamics by
\begin{align*}
    (\mu T_t)[f]
    = \int \mu(\d\omega) \, (T_t f)(\omega), \qquad f \in \Lcal.
\end{align*}

Let $\RR_\theta$ denote the set of translation-invariant reversible measures for the above dynamics, i.e., the set of \(\mu \in \Pcal_\theta\) with \(\mu[(T_t f) \, g] = \mu[f \, (T_t g)]\) for all \(f, g \in \Lcal\).
The connection to infinite-volume Gibbs point processes is the following, proven in \Cref{section:reversible_measures_are_gibbs_measures} below.
\begin{proposition}\label{prop:reversible_equals_gibbs_with_the_right_definition}
    We have that $\GG_\theta = \RR_\theta$.
\end{proposition}

Finally, in order to distinguish the points of $X\Ssup{\omega}$ coming from the initial condition $\omega$ from those born after time $t=0$, we denote by \((\omega_t)_{t\geq 0}\) the pure-death process of the initial condition, i.e., the process in which the points of $\omega$ die after a unit-exponential random time, and consider the process $Y_t\Ssup{\omega}$, with
\begin{equation}\label{eqn:separation-initial-later}
X_t\Ssup{\omega}=\omega_t+Y_t\Ssup{\omega}.
\end{equation}
Analogously, if the initial lifespans are fixed, let \((\underline{\omega}_t)_{t\geq 0}\) be the deterministic pure-death process in which both the points of $\omega$ and their lifespans are fixed, and consider the process \(Y_t\Ssup{\underline\omega}\).

\subsection{Main results}

Besides being of interest on its own in mathematical statistical mechanics, the relative-entropy density backs up a Lyapunov-type approach to long-time behavior of the dynamics for arbitrary initial distributions. 
Hence, in order to state our main result, we need the notion of {\em relative entropy} of a measure \(\mu\) with respect to another measure \(\nu\), defined as
\begin{align*}
    \RelEnt(\mu|\nu) := \int\mu(\d\eta)\log\frac{\d \mu}{\d\nu}(\eta),
\end{align*} if the density $\d \mu/\d\nu$ exists and $\infty$ otherwise.
The {\em local relative entropy} for $\L\Subset\R^d$ is then
\begin{align*}
    \RelEnt_\L(\mu|\nu) := \RelEnt(\mu_\L|\nu_\L)
\end{align*} and the {\em relative entropy density}, also called \emph{specific relative entropy}, is defined as
\begin{align*}
    \SpecEnt(\mu|\nu)  :=\liminf_{n\uparrow\infty}\frac{1}{\abs{\L_n}}\RelEnt_{\L_n}(\mu|\nu),
\end{align*} which even exists as a limit if $\mu$ is translation invariant and \(\nu \in \GG_\theta\), see e.g.~\cite[Proposition~5.4]{JKSZ25}. The {\em Gibbs variational principle} then says that for \(\nu \in \GG_\theta\), the map \(\mu \mapsto \SpecEnt(\mu|\nu)\) does not depend on the choice of \(\nu\) and its zero set coincides with \(\GG_\theta\).

One advantage of the approach to long-time behavior via relative entropy densities is that it also works away from phase-uniqueness regimes and for initial distributions that are not absolutely continuous with respect to the reversible measure. 
We consider here starting measures with only local densities, but that are regular in the following sense.
\begin{definition}[Regular measures]
    \(\mu\in\Pcal_\theta\) is said to be \emph{regular} if for any bounded $\L\Subset\R^d$, there exists \(z_\Lambda \in (0, \infty)\) such that
    \begin{align*}
        z_\Lambda^{-N_{\Lambda}}
        \lesssim_\L 
        \frac{\d \mu_\Lambda}{\d \pi_{\Lambda}}
        \lesssim_\L z_\Lambda^{N_{\Lambda}}.
    \end{align*}
\end{definition}
\begin{remark}
    This assumption, while technically quite convenient, could probably be weakened to allow for more general measures. See e.g.~\Cref{rmk:pair}, for an example of how to adapt the proofs.
    Further, note that any Gibbs point process associated to an interaction satisfying the above Assumptions (i)--(iii) (also at some other temperature or for some other interaction range \(\range > 0\)), as well of course as any homogenous Poisson point process of some intensity \(z > 0\), is regular.
\end{remark}

\begin{remark}
    For any regular \(\mu\), the evolved measure \(\mu T_t\) is still regular for all \(t > 0\), as we will show in \Cref{corollary:regular_measures_stay_regular_under_evolution}.
\end{remark}

Next, if $\mu\in \Pcal_\theta$ has finite intensity, then there exists a unique measure $\mu_0$ on \(\Omega\), its {\em Palm measure}, such that
\begin{align*}
    \int\mu(\d \eta)\sum_{x\in \eta}f(x,\theta_x \eta)
    = \int \d x\, \mu_0[f(x,\cdot)],\qquad f\colon \R^d\times\Omega\to[0,\infty),
\end{align*} see for example~\cite{Kallenberg1983Random}.
The {\em reduced Palm measure} $\mu_0^!$ is then defined via \(\mu_0^! := \mu_0 \circ \gamma_0^{-1}\), where \(\gamma_0 \colon \eta \mapsto \eta - \delta_0\).

\medskip
We can now state our main result, which improves upon~\cite{JKSZ25} by showing that the free energy dissipation inequality established there is in fact an equality, yielding a de~Bruijn-type identity, cf.~\cite{Bakry2014,Dello2024Wasserstein}, that expresses the relative entropy dissipation in terms of the Fisher information. Here, the (rescaled) modified {\em Fisher information} is given by 
\begin{equation*}
    \mathcal{J}(\mu\vert\nu):=\lim_{n \uparrow \infty} \frac{1}{\abs{\Lambda_n}}\mathcal{J}_\L(\mu_\L\lvert \nu),
\end{equation*}
where the finite-volume modified Fisher information with \textit{free boundary conditions} is given by  
\begin{align*}
    \mathcal{J}_\L(\mu_\L\lvert \nu) = \mathcal{J}^{(\emptyset)}_\L(\mu_\L \lvert \nu) 
    := \int\d x\int \nu_\L(\d\eta)\  b(x,\eta_\L)\ D_x\frac{\d\mu_\L}{\d\nu_\L}(\eta_\L) \ D_x\log\frac{\d\mu_\L}{\d\nu_\L}(\eta_\L),
\end{align*} 
and where \((D_x f)(\eta) := f(\eta + \delta_x) - f(\eta)\) is the add-one-cost operator. In the course of our proof, we will show that the above limit exists under some mild assumptions on $\mu$ and does not depend on the choice of the boundary condition in a precise sense, see \Cref{lemma:equivalence_fisher_infos}.

\begin{theorem}[de Bruijn identity]\label{thm:main_thm}
    Let \(\nu\in \GG_\theta\) and \(\mu\in \Pcal_\theta\) a regular starting measure. 
    Then,
\begin{align}\label{eq:de_bruijn}
        \SpecEnt(\mu \,\vert\, \nu) - \SpecEnt(\mu T_t \,\vert\, \nu) = \int_0^t \, \mathcal{J}(\mu T_s \vert\nu) \, \d s
        = \int_{0}^{t} \, \xi^{\mu}(s) \, \d s,
    \end{align} 
    with 
    \begin{align*}
        \xi^\mu(s) := \RelEnt\big(b(0, \cdot) (\mu T_s) \,\big\vert\, (\mu T_s)_0^! \big) +  \RelEnt\big((\mu T_s)_0^! \,\big\vert\, b(0, \cdot) (\mu T_s) \big)\ge 0.
    \end{align*}
    Furthermore, \(\xi^\mu(s) = 0\) if and only if \(\mu \in \GG_\theta\). In particular, \(t \mapsto \SpecEnt(\mu T_t \,\vert\, \nu)\) is strictly decreasing for all times $t \geq 0$ if and only if \(\mu \not\in \GG_\theta\).
\end{theorem}

We will present the proof of \Cref{eq:de_bruijn} in \Cref{sec:de_bruijn_proof}.
The second statement in this theorem follows from \Cref{prop:not_reversible_in_finite_time}, showing that \(\mu T_s \in  \GG_\theta\) only if \(\mu \in  \GG_\theta\), together with \cite[Lemma~3.4]{JKSZ25}, which equivalently states that \(\xi^\mu(s) = 0\) if and only if \(\mu T_s \in  \GG_\theta\). 

For completeness we recall here an important consequence of the above identity \Cref{eq:de_bruijn} for the long-time behavior of the dynamics that we could already derive from the corresponding inequality in \cite[Theorem~2.4]{JKSZ25}. 

\begin{theorem}[Attractor property]\label{theorem:attractor-property}
    Let \(\mu\) be a regular starting measure. Then, for any $\mu^*$ such that there exists an increasing sequence of times $(t_k)_{k\ge 0}$, with $t_k\uparrow\infty$ and \(\lim_{k\uparrow\infty}\mu T_{t_k}=\mu^*\) in the $\tau_\mathcal{L}$-topology, we have that $\mu^*\in\GG_\theta$.
    In particular, \(\mu^*\) is reversible.
\end{theorem}
Finally, the following result states that in a high-temperature regime, the relative-entropy density decays exponentially fast.

\begin{proposition}\label{prop:quantitative-decay}
    There exists \(\beta_0 = \beta_0(\range, \Vert b\Vert_\infty) > 0\) such that for all \(\beta \leq \beta_0\), there exists \(\kappa(\beta) > 0\) such that
    \begin{align*}
        \SpecEnt(\mu T_t \,\vert\, \nu)
        \leq {\rm e}^{-\kappa(\beta) t} \SpecEnt(\mu \,\vert\, \nu)
    \end{align*} 
    for all (regular)  initial distributions \(\mu \in \Pcal_\theta\).
\end{proposition}

For \(\Lambda \Subset \mathbb{R}^d\) we recall from \cite[Definition 4.1]{JKSZ25} the local semigroup \(\Tl=\big(\Tl_t \big)_{t \geq 0}\) associated to the following (formal) generator
    \begin{align*}
        &(\LL_\Lambda^\nu f)(\eta_\Lambda)
        = \int_{\Lambda} \d x\, b^\nu_\Lambda(x, \eta_\Lambda) \big(f(\eta_\Lambda+\delta_x) - f(\eta_\Lambda)\big) + \sum_{x \in \eta_\Lambda} \big(f(\eta_\Lambda -\delta_x) - f(\eta_\Lambda)\big), 
    \end{align*} 
    with  $b^\nu_\Lambda(x, \eta_\Lambda):= \int \nu(\d \zeta_{\Lambda^c} |\eta_\Lambda)\, b(x, \eta_\Lambda\zeta_{\Lambda^c})$ and the {\em modified Fisher information} with boundary conditions sampled from $\nu$:
    \begin{align*}
         \mathcal{J}^{(\nu)}_\L(\mu_\L\lvert \nu) 
         := \int\d x\int \nu_\L(\d\eta)\  b^{\nu}_\L(x,\eta_\L)\ D_x\frac{\d\mu_\L}{\d\nu_\L}(\eta_\L) \ D_x\log\frac{\d\mu_\L}{\d\nu_\L}(\eta_\L),
    \end{align*} 
    which we could also write as \(\mathcal{E}_\Lambda(\d\mu_\L/\d\nu_\L, \log(\d\mu_\L/\d\nu_\L))\) for the Dirichlet energy 
    \begin{align*}
        \mathcal{E}_\Lambda(f, g)
        = \nu_\Lambda\Big[f (-\mathscr{L}^\nu_\Lambda) g \Big]
        = \int\d x\int \nu_\L(\d\eta)\  b^{\nu}_\L(x,\eta_\L)\ D_x f(\eta_\L) \ D_x g(\eta_\L).
    \end{align*} 
    Note that here it is more convenient to work with boundary conditions sampled from $\nu$ as opposed to free boundary conditions, but the choice of boundary conditions becomes irrelevant in the thermodynamic limit, see \Cref{lemma:equivalence_fisher_infos}. 
    So to conclude the the claimed exponential decay of the relative entropy density it therefore suffices to show that for any $\Lambda \Subset \R^d$, the term \(\mathcal{E}_\Lambda(\d\mu_\L/\d\nu_\L, \log(\d\mu_\L/\d\nu_\L))\) can be bounded in terms of the relative entropy in $\Lambda$, up to a multiplicative constant that does not depend on $\Lambda$. By adapting the strategy of~\cite{DaiPraPosta2013} to our slightly different class of models, this can be done at very high temperatures, i.e., $\beta \ll 1$. 
    \begin{lemma}\label{lemma:appendix_DaiPraPosta_exponential_decay_entropy_finite_volume}
    There is a \(\beta_0 = \beta_0(\range, \Vert b\Vert_\infty) > 0\) such that for all \(\beta \leq \beta_0\) it holds that
    \begin{align*}
        \mathcal{E}_\Lambda \Big(\frac{\d\mu_\L}{\d\nu_\L}, \, \log \frac{\d\mu_\L}{\d\nu_\L} \Big)
        \geq \kappa(\beta) \, \RelEnt_{\L}(\mu \lvert \nu) 
    \end{align*} for a \(\kappa(\beta) > 0\) and all (regular) initial distributions \(\mu \in \Pcal_\theta\) and all \(\Lambda \Subset \mathbb{R}^d\).
    \end{lemma}

    We defer the proof of \Cref{lemma:appendix_DaiPraPosta_exponential_decay_entropy_finite_volume} to \Cref{section:proof-quantitative-decay} but note that the statement of \Cref{prop:quantitative-decay} now follows from \Cref{eq:de_bruijn} and
    \begin{align*}
        \xi_s(\mu) 
        &=  \lim_{n \to \infty} \frac{1}{\abs{\Lambda_n}} \mathcal{J}^{(\nu)}_{\L_n}(\mu T_s \lvert \nu) 
        = \lim_{n \to \infty} \frac{1}{\abs{\Lambda_n} }  \mathcal{E}_{\L_n}\Big( \tfrac{\d (\mu T_s)_{\L_n}}{ \d\nu_{\L_n}}, \, \log  \tfrac{\d (\mu T_s)_{\L_n}}{ \d\nu_{\L_n}} \Big)  \\
        &\geq \kappa(\beta) \lim_{n \to \infty} \frac{1}{\abs{\Lambda_n}} \RelEnt_{\L_n}(\mu T_s \lvert \nu)  = \kappa(\beta) \SpecEnt(\mu T_s \,\vert\, \nu),
    \end{align*}
    using \Cref{lemma:equivalence_fisher_infos} for the first equality and \Cref{lemma:appendix_DaiPraPosta_exponential_decay_entropy_finite_volume} for the inequality.

\section{Outlook}\label{section:outlook}
Let us briefly comment on some possible refinements and extensions of our analysis. 
\begin{itemize}
\item The regularity of (versions of) the time-evolved log-densities and the resulting ability to work with them opens the door for a finer analysis of the long-time behavior, possibly even beyond translation-invariant dynamics and starting measures, see e.g.~\cite{holley_one_1977,JK25}, where fine estimates of these log-densities play a key role. 
\item Another interesting direction of future research would be to investigate the rate of entropy dissipation beyond the high-temperature regime in \Cref{prop:quantitative-decay}. Of course, one cannot expect exponential decay to hold at all inverse temperatures $\beta>0$, but is it at least possible to derive something quantitative as opposed to the purely qualitative statement in \Cref{theorem:attractor-property}? 
\item  Last but not least, one may ask how generic the existence of a reversible \textit{Gibbs} measure actually is. More precisely, under which conditions does a reversible birth-and-death dynamics actually admit a Gibbs measure as a reversible measure? Since Gibbsianity is essentially about the existence of a quasilocal version of the local conditional distributions, this should be interpreted as a statement about the transfer of regularity from the transition rates to the equilibrium measure. 
On the lattice this question has already been studied in~\cite{vossboehme} in the reversible case, and also for non-reversible dynamics in~\cite{kunsch_time_1984}.  
\end{itemize}
\section{Proof strategy}\label{section:strategy}
The proof of our main result \Cref{thm:main_thm} can roughly be split into two separate parts. First, the proof of the de~Bruijn-type identity itself and second, the proof that one cannot converge to a Gibbs measure in finite time if one starts outside of the Gibbs simplex. Let us start with the former. 

\subsection{The de~Bruijn-type identity}\label{sec:de_bruijn_proof}
As already noted in \cite[Section~6.1]{JKSZ25}, if one restricts the dynamics to finite-volumes then one can obtain the de~Bruijn-type identity 
\begin{align*}
    I_\Lambda(\mu \lvert \nu) - I_\Lambda(\mu \Tl_t) = \int_0^t \mathcal{J}_\Lambda(\mu \Tl_s\lvert \nu) ds, 
\end{align*}
where $\mathcal{J}_\Lambda$ is the (rescaled) modified Fisher information, by somewhat elementary means, see also \Cref{remark:local-approximation}. The main work in \cite{JKSZ25} was then devoted to comparing the error made by this local approximation of the global dynamics. This came at the cost of only obtaining an inequality instead of an identity in the thermodynamic limit. 

To mend this discrepancy between finite and infinite volumes, we therefore use a different strategy in this work. While still relying on local approximations of the global dynamics, we work much closer with the restriction of the original process to finite volumes. This allows us to establish the following entropy dissipation identity for the global dynamics in finite volumes $\Lambda\Subset\R^d$.

\begin{proposition}\label{theorem_derivative_of_the_local_entropy}
    Let \(\mu\) be regular. 
    Then,    \begin{align}\label{eqn:time-derivative-relative-entropy}
        \frac{\d }{\d t}\Big\vert_{t = 0} \RelEnt_\Lambda(\mu T_t \,\vert\, \nu)
        = \mu\Big[\mathscr{L} \log\frac{\d\mu_\Lambda}{\d \nu_\Lambda}\Big].
    \end{align}
\end{proposition}

While this is an elementary routine calculation for interacting particle systems on the lattice, there are quite some technical hurdles that need to be overcome for birth-and-death processes in continuum. A priori, it is not even clear if the time-propagated measure $\mu T_t$ has a local density with respect to $\nu$ (or the Poisson point process $\pi$), even if $\mu$ is regular.

\subsubsection{Regularity of log-densities}

Indeed, for lattice (spin-flip) dynamics, as e.g.\ considered in~\cite{holley_free_1971} or~\cite{handa1996entropy}, the relative entropy of the finite-volume restrictions is given by the finite sum
\[\RelEnt_\Lambda(\mu T_t \,\vert\, \nu)
    =
    \int_{\Omega}\mu T_t(d\omega) \log\frac{d(\mu T_t)_\Lambda}{d\nu_\Lambda}(\omega)
    =
    \sum_{\eta_\Lambda}\mu T_t[\mathbf{1}_{\eta_\Lambda}]\log\frac{\mu T_t[\mathbf{1}_{\eta_\Lambda}]}{\nu[\mathbf{1}_{\eta_\Lambda}]}, 
\]
where $\Omega_\Lambda = \{\pm 1 \}^\Lambda$, with $\Lambda \Subset \Z^d$, is the finite set of local configurations. Now note that in this discrete setting the indicator functions $\mathbf{1}_{\eta_\Lambda}$ are always in the domain of the generator of the spin-flip semigroup and \eqref{eqn:time-derivative-relative-entropy} directly follows from the chain rule. Under very mild assumptions on the dynamics there are moreover no regularity issues because $\nu$ puts strictly positive mass on every local configuration, see e.g.~\cite[Proposition~4.3]{JK25}, hence \textit{every} probability measure $\mu_\Lambda$ is absolutely continuous with respect to $\nu_\Lambda$. This is of course too much to hope for in the continuum. 

From a slightly different point of view, the reason why there are no regularity issues in the discrete setting is that $\nu_\Lambda$ is equivalent to a nice reference measure to which every other measure is at least absolutely continuous:  the counting measure. 
In continuum, this leads us to consider the next best thing: using the Poisson point process \(\pi_\Lambda\) as a reference measure. 

In order to use this to our advantage,  we (a) have to make sure that \(\mu T_t\) has local densities \(\d (\mu T_t)_\Lambda/\d \pi_\Lambda\), or equivalently \(\d (\mu T_t)_\Lambda/\d \nu_\Lambda\), if the local densities \(\d \mu_\Lambda/\d \nu_\Lambda\) exist in the first place, and (b) have good enough regularity properties to differentiate them in time.

\begin{remark}\label{remark:local-approximation}
In the previous work~\cite{JKSZ25}, we worked around this issue by considering instead of the semigroup \((T_t)_{t \geq 0}\) the local evolutions \((\Tl)_{t \geq 0}\) in \(\Lambda\) with random boundary configurations sampled from \(\nu\), thus having reversible measure \(\nu_\Lambda\). There, the corresponding Markov processes \(\Xl\) are truly Markov processes living in the volume \(\Lambda\), which is of course not true for the restrictions of \(X\) to the observation window \(\Lambda\). In fact, \((\Tl)_{t \geq 0}\) is a \(C_0\)-semigroup on the space \(L^1(\nu_\Lambda)\) by the pointwise convergence \(\Tl_t f \to f\), as $t \downarrow 0$, and uniform integrability of \((\Tl_t f)_{t \geq 0}\) due to the de~la~Vallée--Poussin criterion, Jensen's inequality and stationarity.
Consequently, the reversibility of \(\nu_\Lambda\) here gets us the simple identity 
\[\frac{\d (\mu \Tl_t)_\Lambda}{\d \nu_\Lambda} = \Tl_t \frac{\d \mu_\Lambda}{\d \nu_\Lambda}.\] 
The price we had to pay is just having an asymptotic comparison of \(\RelEnt_\Lambda(\mu T_t \,\vert\, \nu)\) and \(\RelEnt_\Lambda(\mu \Tl_t \,\vert\, \nu)\) in form of an inequality.
Here, we work much closer with the original process \(X\) and construct good (versions of the) densities  \(\d (\mu T_t)_\Lambda/\d \nu_\Lambda\). 
\end{remark}

Assuming that all of the involved densities exist, a formal calculation shows that 
\begin{equation}\label{eqn:proof-motivation}
\begin{split}
    \frac{1}{t} \big(\RelEnt_{\Lambda}(\mu T_t \,\vert\, \nu) -& \RelEnt_{\Lambda}(\mu \,\vert\, \nu) \big)
        = 
        \frac{1}{t} \Big(\mu T_t \Big[\log\frac{d\mu_\Lambda}{d\nu_\Lambda}\Big] - \mu\Big[\log\frac{d\mu_\Lambda}{d\nu_\Lambda}\Big]+ \mu T_t\Big[\log\frac{d(\mu T_t)_\Lambda}{d\nu_\Lambda}-\log\frac{d\mu_\Lambda}{d\nu_\Lambda}\Big]\Big) \\\
        &=\frac{1}{t} \Big(\mu T_t \Big[\log\frac{d\mu_\Lambda}{d\nu_\Lambda}\Big] - \mu\Big[\log\frac{d\mu_\Lambda}{d\nu_\Lambda}\Big]+ 
        \pi_\Lambda \Big[\frac{d(\mu T_t)_\Lambda}{d\pi_\Lambda}\Big(\log\frac{d(\mu T_t)_\Lambda}{d\nu_\Lambda}-\log\frac{d\mu_\Lambda}{d\nu_\Lambda}\Big)\Big]\Big).
\end{split}
\end{equation}
This splits our problem into two separate parts. First, we deal with the first two terms and show that if $\mu$ is a regular measure, then the time-derivative behaves as one would expect from standard semigroup theory (which we cannot directly apply here). 

\begin{lemma}\label{lem:log_generator}
    Let \(\mu\) be regular and set \(g_\Lambda := \d \mu_\Lambda/\d\nu_\Lambda\).
    Then,
    \begin{align}
        \frac{1}{t}\big((\mu T_t)[\log g_\Lambda] - \mu[\log g_\Lambda]\big)
        \xrightarrow[t \downarrow 0]{} \mu[\mathscr{L}(\log g_\Lambda)].
    \end{align} 
\end{lemma}

The proof can be found in \Cref{section:domain_of_generator}.
Dealing with the third term in \eqref{eqn:proof-motivation} is harder and requires the construction of sufficiently nice versions of the Radon--Nikodym derivatives, which allow us to take the time-derivative \textit{configuration-wise}. 
\begin{proposition}\label{prop:differentiability-of-densities}
    There is a measurable function \((t, \omega) \mapsto g_{\Lambda, t}(\omega)\) such that for each \(t \geq 0\), \(g_{\Lambda, t}\) is a version of \(\d (\mu T_t)_\Lambda/\d \nu_\Lambda\) and such that for fixed \(\omega\), \(t \mapsto g_{\Lambda, t}(\omega)\) is differentiable.
\end{proposition}

This strong regularity property of the densities can be seen as the main technical contribution of this manuscript and opens the door for further investigations on the dynamics of point processes. The proof relies on a careful analysis of the dynamics given in terms of the graphical representation. Recall from \eqref{eqn:separation-initial-later} that the graphical representation allows us to separate the configuration of the process at time $t$ into two parts: the points coming from the initial condition and those born after time $t=0$. 
In \Cref{section:density_properties_of_bd_process}, for the latter part, i.e., the points not already present in the initial condition, we  establish that for any fixed volume $\Lambda \Subset \R^d$ and for a fixed initial configuration $\underline{\omega}$ with deterministic lifespans, the restriction of the infinite-volume process $(Y^{(\underline{\omega})}_t)_{t \geq 0}$ to $\Lambda$ possesses densities $\psi_{\Lambda,t}(\cdot \lvert \underline{\omega})$ with respect to the Poisson point process $\pi_\Lambda$ with good regularity in time. 
This is then used in \Cref{section:entropy_dissipation_local_equality} to prove \Cref{prop:differentiability-of-densities}.
Once we have this plus some uniform integrability one can conclude \Cref{theorem_derivative_of_the_local_entropy} starting from \eqref{eqn:proof-motivation}, see \Cref{sec:proof-theorem-derivative-local-entropy}. 

\subsubsection{Thermodynamic limit and de~Brujin-type identity}
The intermediate aim is now to prove a continuum analogue of~\cite[Theorem~1]{handa1996entropy}. This allows us to identify the thermodynamic limit of the entropy production \(\mu[(-\mathscr{L}) \log(\d\mu_\Lambda/\d \nu_\Lambda)]\) in the volume \(\Lambda\).
Our proof is strongly inspired by the one given in \cite{handa1996entropy}, but we have to identify the correct analogues in all computations and take care of additional technical subtleties that arise due to the continuum setting. Additionally, we make the observation in \Cref{sec:de_bruijn_proof} that at least in our setting we can commute limit operations and hence get identities for the derivative of the specific relative entropy \(\SpecEnt(\mu T_t \,\vert\, \nu)\) from the thermodynamic limit of the derivatives of local relative entropies \(\RelEnt(\mu T_t \,\vert\, \nu)\).
Let us now state the main result of this step on our way towards the proof of the main result.

\begin{proposition}\label{theorem:main_theorem_HANDA}
    Let \(\mu\in \Pcal_\theta\) have a finite first moment \(\mu[N_{[0,1]^d}] < \infty\) and let \(\mu\) be such that \(\mu_\Lambda\) and \(\pi_\Lambda\) are equivalent for all \(\Lambda \Subset \mathbb{R}^d\).
    Denote 
    \begin{align}
        \sigma_\Lambda(\mu) 
        := \mu\Big[(-\mathscr{L}) \log\frac{\d\mu_\Lambda}{\d \nu_\Lambda}\Big]
    \end{align} if the expectation exists and otherwise set \(\sigma_\Lambda(\mu) = +\infty\).
    Then, the limit
    \begin{align}
        \sigma(\mu) 
        = \lim_{\Lambda \uparrow \mathbb{R}^d} \frac{1}{\abs{\Lambda}} \sigma_\Lambda(\mu) 
    \end{align} 
    exists and 
    \begin{align}
        \sigma(\mu) = \RelEnt\big(b(0, \cdot) \mu \,\vert\, \mu_0^!\big) + \RelEnt\big(\mu_0^! \,\vert\, b(0, \cdot) \mu \big).
    \end{align}
    Furthermore, if \(\sigma(\mu) < \infty\), we have
    \begin{align}
        \frac{1}{\abs{\Lambda}}(-\mathscr{L}) \log\frac{\d \mu_\Lambda}{\d \nu_\Lambda}
        \xrightarrow[\Lambda \uparrow \mathbb{R}^d]{}
        \mathbb{E}^{\mu}\Big[b(0, \cdot)\,\log \frac{\d(b(0, \cdot) \mu)}{\d \mu_0^!} \,\Big\vert\, \mathscr{S} \Big] 
        + \mathbb{E}^{\mu_0}\Big[\log \frac{\d \mu_0^!}{\d(b(0, \cdot) \mu)} \circ \gamma_0 \,\Big\vert\, \mathscr{S} \Big]
    \end{align} in \(L^1(\mu)\), where \(\mathscr{S}\) is the \(\sigma\)-algebra of translation-invariant and measurable sets and \(\gamma_0 \colon \eta \mapsto \eta \setminus \{0\}\).
\end{proposition}

We refer the reader to \Cref{section:well_definedness} for an explanation why the assumptions in \Cref{theorem:main_theorem_HANDA} are in fact optimal. 
Instrumental for proving \Cref{theorem:main_theorem_HANDA} is the following \textit{pointwise} identity with a remainder vanishing in the thermodynamic \(L^1(\mu)\)-limit.
\begin{lemma}\label{lemma:rest_term}
    There exist a jointly measurable strictly positive function $\varrho$ and a function \(\eta\mapsto R_\Lambda(\eta)\) with \(\mu[R_{\Lambda}]/\abs{\Lambda} \to 0\), as $\Lambda \uparrow \mathbb{R}^d$, such that
    \begin{align}
        (-\mathscr{L}) \log\frac{\d \mu_\Lambda}{\d \nu_\Lambda} (\eta)
        = \int_{\Lambda} \, b(0, \theta_x \eta) \, \log \frac{1}{\varrho_{\Lambda}(x, \theta_x \eta)}
        + \sum_{x \in \eta_\Lambda} \log \varrho_\Lambda(x, \theta_x(\eta -\delta_{x}))
    + R_{\Lambda}(\eta) 
    \end{align} 
    with 
     \begin{align*}
        \varrho_\Lambda(x, \cdot)
        = \frac{\d(\mu_0^!)_{\Lambda -x}}{\d (b(0, \cdot) \mu)_{\Lambda -x}}\quad\text{and}\quad \frac{1}{\varrho_\Lambda(x, \cdot)}
        = \frac{\d (b(0, \cdot) \mu)_{\Lambda -x}}{\d(\mu_0^!)_{\Lambda -x}}
    \end{align*} 
    (as in ``is a version of'') for \(\mathrm{d}x\)-a.e.\ \(x\).
    Additionally, we have the following simple estimate for the remainder \(R_\Lambda\),
    \begin{align*}
\mu\big[\abs{R_\Lambda}\big]
        \leq \big\vert \{x \in \Lambda \colon  d(x, \partial \Lambda) \leq \range \}\big\vert \cdot \log\big(\Vert b\Vert_\infty \Vert 1/b\Vert_\infty\big) \big(\Vert b\Vert_\infty + \mu\big[N_{[0,1 ]^d}\big] \big).
    \end{align*}
\end{lemma}

For concluding \Cref{theorem:main_theorem_HANDA} from \Cref{lemma:rest_term} we will additionally use the following ergodic theorem. 

\begin{lemma}[Ergodic theorem]\label{lemma:ergodic_theorem}
    Let $\mu\in \Pcal_\theta$. Then, there is a translation-invariant random measure \(\zeta^{\mu}(\cdot \,\vert\, \omega)\) such that
    \begin{align}
        \mathbb{E}^{\mu}\Big[\sum_{x \in \eta} g(x, \theta_x(\eta -\delta_{x})) \,\Big\vert\, \mathscr{S} \Big]
        = \int \d x \int \zeta^\mu(\d \eta \,\vert\, \cdot) \,g(x, \eta)
    \end{align} and
    \begin{align}
        \frac{1}{\vert\Lambda\vert} \sum_{x \in \omega_\Lambda} h(\theta_x(\omega -\delta_{x}))
        \xrightarrow[\Lambda \uparrow \mathbb{R}^d]{} \int \zeta^\mu(\d\eta \,\vert\, \omega) \, h(\eta)
    \end{align} \(\mu\)-a.s.\ and in \(L^1(\mu)\) for all \(h \in L^1(\mu_0^!)\).
\end{lemma}

Combining this with the pointwise representation of the entropy dissipation from \Cref{lemma:rest_term} yields \Cref{theorem:main_theorem_HANDA}, see \Cref{sec:proof-main-result-handa}. This finally puts us in the position to prove the desired de~Bruijn-type identity for the global dynamics. 

\begin{proof}[Proof of \Cref{eq:de_bruijn}]
    By the semigroup property and the fact that for \(\mu\) regular, the evolved measures \(\mu T_t\) stay regular  by \Cref{corollary:regular_measures_stay_regular_under_evolution}, we can apply \Cref{theorem_derivative_of_the_local_entropy} and \Cref{lemma:rest_term} to see that
    \begin{align*}
        \frac{1}{\abs{\Lambda}} \Big(-\frac{\d}{\d r}_{r=s} \RelEnt_\Lambda(\mu T_s \,\vert\, \nu) \Big)
        = \frac{1}{\abs{\Lambda}} \Big(-\frac{\d}{\d r}_{r=0} \RelEnt_\Lambda((\mu T_s)T_r \,\vert\, \nu) \Big) 
        = \widetilde{\sigma}_\Lambda(\mu T_s) + \frac{\, o(\abs{\Lambda})}{\abs{\Lambda}} 
    \end{align*} for \(s \in [0, t]\), where the remainder on the right-hand side, proved by \Cref{lemma:rest_term}, is uniform in \(s \in [0, t]\).
    Then, (the proof of) \Cref{theorem:main_theorem_HANDA} gives \(\sup_{\Lambda}  \frac{1}{\abs{\Lambda}} \widetilde{\sigma}_\Lambda(\mu T_s)  = \sigma(\mu T_s)\) and hence we have
    \begin{align*}
        \SpecEnt(\mu \,\vert\, \nu) - \SpecEnt(\mu T_t \,\vert\, \nu)
        &= \lim_{\Lambda \uparrow \mathbb{R}^d} \int_{0}^{t} \, \frac{1}{\abs{\Lambda}} \Big(-\frac{\d}{\d r}_{r=s} \RelEnt_\Lambda(\mu T_r \,\vert\, \nu) \Big) \, \d s
        = \lim_{\Lambda \uparrow \mathbb{R}^d} \int_{0}^{t} \, \widetilde{\sigma}_\Lambda(\mu T_s) \, \d s + \frac{t \, o(\abs{\Lambda})}{\abs{\Lambda}} \\
        &=  \int_{0}^{t} \, \lim_{\Lambda \uparrow \mathbb{R}^d}  \frac{1}{\abs{\Lambda}} \widetilde{\sigma}_\Lambda(\mu T_s) \, \d s 
        = \int_{0}^{t} \, \sigma(\mu T_s) \, \d s 
        = \int_{0}^{t} \, \xi^{s}(\mu) \, \d s,
    \end{align*} 
    as desired. 
\end{proof}

\subsection{Impossibility of convergence in finite time}
In order to show that the relative entropy is strictly decreasing for all times $t>0$ if we start in a non-Gibbsian initial distribution $\mu$, we establish the intuitive fact, that it is not possible for $\mu T_t$ to be in $\mathscr{G}_\theta$ at some finite time $t>0$ if $\mu\not\in\mathscr{G}_\theta$. 

\begin{proposition}\label{prop:not_reversible_in_finite_time}
    Let \(\mu\in \Pcal_\theta\setminus  \GG_\theta\).
    Then, also \(\mu T_t \in \Pcal_\theta\setminus  \GG_\theta\) for all \(t > 0\).
\end{proposition}

Since the operator $\mathscr{L}$ is unbounded, the semigroup $(T_t)_{t \geq 0}$ is in general not of the form $T_t = \sum_{k=0}^\infty \tfrac{t^k}{k!}\mathscr{L}^k$. The core idea of the proof is that for sufficiently regular measures $\mu$ and local functions one can at least find a small interval of times $[0,t^*)$, where $t^*$ depends on both $\mu$ and the observable $f$, inside of which one can represent $\mu[T_t f]$ in terms of an exponential series. 

\begin{proposition}[Series expansion]\label{prop:small_time_exponential_series_expansion}
    Let \(\mu\in \Pcal(\Omega)\) with
    \begin{align}\label{eqn:moment-condition}
        \mu\big[N_{\Delta}^k\big]^{1/k}
        \leq \frac{c_{\mu, 3}  \, c_{\mu, 2}^{\abs{\Delta}/k} \, k}{\log(1 + k/(c_{\mu, 1} \abs{\Delta}))}
    \end{align} for some \(c_{\mu, 1}, c_{\mu, 2}, c_{\mu, 3} \geq 1\), all \(\Delta \Subset \mathbb{R}^d\) and \(k \in \mathbb{N}\).
    Let \(\Lambda\) be a bounded box in \(\mathbb{R}^d\) with \(\abs{\Lambda} \geq \abs{B_R}\).
    Then, for all bounded and \(\Lambda\)-measurable \(f \colon \Omega \to \mathbb{R}\),
    \begin{align*}
        \mu\big[T_t f\big]
        = \sum_{k = 0}^{\infty} \frac{t^k}{k!} \mu\big[\mathscr{L}^k f\big]
    \end{align*} 
    for \(t < \big(12 \e \Vert b \Vert_{\infty} c_{\mu, 1} (c_{\mu, 3}+1) c_{\mu, 2}^{\abs{\Lambda}} \abs{\Lambda}\big)^{-1}\).
\end{proposition}
Let us note that this will also be handy for showing that $\mathscr{R}_\theta = \mathscr{G}_\theta$, see \Cref{section:reversible_measures_are_gibbs_measures}. Next, in order to use this small-time series expansion to show \Cref{prop:not_reversible_in_finite_time}, we need to show that (a) Gibbs measures with respect to the considered interaction $H$ satisfy the moment condition \eqref{eqn:moment-condition}, and (b) any measure that is sufficiently \textit{dynamically} close to a Gibbs measure must also satisfy the moment condition.

\begin{lemma}[Moments of Gibbs measures]\label{lemma:moments_of_Area_interaction_Gibbs_measures}
    If \(\nu \in \mathscr{G}\), then
    \begin{align*}
        \nu\big[N_{\Delta}^k\big]^{1/k}
        \leq \frac{c_{\nu, 2}^{\abs{\Delta}/k} \,k}{\log(1 + k/(c_{\nu, 1} \abs{\Delta}))}
    \end{align*} 
    for some \(c_{\nu, 1}, c_{\nu, 2} > 0\), all \(\Delta \Subset \mathbb{R}^d\) and \(k \in \mathbb{N}\).
\end{lemma}

Let us now make precise what we mean by being dynamically close to a Gibbs measure. 

\begin{lemma}\label{lemma:moments_cannot_be_much_worse_at_starting_point}
    Suppose there are \(c_{\widetilde{\nu}, 1}, c_{\widetilde{\nu}, 2}, c_{\widetilde{\nu}, 3} > 0\) such that for \(\widetilde{\nu} = \mu T_t\) it holds that
    \begin{align*}
        \widetilde{\nu}\big[N_{\Delta}^k\big]^{1/k}
        \leq \frac{c_{\widetilde{\nu}, 3} \,c_{\widetilde{\nu}, 2}^{\abs{\Delta}/k} \,k}{\log(1 + k/(c_{\widetilde{\nu}, 1} \abs{\Delta}))}
    \end{align*} for all \(\Delta \Subset \mathbb{R}^d\) and \(k \in \mathbb{N}\).
    Then, the same holds for \(\mu\) instead of \(\widetilde{\nu}\) just with \(c_{\widetilde{\nu}, 3}\) replaced by \({\rm e}^t c_{\widetilde{\nu}, 3}\).
\end{lemma}

The proof of these two technical ingredients can be found in \Cref{section:impossibility_of_finite_time_gibbs_without_gibbs_start}. 
By putting them together we can now prove \Cref{prop:not_reversible_in_finite_time}. 

\begin{proof}[Proof of \Cref{prop:not_reversible_in_finite_time}]
    Let \(\widetilde{\nu}\) be some (reversible) Gibbs measure and \(\mu \neq \widetilde{\nu}\).
    Then, there is some bounded and \(\Lambda\)-measurable, for a bounded box \(\Lambda\), function \(f\) such that \(\mu[f] \neq \widetilde{\nu}[f]\).
    Suppose there is a \(t_0 > 0\) such that \(\widetilde{\nu} = \mu T_{t_0} \).
    Then, by \Cref{prop:small_time_exponential_series_expansion} together with \Cref{lemma:moments_of_Area_interaction_Gibbs_measures} and \Cref{lemma:moments_cannot_be_much_worse_at_starting_point}, there is an \(\epsilon > 0\) such that
    \begin{align*}
        (\mu T_{t + \delta})[f]
        = \sum_{k = 0}^{\infty} \frac{\delta^k}{k!} (\mu T_{t})\big[\mathscr{L}^k f\big]
    \end{align*} for all \(t \in [0, t_0]\) and \(\delta \in [t, t + 2 \epsilon)\).
    In particular, for \(t = t_0 - \epsilon\) we get
    \begin{align*}
        (\mu T_{t_0 - \epsilon + \delta})[f] = \sum_{k = 0}^{\infty} \frac{\delta^k}{k!} (\mu T_{t_0 - \epsilon})\big[\mathscr{L}^k f\big]
    \end{align*} with
    \begin{align*}
        (\mu T_{t_0 - \epsilon + \delta})[f] = (\mu T_{\delta - \epsilon}T_{t_0})[f] = \widetilde{\nu}[f]
    \end{align*} for \(\epsilon \leq \delta < 2 \epsilon\). The identity theorem for power series yields that \((\mu T_{t_0 - \epsilon + \delta})[f] = \widetilde{\nu}[f]\) for \(0 < \delta < \epsilon\) too. Iterating this procedure yields the contradiction \(\widetilde{\nu}[f] = (\mu T_0 )[f] = \mu[f] \neq \widetilde{\nu}[f]\).
\end{proof}

\section{Domain of the generator}\label{section:domain_of_generator}

We show here that a wide class of functions is included in the domain of the generator $\mathcal{L}$ of the birth-and-death process.
Let us start by showing that the Markov process constructed via the graphical representation in terms of a driving Poisson point process as in \Cref{Definition:Markov_Process} actually has the formal generator $\mathscr{L}$ when acting on bounded and local measurable functions. The lemma below is from~\cite[Lemma~4.2]{JKSZ25}; as we will make use of its proof when showing that the domain contains further functions, we report it here.

\begin{lemma}\label{lem:domain}
    We have that
    \begin{align}\label{eq:generator_integral_for_bounded_local}
        T_t f = f + \int_{0}^{t}\d s \, T_s(\LL f),\qquad f\in\Lcal.
    \end{align}
\end{lemma}
\begin{proof}
    Let \(f\) be bounded and \(\Lambda\)-measurable, \(\Lambda \Subset \mathbb{R}^d\).
    We denote by $\birthsinL{t}$, resp.\ $\deathsinL{t}$, the set of times $s\leq t$ in which a birth, resp.\ death, occurs in $\L$. Furthermore, we split up $\deathsinL{t}$ in the deaths \(\deathsinLfromomega{t}\) coming from the initial condition \(\underline{\omega}\) and the deaths \(\deathsinLfrombirths{t}\) of particles which are not coming from the initial condition but are born in \([0, t)\). We decompose into the individual jump contributions
    \begin{align*}
        &f(X_t\Ssup{\omega}) - f(\omega)
        = \sum_{s\in \birthsinL{t}\cup \deathsinLfrombirths{t} \cup \deathsinLfromomega{t}} \big(f(X_s\Ssup{\omega}) - f(X_{s-}\Ssup{\omega})\big) \\
        &= \sum_{s\in\birthsinL{t}} \big(f(X_s\Ssup{\omega}) - f(X_{s-}\Ssup{\omega})\big) + \sum_{s\in\deathsinLfrombirths{t}} \big(f(X_s\Ssup{\omega}) - f(X_{s-}\Ssup{\omega})\big) + \sum_{s\in\deathsinLfromomega{t}} \big(f(X_s\Ssup{\omega}) - f(X_{s-}\Ssup{\omega})\big) \\
        &=: C_{B}(t) + C_{D^2}(t) + C_{D^\omega}(t).
    \end{align*}

    Denote by \(\mathscr{X}\) the solution map constructing the birth-and-death process \(X\Ssup{\omega}\) from the driving Poisson noise \(\Ncal\) as \(X_t\Ssup{\omega} = \mathscr{X}\Ssup{\omega}_t(\mathcal{N})\) and similarly \(Y_t\Ssup{\omega} = \mathscr{Y}\Ssup{\omega}_t(\mathcal{N})\) with \(\mathscr{X}\Ssup{\omega}_t(\mathcal{N}) =  \omega_t + \mathscr{Y}\Ssup{\omega}_t(\mathcal{N})\). If we want to avoid to think about the intricacies of this infinite-volume solution map, we can also consider the solution map \(\mathscr{X}^{(\omega, \widehat{\Lambda})}_{t}(\mathcal{N})\) of the analogous birth-and-death process just living in the bounded but very large volume \(\widehat{\Lambda}\).
    This solution map is easily constructed manually. Then, all calculations here can be done in the same way and the identity \Cref{eq:generator_integral_for_bounded_local} can be recovered by approximation with increasing volumes \(\widehat{\Lambda}\), using the finite-speed-of-propagation result \Cref{le_comparison_finite_infinite}.
    The solution map will be used in the following computations to apply the Mecke formula several times. We will also sometimes use the short-hand \(D_x^{-}f = f(\cdot - \delta_x) - f\).

\medskip
 \noindent
    \textit{1. Birth term \(C_B(t)\):}
    The birth term is given by
    \begin{align*}
        &\sum_{s\in\birthsinL{t}} \big(f(X_s\Ssup{\omega}) - f(X_{s-}\Ssup{\omega})\big) = \int_{\Lambda \times [0, \infty)^2\times [0, t]} \Ncal(\d x, \d u, \d r, \d s)\, \big(f(X_{s-}\Ssup{\omega} + \delta_x) - f(X_{s-}\Ssup{\omega})\big) \1_{[0, b(x, X_{s-}\Ssup{\omega})]}(u),
    \end{align*} 
    so that, by applying the Mecke formula,
    \begin{align*}
&\mathbb{E}\Big[\sum_{s\in\birthsinL{t}} \big(f(X_s\Ssup{\omega}) - f(X_{s-}\Ssup{\omega})\big)  \Big]\\
        &= \mathbb{E}\Big[\int_{\Lambda \times [0, \infty)^2\times [0, t] } \Ncal(\d x, \d u, \d r, \d s)\, \big(f(\mathscr{X}_{s-}^{(\omega)}(\Ncal) + \delta_x) - f(\mathscr{X}_{s-}^{(\omega)}(\Ncal))\big) \1_{[0, b(x, \mathscr{X}_{s-}^{(\omega)}(\Ncal))]}(u) \Big] \\
        &= \int_{0}^{t} \d s \int_{\Lambda} \d x\, \int_{0}^{\infty} \d r\ \e^{-r} \int_0^\infty \d u \, \mathbb{E}\Big[\big(\mathscr{X}_{s-}^{(\omega)}(\Ncal + \delta_{(x, s, r, u)}) + \delta_x) - f(\mathscr{X}_{s-}^{(\omega)}(\Ncal + \delta_{(x, s, r, u)}))\big)\\
        &\phantom{= \int_{0}^{t} \d s \int_{\Lambda} \d x\, \int_{0}^{\infty} \d r\ \e^{-r} \int_0^\infty \d u \, \mathbb{E}\big[\big(\mathscr{X}_{s-}^{(\omega)}(\Ncal + \delta_{(x, s, r, u)})\ }
        \1_{[0, b(x, \mathscr{X}_{s-}^{(\omega)}(\Ncal + \delta_{(x, s, r, u)}))]}(u)\Big] \\
        &= \int_{0}^{t} \d s\,  \mathbb{E}\Big[\int_{\Lambda}\d x\, b(x, X_{s-}\Ssup{\omega}) \big(f(X_{s-}\Ssup{\omega} + \delta_x) - f(X_{s-}\Ssup{\omega})\big) \Big]\\
        &= \int_{0}^{t} \d s\,  \mathbb{E}\Big[\int_{\Lambda}\d x\, b(x, X_s\Ssup{\omega}) \big(f(X_s\Ssup{\omega} + \delta_x) - f(X_s\Ssup{\omega})\big)\Big].
    \end{align*}

\medskip
 \noindent
    \textit{2. Second-order death term \(C_{D^2}(t)\):}
    We have, again by the Mecke formula,
    \begin{align*}
        C_{D^2}(t)
        &:= \mathbb{E}\Big[\sum_{s\in\deathsinLfrombirths{t}} \big(f(X_s\Ssup{\omega}) - f(X_{s-}\Ssup{\omega})\big)  \Big]\\
        &= \mathbb{E}\Big[ \int_{\Lambda \times [0, \infty)^2\times [0, t]} \Ncal(\d x, \d u, \d r, \d s)\, \big(f(X_{(s+r)-}\Ssup{\omega} - \delta_x) - f(X_{(s+r)-}\Ssup{\omega})\big) \\
        &\phantom{\mathbb{E}\Big[ \int_{\Lambda \times [0, \infty)^2\times [0, t]} \Ncal(\d x, \d u, \d r, \d s)\, \big(f(X_{(s+r)-}\Ssup{\omega} - \delta_x) -}
        \1_{[0, b(x, X_{s-}\Ssup{\omega})]}(u) \1_{s+r \leq t}\Big] \\
        &= \int_{0}^{t} \d s \int_{\Lambda} \d x\, \int_{0}^{\infty} \d r\ \e^{-r} \int_0^\infty \d u \, \mathbb{E}\Big[D_x^{-}f(\mathscr{X}_{(s+r)-}^{(\omega)}(\Ncal + \delta_{(x, s, r, u)}))\\
        &\phantom{= \int_{0}^{t} \d s \int_{\Lambda} \d x\, \int_{0}^{\infty} \d r\ \e^{-r} \int_0^\infty \d u \, \mathbb{E}\big[\big(\mathscr{X}_{s-}^{(\omega)}\quad }
        \1_{[0, b(x, \mathscr{X}_{(s+r)-}^{(\omega)}(\Ncal + \delta_{(x, s, r, u)})))]}(u)\Big] \1_{r+s \leq t}.
    \end{align*} 
    We compute, by considering the difference quotients, that
    \begin{align*}
        \rightDerivative{t} R_{D^2}(t)
        &= \int_{0}^{t} \d s \, \e^{-(t-s)} \int_{\Lambda} \d x\, \int_0^\infty \d u \, \mathbb{E}\big[D_x^{-}f(\mathscr{X}_{t-}\Ssup{\omega}(\mathcal{N} + \delta_{(x, s, t-s, u)}) ) \1_{[0, b(x, \mathscr{X}_{s-}\Ssup{\omega}(\mathcal{N} + \delta_{(x, s, t-s, u)})]}(u)  \big] \\
        &= \int_{0}^{t} \d s \, \e^{-(t-s)} \int_{\Lambda} \d x\, \int_0^\infty \d u \, \mathbb{E}\big[D_x^{-}f(\mathscr{X}_{t-}\Ssup{\omega}(\mathcal{N} + \delta_{(x, s, t-s, u)}) ) \1_{[0, b(x, \mathscr{X}_{s-}\Ssup{\omega}(\mathcal{N})]}(u)  \big] \\ 
        &= \int_{0}^{t} \d s \, \int_{\Lambda} \d x\, \int_{0}^{\infty} \d r\ \e^{-r} \int_0^\infty \d u \, \mathbb{E}\big[D_x^{-}f(\mathscr{X}_{t-}\Ssup{\omega}(\mathcal{N} + \delta_{(x, s, t-s, u)}) ) \1_{[0, b(x, \mathscr{X}_{s-}\Ssup{\omega}(\mathcal{N})]}(u)  \big] \1_{r+s > t} \\ 
        &= \int_{0}^{t} \d s \, \int_{\Lambda} \d x\, \int_{0}^{\infty} \d r\ \e^{-r} \int_0^\infty \d u \, \mathbb{E}\big[D_x^{-}f(\mathscr{X}_{t}\Ssup{\omega}(\mathcal{N} + \delta_{(x, s, r, u)}) ) \1_{[0, b(x, \mathscr{X}_{s-}\Ssup{\omega}(\mathcal{N})]}(u)  \big] \1_{r+s > t}.
    \end{align*}
    On the other hand
    \begin{align*}
        &\mathbb{E}\Big[\sum_{x \in Y_t^\omega} \big(f(X_{t}\Ssup{\omega} - \delta_x) - f(X_{t}\Ssup{\omega})\big)  \Big] \\
        &= \mathbb{E}\Big[ \int_{\Lambda \times [0, \infty)^2\times [0, t]} \Ncal(\d x, \d u, \d r, \d s)\, \big(f(X_{t}\Ssup{\omega} - \delta_x) - f(X_{t}\Ssup{\omega})\big) \1_{[0, b(x, X_{s-}\Ssup{\omega})]}(u) \1_{s+r > t} \Big] \\
        &= \int_{0}^{t} \d s \int_{\Lambda} \d x\, \int_{0}^{\infty} \d r\ \e^{-r} \int_0^\infty \d u \, \mathbb{E}\big[D_x^{-}f(\mathscr{X}_t\Ssup{\omega}(\mathcal{N} + \delta_{(x, s, r, u)})) \1_{[0, b(x, \mathscr{X}_{s-}\Ssup{\omega}(\mathcal{N}))]}(u) \big] \1_{r+s > t}
    \end{align*} too.
    We conclude
    \begin{align*}
        C_{D^2}(t)
        = \int_0^t \d s \, \mathbb{E}\Big[\sum_{x \in Y_s^\omega} \big(f(X_{s}\Ssup{\omega} - \delta_x) - f(X_{s}\Ssup{\omega})\big)  \Big].
    \end{align*}

\medskip
 \noindent
    \textit{3. First-order death term \(C_{D^\omega}(t)\):}
    For the first-order death term, writing \(\tau_x := \inf\{s \geq 0 \colon x \not\in X_s\Ssup{\omega}\} \) for \(x \in \omega_\Lambda\) and using that the lifespans of the initial condition are iid unit exponentials, we have that
    \begin{align*}
        &\rightDerivative{t} \mathbb{E}[(D_x^{-}f)(X\Ssup{\omega}_{\tau_x -}) \1_{\tau_x \in [0, t]}]
        = \mathbb{E}[(D_x^{-}f)(X\Ssup{\omega}_{t}) \1_{\tau_x > t}].
    \end{align*}  
    It follows that
    \begin{align*}
        C_{D^\omega}(t)
        &= \sum_{x \in \omega_\Lambda } \mathbb{E}[(D_x^{-}f)(X\Ssup{\omega}_{\tau_x -}) \1_{\tau_x \in [0, t]}] 
        = \sum_{x \in \omega_\Lambda } \int_0^t \d s \, \mathbb{E}[(D_x^{-}f)(X\Ssup{\omega}_{s}) \1_{\tau_x > s}] \\
        &= \int_0^t \d s \, \mathbb{E}\Big[\sum_{x \in (\omega_{s})_\Lambda} \big(f(X_{s}\Ssup{\omega} - \delta_x) - f(X_{s}\Ssup{\omega})\big)  \Big].
    \end{align*}

\medskip
 \noindent
    \textit{4. Summing up the individual contributions:} 
    Adding everything up, we have
    \begin{align*}
        &(T_t f)(\omega) - f(\omega)
        = C_B(t) + C_{D^2}(t) + C_{D^\omega}(t) \\
        &= \int_0^t \d s \, \bigg\{\mathbb{E}\Big[\int_{\Lambda}\d x\, b(x, X_s\Ssup{\omega}) \big(f(X_s\Ssup{\omega} + \delta_x) - f(X_s\Ssup{\omega})\big)\Big]
        + \mathbb{E}\Big[\sum_{x \in Y_s^\omega \cup (\omega_{s})_\Lambda} \big(f(X_{s}\Ssup{\omega} - \delta_x) - f(X_{s}\Ssup{\omega})\big)  \Big] \bigg\} \\
        &= \int_0^t \d s \, (T_s \mathscr{L}f)(\omega),
    \end{align*}
    as desired.
\end{proof}

In the following, we use the preceding lemma to show that, for regular measures \(\mu\), also \(\log (\d \mu_\Lambda/\d\nu_\Lambda)\) lies in the domain of the generator, with a \(\mu\)-integrable error term.

\begin{proof}[Proof of \Cref{lem:log_generator}]
    Write \(f = \log g_\Lambda\). From the proof of \Cref{lem:domain} or by approximation of \(f\) with bounded functions, we see that
    \begin{align*}
        T_t f = f + \int_{0}^{t}\d s \, T_s(\LL f).
    \end{align*} This approximation argument is possible by dominated convergence thanks to the assumed regularity of \(\mu\), hence bounds for \(f\) and \(\mathscr{L}f\), and the moments of Poisson point process \(\pi_\Lambda\) in \(\Lambda\).
    Then, pointwise
    \begin{align*}
        \frac{1}{t}(T_t f - f)\xrightarrow[t \downarrow 0]{} 0
    \end{align*}
  by continuity of \(s \mapsto T_s \mathscr{L} f\), which in turn is again clear from the properties of \(f\) and the bounded birth rates.
    Now, since, for any convex increasing function \(\phi\),
    \begin{align*}
     &\mu\Big[\phi\Big( \Big\vert \frac{T_t f - f}{t}  \Big\vert\Big)\Big]
        = \mu\Big[\phi\Big( \Big\vert \frac{1}{t} \int_{0}^{t} (T_s \mathscr{L} f)  \, \d s  \Big\vert\Big)\Big]\leq \mu\Big[\frac{1}{t} \int_{0}^{t} (T_s \phi(\vert \mathscr{L} f \vert ))  \, \d s  \Big] 
       = \frac{1}{t} \int_{0}^{t}\mu\big[ T_s \phi(\vert \mathscr{L} f\vert) \big] \d s,
    \end{align*} 
    we see that \((t^{-1}(T_t f - f))_{t \in (0, t_0]}\) is uniformly \(\mu\)-integrable for \(t_0 \in (0, \infty)\) as desired.
\end{proof}

\begin{remark}\label{rmk:pair}
    We see from the proof that regularity is not really necessary, much weaker bounds work too.
    For example, the same proof works when \(\mu\) is a Gibbs measure for a Hamiltonian \(H^\varphi(\omega) = \sum_{\{x,y\}\subset\omega}\varphi(\abs{x-y})\) coming from a positive, bounded potential \(\varphi\) with finite range $\range$.
    Then,
    \begin{align*}
        g_\Lambda(\omega_\Lambda) 
        = \int \mu(\d\zeta) \frac{\e^{-H^{\varphi}_{\Lambda, \zeta}(\omega_\Lambda)}}{Z^{\varphi}_{\Lambda, \zeta}}
    \end{align*} with \(\e^{-\abs{\Lambda}} \leq Z^{\varphi}_{\Lambda, \zeta} \leq 1\), and hence \(g_\Lambda(\omega_\Lambda) \leq \e^{\abs{\Lambda}}\), as well as
    \begin{align*}
        \log g_\Lambda(\omega_\Lambda) 
        &\geq -\int \mu(\d\zeta) H^{\varphi}_{\Lambda, \zeta}(\omega_\Lambda) - \int \mu(\d\zeta) \log Z^{\varphi}_{\Lambda, \zeta} \\
        &\geq -\int \mu(\d\zeta) \Big\{\Vert\varphi\Vert_\infty \abs{\omega_\Lambda}^2 + \Vert\varphi\Vert_\infty \abs{\omega_\Lambda} \abs{\zeta_{\Lambda\oplus B(0,\range)}} \Big\}\\
        &= - \Vert\varphi\Vert_\infty \abs{\omega_\Lambda}^2 - \Vert\varphi\Vert_\infty  \abs{\Lambda\oplus B(0,\range)} \mu[N_{[0,1]^d}] \abs{\omega_\Lambda},
    \end{align*} 
    i.e., there is a \(c > 0\) such that
    \begin{align*}
        g_\Lambda(\omega_\Lambda)
        \geq \e^{- c (\abs{\omega_\Lambda}^2 + \abs{\Lambda}\abs{\omega_\Lambda})}.
    \end{align*}
\end{remark}

Finally, let us observe the following simple necessary condition for \(\mathscr{L} \log(\d \mu_\Lambda/\d\nu_\Lambda)\) to be well-defined.

\begin{lemma}\label{section:well_definedness}
    Let \(\mathscr{L} \log (\d \mu_\Lambda/\d\nu_\Lambda)\) be \(\mu\)-a.e.\ well-defined.
    Then \(\mu_\Lambda\) and \(\pi_\Lambda\) are equivalent.
\end{lemma}
\begin{proof}
    Note that we can choose \(g_\Lambda = \d \mu_\Lambda/\d \nu_\Lambda\) and \(\widetilde{g}_\Lambda = \d \mu_\Lambda/\d \pi_\Lambda\) such that \(\widetilde{g}_\Lambda(\eta) = 0\) iff \(g_\Lambda(\eta) = 0\) because \(\nu_\Lambda\) and \(\pi_\Lambda\) are equivalent.
    Hence, for the \(\mu_\Lambda\)-a.e.\ well-definedness of
    \begin{align*}
        (\mathscr{L} \log g_\Lambda)(\eta),
    \end{align*} we need at least that
    \begin{align*}
        \{(x, \eta) \colon  \widetilde{g}_\Lambda(\eta +\delta_{x}) = 0 < \widetilde{g}_\Lambda(\eta)\}
    \end{align*} is a \((\d x \otimes \mu)\)-zero set and also that
    \begin{align*}
        \big\{\eta \colon \widetilde{g}_\Lambda(\eta) > 0, \text{ but } \widetilde{g}_\Lambda(\eta -\delta_{x}) = 0 \text{ for some } x \in \eta \big\}
    \end{align*} 
    is a \(\mu\)-zero set.
    The former is equivalent to the condition that
    \begin{align}\label{eq_null_set_1}
        \{(x, \eta) \colon \widetilde{g}_\Lambda(\eta +\delta_{x}) = 0 < \widetilde{g}_\Lambda(\eta)\}
    \end{align} 
    is a \((\d x \otimes \pi)\)-zero set, and the latter is indeed equivalent to the fact that
    \begin{align}\label{eq_null_set_2}
        \{(x, \eta) \colon \widetilde{g}_\Lambda(\eta +\delta_{x}) > 0 = \widetilde{g}_\Lambda(\eta)\}
    \end{align} 
    is a \((\d x \otimes \pi)\)-zero set.
    
    Let us now see that in fact \(\{\widetilde{g}_\Lambda = 0\}\) is a \(\pi_\Lambda\)-zero set then. By the fact \Cref{eq_null_set_2} it holds that \(\widetilde{g}_\Lambda(\emptyset) > 0\). 
    We therefore also have
    \begin{align*}
        \pi_\Lambda\big(\{\widetilde{g}_\Lambda = 0\} \cap \{N_\Lambda = 1\} \big)&
        \leq \pi\Big[\sum_{x \in \eta_\Lambda} \1_{\{\widetilde{g}_\Lambda = 0\}}(\eta_\Lambda) \1_{N_\Lambda = 0}(\eta_\Lambda-\delta_{x}) \Big] \\
        &= \int_{\Lambda} \d x \int \pi_\Lambda(\d \eta_\Lambda) \, \1_{\{\widetilde{g}_\Lambda = 0\}}(\eta_\Lambda +\delta_{x}) 1_{{N_\Lambda = 0}}(\eta) \\
        &= \int_{\Lambda} \d x \int \pi_\Lambda(\d \eta_\Lambda) \, \1_{\{\widetilde{g}_\Lambda = 0\}}(\eta_\Lambda +\delta_{x}) 1_{\{\emptyset\}}(\eta_\Lambda) \1_{\{\widetilde{g}_\Lambda > 0\}}(\eta_\Lambda) \\
        &\leq \int_{\Lambda} \d x \int \pi_\Lambda(\d \eta_\Lambda) \, \1_{\{(y, \zeta) \,\vert\, \widetilde{g}_\Lambda(\zeta +\delta_{y}) = 0 < \widetilde{g}_\Lambda(\zeta)\}}(x, \eta_\Lambda) = 0.
    \end{align*}
    By induction over \(k \in \mathbb{N}_0\) and using the same argument, we get that \(\pi_\Lambda\big(\{\widetilde{g}_\Lambda = 0\} \cap \{N_\Lambda = k\} \big) = 0\) for all \(k \in \mathbb{N}_0\), i.e., \(\pi_\Lambda\big(\{\widetilde{g}_\Lambda = 0\}\big) = 0\).
    Hence, there exists a version of \(\widetilde{g}_\Lambda = \d \mu_\Lambda/\d \pi_\Lambda\) that is strictly positive everywhere.
\end{proof}

\section{Density properties of the birth-and-death process}\label{section:density_properties_of_bd_process}

In this section, we will show that, for some fixed volume \(\Lambda \subseteq \mathbb{R}^d\) and for a fixed initial configuration \(\underline{\omega}\) with deterministic lifespans, the restriction \((Y_t\Ssup{\underline{\omega}})_\Lambda\) to \(\Lambda\) of the infinite-volume birth-and-death-process \(Y_t\Ssup{\underline{\omega}}\) minus the initial configuration \(\underline{\omega}\) possesses densities w.r.t.\ the Poisson point process \(\pi_\Lambda\) with good regularity in time \(t\).
We will construct these local densities from correlation functions of good regularity discussed in \Cref{section:nice_correlation_functions}. For the existence of these regular correlation functions, we need as an input some continuity properties of the distribution of \((Y_t\Ssup{\underline{\omega}})_\Lambda\) studied in \Cref{section:continuity_properties_birth_death_process}.
Our main result of this section is
\begin{proposition}\label{proposition:good_densities_of_the_birth_and_death_process}
    There exists a measurable map \((t, \underline{\omega}, \zeta_\Lambda) \mapsto \psi_{\Lambda, t}(\zeta_\Lambda \,\vert\, \underline{\omega})\) such that \(\psi_{\Lambda, t}(\cdot \,\vert\, \underline{\omega})\) is a version of the density of \((Y_t\Ssup{\underline{\omega}})_\Lambda\) w.r.t.\ the Poisson point process \(\pi_\Lambda\) in \(\Lambda\), and \(t \mapsto \psi_{\Lambda, t}(\zeta_\Lambda \,\vert\, \underline{\omega})\) is right-differentiable at every \(t \geq 0\).
    Furthermore,
    \begin{enumerate}
        \item The right-derivative \(\rightDerivative{t} \psi_{\Lambda, t}(\zeta_\Lambda \,\vert\, \underline{\omega})\) is continuous at every time \(t \not\in \{r \colon \exists (x, r) \in \underline{\omega} \text{ with } x \in \Lambda \oplus B(0, \range)\}\) which is not one of the (finitely many) deterministic lifespans of points of \(\underline{\omega}\) in \(\Lambda \oplus B(0, \range)\),
        
        \item There exist \(t_0 > 0\) and \(z_{\max} > 0\) such that, for all \(t \leq t_0\), we have
        \begin{align}
            \abs{\psi_{\Lambda, t}(\zeta_\Lambda \,\vert\, \cdot)}
            \lesssim z_{\max}^{\abs{\zeta_\Lambda}} (1-\e^{-t})^{\abs{\zeta_\Lambda}}\quad\text{and}\quad 
            \abs{\rightDerivative{t} \psi_{\Lambda, t}(\zeta_\Lambda \,\vert\, \cdot)}
            \lesssim \abs{\zeta_\Lambda} z_{\max}^{\abs{\zeta_\Lambda}-1} (1-\e^{-t})^{\abs{\zeta_\Lambda}-1},
        \end{align}
        if \(\abs{\zeta_\Lambda} \geq 1\), and
        \begin{align}
             \psi_{\Lambda, t}(\emptyset \,\vert\, \cdot) \sim 1\quad\text{and}\quad \abs{\rightDerivative{t} \psi_{\Lambda, t}(\emptyset \,\vert\, \cdot)}
            \lesssim 1.
        \end{align}
    \end{enumerate}
\end{proposition}

The proof of \Cref{proposition:good_densities_of_the_birth_and_death_process} is given right at the end of \Cref{section:nice_correlation_functions}.
But first, we need to discuss continuity properties of the birth-and-death process in the following section.

\subsection{Continuity properties}\label{section:continuity_properties_birth_death_process}

Consider (small) disjoint bounded boxes \(Q_\ell(x_1), \dots, Q_\ell(x_n)\) of side-length \(\ell > 0\) with centers \(x_1, \dots, x_n\). The distribution of points of the birth-and-death-process \(Y_t\Ssup{\underline{\omega}}\) into \(Q_\ell(x_1), \dots, Q_\ell(x_n)\), does not change much if we individually translate the centers \(x_1, \dots, x_n\) to centers \(\widetilde{x}_1, \dots, \widetilde{x}_n\) by some small magnitude.
\begin{proposition}\label{lemma:birth_death_process_expected_number_of_points_continuity_in_box_centers}
    Fix some initial condition \(\underline{\omega}\).
    Consider \(0 < \ell \leq L \) and \(x_1, \dots, x_n, \widetilde{x}_1, \dots, \widetilde{x}_n \in \mathbb{R}^d\) with \(\vert x_i - \widetilde{x}_i \vert \leq L\) such that the boxes \(Q_\ell(x_1), \dots, Q_\ell(x_n)\) are disjoint and \(Q_\ell(\widetilde{x}_1), \dots, Q_\ell(\widetilde{x}_n)\) are disjoint too. Then, 
    \begin{align*}
        \left\vert  \mathbb{E}\Big[\prod_{i = 1}^{n} N_{Q_\ell(x_i)}(Y_t\Ssup{\underline{\omega}})\Big] - \mathbb{E}\Big[\prod_{i = 1}^{n }N_{Q_\ell(\widetilde{x}_i)}(Y_t\Ssup{\underline{\omega}})\Big] \right\vert
        \lesssim \epsilon_L \abs{Q_\ell}^n
    \end{align*} with \(\epsilon_L \to 0\) as $L \to 0$.
\end{proposition}

As a consequence of the above continuity, we mainly have the following application in mind.
\begin{lemma}[All points are Lebesgue points]\label{lemma:correlation_function_lebesgue_points}
    For all \(n \in \mathbb{N}\) and distinct \(x_1, \dots, x_n \in \mathbb{R}^d\), the limit
    \begin{align*}
        \lim_{\ell \downarrow 0} \frac{\mathbb{E}\left[\prod_{i = 1}^{n }N_{Q_\ell(x_i)}(Y_t\Ssup{\underline{\omega}})\right]}{\vert Q_\ell \vert^n}
    \end{align*} exists.
\end{lemma}
\begin{proof}
    Assume that none of \(x_1, \dots, x_n\) is contained in \(\omega\); otherwise add up those point separately.
    We show that the sequence \(\left(\vert Q_\ell \vert^{-n}\mathbb{E}\left[\prod_{i = 1}^{n }N_{Q_\ell(x_i)}(Y_t\Ssup{\underline{\omega}})\right] \right)_{\ell > 0}\) has the Cauchy property.
    Fix some (small) \(L > 0\) and subdivide \(Q_\ell(x_i)\) into disjoint translates of \(Q_\ell\) with centers \(x_i^{(j)}\), such that
    \begin{align*}
        \frac{\mathbb{E}\left[\prod_{i = 1}^{n }N_{Q_\ell(x_i)}(Y_t\Ssup{\underline{\omega}})\right]}{\vert Q_\ell \vert^n}
        = \frac{1}{\vert Q_\ell \vert^n} \sum_{j_1} \dots \sum_{j_n} \mathbb{E}\left[\prod_{i = 1}^{n }N_{Q_\ell(x_{i}^{(j_i)})}(Y_t\Ssup{\underline{\omega}})\right].
    \end{align*}
    Since, by \Cref{lemma:birth_death_process_expected_number_of_points_continuity_in_box_centers},
    \begin{align*}
        \left\vert  \mathbb{E}\Bigg[\prod_{i = 1}^{n} N_{Q_\ell(x_{i}^{(j_i)})}(Y_t\Ssup{\underline{\omega}})\Bigg] - \mathbb{E}\Bigg[\prod_{i = 1}^{n }N_{Q_\ell(x_{i})}(Y_t\Ssup{\underline{\omega}})\Bigg] \right\vert
        \lesssim \epsilon_L \abs{Q_\ell}^n,
    \end{align*} 
    for some \(\epsilon_L \to 0\), as $L \to 0$, we have
    \begin{align*}
        \frac{\mathbb{E}\left[\prod_{i = 1}^{n }N_{Q_\ell(x_i)}(Y_t\Ssup{\underline{\omega}})\right]}{\vert Q_\ell \vert^n}
        &= \frac{1}{\vert Q_\ell \vert^n} \left(\frac{\vert Q_\ell \vert}{\vert Q_\ell \vert}\right)^n \left( \mathbb{E}\Bigg[\prod_{i = 1}^{n }N_{Q_\ell(x_{i})}(Y_t\Ssup{\underline{\omega}})\Bigg] + \mathcal{O}(\epsilon_L) \abs{Q_\ell}^n\right) \\
        &= \frac{\mathbb{E}\left[\prod_{i = 1}^{n }N_{Q_\ell(x_i)}(Y_t\Ssup{\underline{\omega}})\right]}{\vert Q_\ell \vert^n} + \mathcal{O}(\epsilon_L),
    \end{align*}
    hence the sequence is Cauchy.
\end{proof}

In the course of the proof of \Cref{lemma:birth_death_process_expected_number_of_points_continuity_in_box_centers} we will rely heavily on the representation of our birth-and-death process \(Y\Ssup{\underline{\omega}}\) via the driving Poisson noise \(\mathcal{N}\) as in \Cref{Definition:Markov_Process}, which gives us good global controls and a framework to easily couple certain variations of the birth-and-death process. In fact, it seems hard to find a proof for the statement with \(n > 1\), which would rely only on standard tools like inductively using the generator plus some version of Grönwall's inequality. Note that here the statement immediately becomes wrong if we would instead consider the expectation of the distance instead of the distance of expectations. 
We will also make use of the finite speed of propagation result below, in which we compare the original dynamics with that of a Markov process in $\Lambda$; its proof follows in the same way as~\cite[Lemma~4.4]{JKSZ25}.

We define the following dynamics in $\Lambda$ with empty boundary condition.
\begin{definition}[Local dynamics]\label{Def:modified_dynamics_with_stochastic_boundary_conditions}
Let $\L\Subset\R^d$ and denote by \((\Tl_t)_{t \geq 0}\) the semigroup associated to the following (formal) generator
\begin{align*}
                &(\LL_\Lambda f)(\eta_\Lambda)
                = \int_{\Lambda} \d x\, b(x, \eta_\Lambda) \big(f(\eta_\Lambda+\delta_x) - f(\eta_\Lambda)\big) + \sum_{x \in \eta_\Lambda} \big(f(\eta_\Lambda -\delta_x) - f(\eta_\Lambda)\big).
\end{align*} 
We denote by \(X\Ssup{\omega,\L}\) the corresponding Markov process with initial condition $\omega$.
\end{definition}
Analogously to the infinite-volume case, we denote by \(Y^{(\omega,\Lambda)}\) and \(Y^{(\underline\omega, \Lambda)}\) the processes without the points of the initial condition $\omega$.
For $\Lambda\subset\widehat\Lambda\Subset\R^d$, note the following distinction between the two restrictions \(Y_{\cdot, \Lambda}^{(\omega)}\) and \(Y_{\cdot, \Lambda}^{(\omega, \widehat{\Lambda})}\) to the observation window $\widehat\Lambda$: while the first one comes from the infinite-volume dynamics, which is influenced by the configuration on the entire space $\R^d$, the second one represents the finite-volume dynamics that evolves purely inside $\widehat\Lambda$, with empty boundary conditions outside. The lemma below quantifies the error coming from this approximation, showing that the probability that the influence from outside $\widehat\Lambda$ reaches $\Lambda$ within time $T$ decays super-exponentially with the distance between $\Lambda$ and $\widehat\Lambda$.
\begin{lemma}\label{le_comparison_finite_infinite}
    For any bounded \(\Lambda\subset\R^d\), let $\widehat\Lambda\Subset\R^d$, with \(\widehat\Lambda\supset\Lambda\) and \(m := \mathrm{dist}(\Lambda, \widehat{\Lambda}^{c})\). We have
    \begin{align*}
\mathbb{P}\Big(\exists t \in [0, T] \colon Y_{t, \Lambda}^{(\omega, \widehat{\Lambda})} \neq Y_{t, \Lambda}^{(\omega)} \Big)
        \lesssim_{\Lambda, d, \range, T, \Vert b \Vert_\infty} 
        \e^{-c m f(m)}
    \end{align*}
    for some \(f\) with \(f(m) \to \infty\), as $m \to \infty$, and $c=c(\Lambda, d, \range, T, \Vert b \Vert_\infty)$.
\end{lemma}

\begin{proof}[Proof of \Cref{lemma:birth_death_process_expected_number_of_points_continuity_in_box_centers}]
For the sake of readability, we only show this in the case \(n = 2\), which already contains all the difficulty of the general case \(n \geq 2\). We will show that, for \(\epsilon_L \to 0\), as $L \to 0$,
    \begin{align*}
        \big\vert  \mathbb{E}\bigl[N_{Q_\ell(x_{1})}(Y_t\Ssup{\underline{\omega}})N_{Q_\ell(x_{2})}(Y_t\Ssup{\underline{\omega}}) \bigr] - \mathbb{E}\bigl[N_{Q_\ell(\widetilde{x}_{1})}(Y_t\Ssup{\underline{\omega}})N_{Q_\ell(\widetilde{x}_{2})}(Y_t\Ssup{\underline{\omega}}) \bigr] \big\vert
        \lesssim \epsilon_L \vert Q_\ell\vert^2.
    \end{align*}
    First, by the triangle inequality and symmetry, we only have to show
    \begin{align*}
        \big\vert  \mathbb{E}\bigl[N_{Q_\ell(x_{1})}(Y_t\Ssup{\underline{\omega}})N_{Q_\ell(x_{2})}(Y_t\Ssup{\underline{\omega}}) \bigr] - \mathbb{E}\bigl[N_{Q_\ell(\widetilde{x}_{1})}(Y_t\Ssup{\underline{\omega}})N_{Q_\ell(x_{2})}(Y_t\Ssup{\underline{\omega}}) \bigr] \big\vert
        \lesssim \epsilon_L \vert Q_\ell\vert^2,
    \end{align*} 
    i.e., we can pretend that \(x_2 = \widetilde{x}_2\).
    Using \Cref{le_comparison_finite_infinite}, we can also just work with a birth-and-death process in a large (where ``large'' will depend on \(L\)) but finite box \(\widehat{\Lambda}\) containing \(\Lambda\), because then
    \begin{align*}
        &\big\vert  \mathbb{E}\bigl[N_{Q_\ell(x_{1})}(Y_t\Ssup{\underline{\omega}})N_{Q_\ell(x_{2})}(Y_t\Ssup{\underline{\omega}}) \bigr] - \mathbb{E}\bigl[N_{Q_\ell(\widetilde{x}_{1})}(Y_t\Ssup{\underline{\omega}})N_{Q_\ell(x_{2})}(Y_t\Ssup{\underline{\omega}}) \bigr] \big\vert \\
        &= \big\vert \mathbb{P}\left(N_{Q_\ell(x_{1})}(Y_t\Ssup{\underline{\omega}}) = 1, N_{Q_\ell(x_{2})}(Y_t\Ssup{\underline{\omega}}) = 1 \right) 
        - \mathbb{P}\left(N_{Q_\ell(\widetilde{x}_{1})}(Y_t\Ssup{\underline{\omega}}) = 1, N_{Q_\ell(\widetilde{x}_{2})}(Y_t\Ssup{\underline{\omega}}) = 1 \right)  \big\vert + \mathcal{O}(\vert Q_\ell\vert^3) \\
        &= \big\vert \mathbb{P}\big(N_{Q_\ell(x_{1})}(Y_t^{\widehat{\Lambda},\underline{\omega}}) = 1, N_{Q_\ell(x_{2})}(Y_t^{\widehat{\Lambda},\underline{\omega}}) = 1 \big)  
        - \mathbb{P}\big(N_{Q_\ell(\widetilde{x}_{1})}(Y_t^{\widehat{\Lambda},\underline{\omega}}) = 1, N_{Q_\ell(\widetilde{x}_{2})}(Y_t^{\widehat{\Lambda}, \underline{\omega}}) = 1 \big)  \big\vert + \mathcal{O}(\vert Q_\ell\vert^3) \\
        &= \big\vert \mathbb{E}\bigl[N_{Q_\ell(x_{1})}(Y_t^{\widehat{\Lambda},\underline{\omega}})N_{Q_\ell(x_{2})}(Y_t^{\widehat{\Lambda}, \underline{\omega}}) \bigr] - \mathbb{E}\bigl[N_{Q_\ell(\widetilde{x}_{1})}(Y_t^{\widehat{\Lambda},\underline{\omega}})N_{Q_\ell(x_{2})}(Y_t^{\widehat{\Lambda},\underline{\omega}}) \bigr]\big\vert + \mathcal{O}(\vert Q_\ell\vert^3),
    \end{align*}
    at least if we choose \(\widehat{\Lambda}\) large enough so that the error given by \Cref{le_comparison_finite_infinite} is \(\mathcal{O}(\vert Q_\ell\vert^3)\) too, and as long as the other errors do not depend on the size of \(\widehat{\Lambda}\).

    It is straightforward to manually construct a ``solution map'' \(\mathscr{Y}\) that takes an initial condition \(\underline{\eta}\) (points with lifespans) and a driving Poisson noise \(\mathcal{N}\), and returns the desired birth-and-death process \(t \mapsto \mathscr{Y}^{\underline{\eta}}_{t}(\mathcal{N})\) in \(\widehat{\Lambda}\).
    Using this solution map \(\mathscr{Y}\) now to represent \(Y^{\widehat{\Lambda}, \underline{\omega}}\) and writing \(Y^{\widehat{\Lambda}, \underline{\omega}}_t = \mathscr{Y}_{t}(\mathcal{N})\),
    \begin{align*}
        &N_{Q_\ell(x_{i})}(Y_t^{\widehat{\Lambda}})
        = \sum_{(y_i, s_i, r_i, u_i) \in \mathcal{N}} \1_{Q_\ell(x_i)}(y_i) \1_{[0, t]}(s_i)  \1_{(t-s_i, \infty)}(r_i) \1_{[u_i, \infty)}(b(y_i, \mathscr{Y}_{s_i-}(\mathcal{N})))
    \end{align*}
    and the multivariate Mecke formula gives
    \begin{align*}
        \mathbb{E}\bigl[N_{Q_\ell(x_{1})}(Y_t^{\widehat{\Lambda}, \underline{\omega}})N_{Q_\ell(x_{2})}(Y_t^{\widehat{\Lambda}, \underline{\omega}}) \bigr]
        &= \int_{Q_\ell(x_1)} \d y_1 \, \int_{Q_\ell(x_2)} \d y_2 \, \int_{0}^{t} \d s_1 \, \int_{0}^{t} \d s_2 \, \int_{t-s_1}^{\infty} \d r_1 \, \int_{t-s_2}^{\infty} \d r_2 \, \\
        &\quad\quad \int_{0}^{\Vert b \Vert_\infty} \d u_1 \, \int_{0}^{\Vert b \Vert_\infty} \d u_2 \,\, F((y_1, s_1, r_1, u_1), (y_2, s_2, r_2, u_2))
    \end{align*} with
    \begin{align*}
        &F((y_1, s_1, r_1, u_1), (y_2, s_2, r_2, u_2)):= \mathbb{E}\big[\1_{[u_1, \infty)}\bigl(b(y_1, \mathscr{Y}_{s_1-}(\mathcal{N} + \delta_{(y_1, s_1, r_1, u_1)} + \delta_{(y_2, s_2, r_2, u_2)}) \bigr)\\
        &\hspace{7cm}
        \1_{[u_2, \infty)}\bigl(b(y_2, \mathscr{Y}_{s_2-}(\mathcal{N} + \delta_{(y_1, s_1, r_1, u_1)} + \delta_{(y_2, s_2, r_2, u_2)}) \bigr) \big] \\
        &= \mathbb{E}\big[\1_{[u_1, \infty)}\bigl(b(y_1, \mathscr{Y}_{s_1-}(\mathcal{N}  + \delta_{(y_2, s_2, r_2, u_2)}) \bigr) \1_{[u_2, \infty)}\bigl(b(y_2, \mathscr{Y}_{s_2-}(\mathcal{N} + \delta_{(y_1, s_1, r_1, u_1)}) \bigr) \big],
    \end{align*} 
    where the last equality is evident from the fact that the solution map cannot see points from the future.

    It follows that
    \begin{align*}
        &\Big\vert  \mathbb{E}\bigl[N_{Q_\ell(x_{1})}(Y_t^{\widehat{\Lambda}, \underline{\omega}})N_{Q_\ell(x_{2})}(Y_t^{\widehat{\Lambda}, \underline{\omega}}) \bigr] - \mathbb{E}\bigl[N_{Q_\ell(\widetilde{x}_{1})}(Y_t^{\widehat{\Lambda}, \underline{\omega}})N_{Q_\ell(x_{2})}(Y_t^{\underline{\widehat{\Lambda}, \omega}}) \bigr] \Big\vert \\
        &\leq 
         \int_{Q_\ell(x_1)} \d y_1 \, \int_{Q_\ell(x_2)} \d y_2 \, \int_{0}^{t} \d s_1 \, \int_{0}^{t} \d s_2 \, \int_{t-s_1}^{\infty} \d r_1 \, \int_{t-s_2}^{\infty} \d r_2 \int_{0}^{\Vert b \Vert_\infty} \d u_1 \, \int_{0}^{\Vert b \Vert_\infty} \d u_2  \\
        &\quad\quad \, \Big\{ \mathbb{E}\big[\big\vert \1_{[u_1, \infty)}\bigl(b(y_1, \mathscr{Y}_{s_1-}(\mathcal{N}  + \delta_{(y_2, s_2, r_2, u_2)}) \bigr) - \1_{[u_1, \infty)}\bigl(b(y_1 + z, \mathscr{Y}_{s_1-}(\mathcal{N}  + \delta_{(y_2, s_2, r_2, u_2)}) \bigr) \big\vert \big] \\
        &\quad\quad + \mathbb{E}\big[\big\vert \1_{[u_2, \infty)}\bigl(b(y_2, \mathscr{Y}_{s_2-}(\mathcal{N} + \delta_{(y_1, s_1, r_1, u_1)}) \bigr) - \1_{[u_2, \infty)}\bigl(b(y_2, \mathscr{Y}_{s_2-}(\mathcal{N} + \delta_{(y_1 + z, s_1, r_1, u_1)}) \bigr) \big\vert \big] \Big\},
    \end{align*} 
    with \(z := \widetilde{x}_1 - x_1\). Now,
    \begin{align*}
         &\int_{0}^{\Vert b \Vert_\infty} \d u_1 \,  \mathbb{E}\big[\big\vert \1_{[u_1, \infty)}\bigl(b(y_1, \mathscr{Y}_{s_1-}(\mathcal{N}  + \delta_{(y_2, s_2, r_2, u_2)}) \bigr)
         - \1_{[u_1, \infty)}\bigl(b(y_1 + z, \mathscr{Y}_{s_1-}(\mathcal{N}  + \delta_{(y_2, s_2, r_2, u_2)}) \bigr) \big\vert \big] \\
         &= \mathbb{E}\Big[ \int_{0}^{\Vert b \Vert_\infty} \d u_1 \,  \Bigl\vert \1_{[u_1, \infty)}\bigl(b(y_1, \mathscr{Y}_{s_1-}(\mathcal{N}  + \delta_{(y_2, s_2, r_2, u_2)}) \bigr) 
         - \1_{[u_1, \infty)}\bigl(b(y_1 + z, \mathscr{Y}_{s_1-}(\mathcal{N}  + \delta_{(y_2, s_2, r_2, u_2)}) \bigr) \Bigr\vert \Big] \\
         &= \mathbb{E}\Big[\Bigl\vert b(y_1, \mathscr{Y}_{s_1-}(\mathcal{N}  + \delta_{(y_2, s_2, r_2, u_2)}) \bigr) - b(y_1 + z, \mathscr{Y}_{s_1-}(\mathcal{N}  + \delta_{(y_2, s_2, r_2, u_2)})\Bigr\vert \Big] \\
         &\leq \Vert b(\cdot, \cdot) - b(\cdot + z, \cdot)  \Vert_\infty.
    \end{align*}
    For the other summand, we first see that
    \begin{align*}
        &\mathbb{E}\Big[\Bigl\vert \1_{[u_2, \infty)}\bigl(b(y_2, \mathscr{Y}_{s_2-}(\mathcal{N} + \delta_{(y_1, s_1, r_1, u_1)}) \bigr) - \1_{[u_2, \infty)}\bigl(b(y_2, \mathscr{Y}_{s_2-}(\mathcal{N} + \delta_{(y_1 + z, s_1, r_1, u_1)}) \bigr) \Bigr\vert \Big] \\
        &= \mathbb{E}\Big[\Bigl\vert \1_{[u_2, \infty)}\bigl(b(y_2, \mathscr{Y}_{s_2-}(\mathcal{N}) \bigr) - \1_{[u_2, \infty)}\bigl(b(y_2, \mathscr{Y}_{s_2-}(\mathcal{N}) \bigr) \Bigr\vert \Big]= 0,
    \end{align*} 
    if \(s_2 \leq s_1\).
    For \(s_2 > s_1\), considering separately the evolution from time \(0\) to \(s_1\) and from \(s_1\) to \(s_2\), 
    \begin{align*}
        &\mathbb{E}\Big[\Bigl\vert \1_{[u_2, \infty)}\bigl(b(y_2, \mathscr{Y}_{s_2-}(\mathcal{N} + \delta_{(y_1, s_1, r_1, u_1)}) \bigr) - \1_{[u_2, \infty)}\bigl(b(y_2, \mathscr{Y}_{s_2-}(\mathcal{N} + \delta_{(y_1 + z, s_1, r_1, u_1)}) \bigr) \Bigr\vert \Big] \\
        &= \mathbb{E}^{\mathcal{N}_{[0, s_1]}}\Big[G(\underline{\mathscr{Y}}_{s_1}(\mathcal{N}_{[0, s_1]}) + \delta_{(y_1, r_1)}, \underline{\mathscr{Y}}_{s_1}(\mathcal{N}_{[0, s_1]}) + \delta_{(y_1 + z, r_1)}) \Big]
    \end{align*} with
    \begin{align*}
        G(\underline{\eta}, \underline{\widetilde{\eta}})
        &:=
        \mathbb{E}^{\mathcal{N}_{[s_1, \infty)}}\Big[\Bigl\vert \1_{[u_2, \infty)}\bigl(b(y_2, \mathscr{Y}_{(s_2-s_1)-}^{\underline{\eta}}(\mathcal{N}_{[s_1, \infty)}) \bigr) - \1_{[u_2, \infty)}\bigl(b(y_2, \mathscr{Y}_{(s_2-s_1)-}^{\underline{\widetilde{\eta}}}(\mathcal{N}_{[s_1, \infty)}) \bigr) \Bigr\vert \Big] \\
        &\leq \mathbb{P}^{\mathcal{N}_{[s_1, \infty)}}\left( \Big( \mathscr{Y}_{(s_2-s_1)-}^{\underline{\eta}}(\mathcal{N}_{[s_1, \infty)}) \Big)_{\widetilde{\Lambda}} \neq  \Big( \mathscr{Y}_{(s_2-s_1)-}^{\underline{\widetilde{\eta}}}(\mathcal{N}_{[s_1, \infty)}) \Big)_{\widetilde{\Lambda}} \right),
    \end{align*} 
    and using the notation \(\underline{\mathscr{Y}}\) for the version of \(\mathscr{Y}\) with associated remaining lifespans, as well as \(\widetilde{\Lambda} = \L\oplus B(0,\range)= \{x\in\R^d\colon \dist(x,\L)\leq \range\}\).
    By \Cref{le_continuity_wrt_translation_of_initial_conditions}, 
    \begin{align*}
         \mathbb{P}^{\mathcal{N}_{[s_1, \infty)}}\Big( \big( \mathscr{Y}_{(s_2-s_1)-}^{\underline{\eta}}&(\mathcal{N}_{[s_1, \infty)}) \big)_{\widetilde{\Lambda}} \neq  \big( \mathscr{Y}_{(s_2-s_1)-}^{\underline{\widetilde{\eta}}}(\mathcal{N}_{[s_1, \infty)}) \big)_{\widetilde{\Lambda}} \Big)\\
         &
         \lesssim f\Big(\sup_{\vert z \vert \leq L}\sup_{x \in \mathbb{R}^d} \Vert b(\cdot, \cdot + \delta_{x}) - b(\cdot, \cdot + \delta_{x + z}) \Vert_\infty\Big),
    \end{align*} 
    for, e.g., \(f(x) = \sqrt{x}\) (the exact decay in \(L \xrightarrow[]{} 0\) does not matter here) and the statement  follows.
\end{proof}

We finish this section with the result needed above, that in the setting of two coupled birth-and-death-processes \(Y\Ssup{\omega},  Y^{(\widetilde{\omega})}\) sharing the same driving Poisson noise \(\mathcal{N}\), there is a strong continuity in the distribution of these processes regarding translations of the initial condition in the sense that \(Y\Ssup{\omega},  Y^{(\widetilde{\omega})}\) agree on a whole time interval \([0, S]\) with very high probability if the translation of \(\omega\) to \(\widetilde{\omega}\) is not too large in magnitude.
\begin{lemma}[Continuity w.r.t.\ translation of initial condition]\label{le_continuity_wrt_translation_of_initial_conditions}
    For all initial configurations \(\omega\) and \(\widetilde{\omega}\) that agree in all points except one, which is shifted by a length \(\leq L\) in \(\widetilde{\omega}\), we have for all $S\ge 0$,
    \begin{align*}
        \mathbb{P}\Big(\forall t \in [0, S]\colon Y_{t, \Lambda}^{(\omega, \widehat{\Lambda})} = Y_{t, \Lambda}^{(\widetilde{\omega}, \widehat{\Lambda})} \Big)
        \geq (1 - \delta_L)^n - \mathbb{P}(\mathrm{Poisson}(c_{\range} S \vert \Lambda\vert \Vert b \Vert_\infty) \geq n),
    \end{align*} 
    uniformly in \(\widehat{\Lambda}\) and for all \(n \in \mathbb{N}\), with
    \begin{align*}
        \delta_L
        \sim_{\range, \Lambda, \Vert b \Vert_\infty} \Big(\sup_{\vert z \vert \leq L}\sup_{x \in \mathbb{R}^d} \Vert b(\cdot, \cdot + \delta_{x}) - b(\cdot, \cdot + \delta_{x + z}) \Vert_\infty\Big).
    \end{align*}
    In particular,
    \begin{align*}
        \mathbb{P}\Bigl(\exists t \in [0, S] \colon Y_{t, \Lambda}^{(\omega, \widehat{\Lambda})} \neq Y_{t, \Lambda}^{(\widetilde{\omega}, \widehat{\Lambda})} \Bigr)
        \lesssim_S f(\delta_L)
    \end{align*} for every \(f \colon \mathbb{R}_+ \to \mathbb{R}_+\) with \(f(x)/x \to 0\), as $x \to \infty$.
\end{lemma}
\begin{proof}
    Consider \(\Delta := \bigl\{x \in \mathbb{R}^d \colon  \mathrm{dist}(x, \Lambda) \leq 10 \range \bigr\}\),
    and let
    \begin{align*}
        \tau_1
        := \inf\bigl\{t \geq 0 \colon  \mathcal{N}\bigl(\Delta \times [0, t] \times \mathbb{R}_+ \times [0, \Vert b \Vert_\infty]\bigr) = 1 \bigr\}
    \end{align*} 
    be the first time \(\mathcal{N}\) could potentially put a particle into \(\Delta\).
    Recursively define
    \begin{align*}
        \tau_j
        := \inf\bigl\{t > \tau_{j-1} \colon \mathcal{N}\bigl(\Delta \times (\tau_{j-1}, t] \times \mathbb{R}_+ \times [0, \Vert b \Vert_\infty]\bigr) = 1 \bigr\}
    \end{align*} for \(j \geq 2\) and put \(\tau_0 := 0\).
    Then, the \((\tau_j - \tau_{j-1})_{j \geq 1}\) are iid \(\mathrm{Exp}(\vert\Delta\vert \Vert b\Vert_\infty)\)-distributed random variables. 
    Also,
    \begin{align*}
        \mathbb{P}\Big(\forall t \in [0, S]\colon Y_{t, \Lambda}^{(\omega, \widehat{\Lambda})} = Y_{t, \Lambda}^{(\widetilde{\omega}, \widehat{\Lambda})} \Big)
        &\geq \mathbb{P}\Big(\forall t \in [0, S]\colon Y_{t, \Delta}^{(\omega, \widehat{\Lambda})} = Y_{t, \Delta}^{(\widetilde{\omega}, \widehat{\Lambda})} , \tau_n \leq S \Big) \\
        &= \mathbb{P}\Big(\forall j \in \{1, \dots, n\}\colon Y_{\tau_j, \Delta}^{(\omega, \widehat{\Lambda})} = Y_{\tau_j, \Delta}^{(\widetilde{\omega}, \widehat{\Lambda})},\, \tau_n \leq S \biggr) \\
        &\geq \mathbb{P}\Big(\forall j \in \{1, \dots, n\}\colon Y_{\tau_j, \Delta}^{(\omega, \widehat{\Lambda})} = Y_{\tau_j, \Delta}^{(\widetilde{\omega}, \widehat{\Lambda})}\Big) - \mathbb{P}(\tau_n > S).
    \end{align*}
    Let us now show that
    \begin{align*}
        \mathbb{P}\Big(Y_{\tau_1, \Delta}^{(\omega, \widehat{\Lambda})} = Y_{\tau_1, \Delta}^{(\widetilde{\omega}, \widehat{\Lambda})} \Big)
        \geq 1 - \delta_L.
    \end{align*}
    Until \(\tau_1\), nothing changes inside \(\Delta\) except for maybe the dying initial condition and everything outside \(\Delta\) changes in the same way. By Fubini--Tonelli, we can consider
    \begin{align*}
        L(x, t) 
        := b(x, X_{t}\Ssup{\omega}) \land b(x, X_{t}^{(\widetilde{\omega})})\quad \text{ and }\quad
        U(x, t)
        := b(x, X_{t}\Ssup{\omega}) \lor b(x, X_{t}^{(\widetilde{\omega})}) 
    \end{align*} 
    to be deterministic functions (until \(\tau_1\)).
    Now, \(Y_{\tau_1, \Delta}^{(\omega, \widehat{\Lambda})} \neq Y_{\tau_1, \Delta}^{(\widetilde{\omega}, \widehat{\Lambda})}\) can only happen if
    \begin{align*}
        u \in
        \bigl[L(x, \tau_1), U(x, \tau_1) \bigr]
    \end{align*} for the chance parameter \(u\) at \(\tau_1\).
    Correspondingly, we define
    \begin{align*}
        E_t
        := \big\{(x, s, r, u) \in \Delta \times [0, t] \times \mathbb{R}_+ \times [0, \Vert b \Vert_\infty] \colon u \in [L(x, s), U(x, s)] \big\},
    \end{align*} and
    \begin{align*}
        E_t^c
        := (\Delta \times [0, t] \times \mathbb{R}_+ \times [0, \Vert b \Vert_\infty]) \setminus E_t,
    \end{align*} 
    as well as their associated stopping times
    \begin{align*}
        \sigma_1
        := \inf\bigl\{t \geq 0 \colon \mathcal{N}\bigl(E_t\bigr) = 1 \bigr\} \quad\text{and}\quad \sigma_1^c
        := \inf\bigl\{t \geq 0 \colon \mathcal{N}\bigl(E_t^c\bigr) = 1 \bigr\}.
    \end{align*}
    We find that 
    \begin{align*}
        \mathbb{P}\bigl(\sigma_1 > t,  \sigma_1^c > s\bigr)
        &= \mathbb{P}\bigl(\mathcal{N}(E_t) = 0, \mathcal{N}(E_s^c) = 0\bigr)
        = \mathbb{P}\bigl(\mathcal{N}(E_t) = 0)\mathbb{P}\bigl(\mathcal{N}(E_s^c) = 0\bigr)= \e^{-\vert E_t\vert} \e^{-\vert E_s^c \vert},
    \end{align*} i.e., \(\sigma_1, \sigma_1^c\) have densities
    \begin{align*}
        &f_{\sigma_1}(t)
        = G(t) \exp(-\int_{0}^{t} G(s) \d s)\quad \text{and}\quad f_{\sigma_1^c}(t)
        = (\vert\Delta\vert \Vert b \Vert_\infty - G(t)) \exp(-\int_{0}^{t} (\vert\Delta\vert \Vert b \Vert_\infty - G(s)) \d s),
    \end{align*} with \(G(t) = \int_{\Delta} \d x \int_{0}^{\Vert b \Vert_\infty} \1_{[L(x, t), U(x, t)]}(u)\).
    Therefore,
    \begin{align*}
        \mathbb{P}\bigl(\sigma_1^c > \sigma_1\bigr)
        &= \int_{0}^{\infty} \d t \, f_{\sigma_1}(t) \, \int_{t}^{\infty} \d s \, f_{\sigma_1}(s)\leq \int_{0}^{\infty} \d t \, \vert\Delta\vert  \epsilon \, \int_{t}^{\infty} \d s \, \vert\Delta\vert \Vert b \Vert_\infty \e^{-(\Vert b \Vert_\infty - \epsilon)\vert\Delta\vert s}=\frac{\epsilon \vert \Delta \vert \Vert b \Vert_\infty}{(\Vert b \Vert_\infty - \epsilon)^2},
    \end{align*} 
    where \(\epsilon := \sup_{\vert z \vert \leq L}\sup_{x \in \mathbb{R}^d} \Vert b(\cdot, \cdot + \delta_{x}) - b(\cdot, \cdot + \delta_{x + z}) \Vert_\infty\).
    Using this, we have
    \begin{align*}
        &\mathbb{P}\Big(Y_{\tau_1, \Delta}^{(\omega, \widehat{\Lambda})} = Y_{\tau_1, \Delta}^{(\widetilde{\omega}, \widehat{\Lambda})} \Big)
        \geq \mathbb{P}\Big(\sigma_1 > \sigma_1^c\Big) \geq 1 - \delta_L.
    \end{align*}
    By repeated conditioning, we also find
    \begin{align*}
        \mathbb{P}\Big(\forall j \in \{1, \dots, n\}\colon Y_{\tau_j, \Delta}^{(\omega, \widehat{\Lambda})} = Y_{\tau_j, \Delta}^{(\widetilde{\omega}, \widehat{\Lambda})}\Big)
        \geq (1 - \delta_L)^n,
    \end{align*}
    as desired. 
\end{proof}

\subsection{Janossy densities and correlation functions}\label{section:nice_correlation_functions}

In order to prove \Cref{proposition:good_densities_of_the_birth_and_death_process}, we will show that the following function, written in terms of Janossy densities \(j_{n, \Lambda}^{(t, \underline{\omega})}\), is precisely a regular enough density of \((Y_t\Ssup{\underline{\omega}})_\Lambda\) w.r.t.\ the Poisson point process \(\pi_\Lambda\) in \(\Lambda\) we are looking for:
    \begin{equation}\label{eq:psi_janossy}
    \begin{split}
        &\psi_{t, \Lambda}(\zeta_\Lambda \,\vert\, \underline{\omega}) \\
        &:= \e^{\abs{\Lambda}} \Big(\sum_{n = 1}^{\infty} j_{n, \Lambda}^{(t, \underline{\omega})}(\zeta_\Lambda) \mathbbm{1}_{N_{\Lambda}(\zeta_\Lambda) = n} + \Big(1 - \sum_{n = 1}^{\infty} \frac{1}{n!} \int_{\Lambda^n} \, j_{n, \Lambda}^{(t, \underline{\omega})}(x_1, \dots, x_n) \, \d x_1 \dots \d x_n \Big) \mathbbm{1}_{N_{\Lambda}(\zeta_\Lambda) = 0}\Big).
    \end{split}
    \end{equation} 
    Here, for all \(n \in \mathbb{N}\) and \(x_1, \dots, x_n \in \mathbb{R}^d\), we construct the Janossy densities by the standard inversion formula (see e.g.~\cite[(5.4.14)]{DaleyVereJones2003})
    \begin{align*}
        j_{n, \Lambda}^{(t, \underline{\omega})}(x_1, \dots, x_n)
        = \sum_{k = 0}^{\infty} \frac{(-1)^k}{k!} \int_{\Lambda^k} \,\rho^{(t, \underline{\omega})}_{n+k}(x_1, \dots, x_{n+k}) \, \d x_{n+1} \dots \d x_{n+k},
    \end{align*} from the the {\em $n$-point correlation functions} \(\rho^{(t, \underline{\omega})}_n\) of \(Y_{t}\Ssup{\underline{\omega}}\).
    To be precise, we use 
    \begin{align*}
    \rho^{(t, \underline{\omega})}_n(x_1, \dots, x_n)
    = \begin{cases}
        \lim_{\ell \downarrow 0} \abs{Q_\ell}^{-n} \mathbb{E}[\prod_{i = 1}^{n} N_{Q_\ell(x_i))}(Y_t\Ssup{\underline{\omega}})], &\text{if } x_i \neq x_j \text{ for all } i \neq j, \\
        0, &\text{ otherwise},
    \end{cases}
\end{align*} where the limit is over boxes \(Q_\ell(x_i)\) with side-length \(\ell > 0\) and center \(x_i\).
The above limit exists by \Cref{lemma:correlation_function_lebesgue_points}, and \((x_1, \dots, x_n) \mapsto \rho_n^{(t, \underline{\omega})}(x_1, \dots, x_n)\) is actually (a version of) the $n$-point correlation function of \(Y_{t}\Ssup{\underline{\omega}}\), thanks to the Lebesgue differentiation theorem (in the precise form of~\cite[Theorem~B]{Jessen1934NoteOT}, the reason being that we have to use boxes instead of balls here), together with the fact that the \(Q_\ell(x_i)\) are pairwise disjoint for \(\ell > 0\) small enough. Moreover, by an argument similar to the proof of \Cref{lemma:birth_death_process_expected_number_of_points_continuity_in_box_centers}, we have constants \(z_{\min}, z_{\max} > 0\) such that
\begin{align}
    z_{\min}^n \, p_t^n (1-p_t)^n\lesssim 
    \rho_n^{(t, \cdot)} 
    \lesssim z_{\max}^n \, p_t^n (1-p_t)^n
\end{align} with \(p_t := \e^{-t}\).
We discuss differentiability in time \(t\) of \(\rho_n^{(t, \underline{\omega})}(x_1, \dots, x_n)\) in the following.
\begin{lemma}\label{le_differentiability_corr_func}
    For fixed \(x_1, \dots, x_n\), we have that \(t \mapsto \rho^{(t, \underline{\omega})}_n(x_1, \dots, x_n)\) is right-differentiable at every \(t \geq 0\) and the right-derivative is given by
    \begin{align*}
        &\rightDerivative{t} \rho_n^{(t, \underline{\omega})}(x_1, \dots, x_n) \\
        &= \sum_{i = 1}^{n} \lim_{\ell \downarrow 0}\abs{Q_\ell}^{-n} \mathbb{E}\Big[\Big(\int_{Q_\ell(x_i)} \, b(x, X_t\Ssup{\underline{\omega}}) \, \d x\Big) \prod_{j \neq i} N_{Q_\ell(x_j)}(Y_t\Ssup{\underline{\omega}}) \Big]  - n \, \rho_n^{(t, \underline{\omega})}(x_1, \dots, x_n),
    \end{align*} 
    and in particular,
    \begin{enumerate}
        \item the right-derivative \(q\Ssup{\underline{\omega}}_{x_1, \dots, x_n}(t)\) is continuous at every time \(t \not\in \{r \colon  \exists (x, r) \in \underline{\omega} \text{ with }x \in \Lambda \oplus B(0, \range)\}\) that is not one of the (finitely many) deterministic lifespans of points of \(\underline{\omega}\) in \(\Lambda \oplus B(0, \range)\), and we have 
        \item the bounds \begin{align*}
        \abs{\rightDerivative{t} \rho_n^{(t, \underline{\omega})}(x_1, \dots, x_n)}
        &\leq  \Vert \birth\Vert_\infty \sum_{i = 1}^{n} \rho_{n-1}^{(t, \underline{\omega})}(x_1, \dots, x_{i-1}, x_{i+1}, \dots, x_n) + n \, \rho_n^{(t, \underline{\omega})}(x_1, \dots, x_n) \\
        &\lesssim n z_{\max}^{n-1} \, (1-p_t)^{n-1}.
    \end{align*}
    \end{enumerate}
\end{lemma}
\begin{proof}
    Let \(x_1, \dots, x_n \in \mathbb{R}^d\) be pairwise distinct (otherwise there is not much to show by definition of \(\rho^{(t, \underline{\omega})}_n\)). 
    Denote \(f_\ell(\zeta) = \prod_{i = 1}^{n} N_{Q_\ell(x_i)}(\zeta)\).
    We will compute and bound
    \begin{align*}
        &\mathbb{E}\big[f_\ell(Y_{t+h}\Ssup{\underline{\omega}}) \big] - \mathbb{E}\big[f_\ell(Y_{t}\Ssup{\underline{\omega}}) \big]\\
        &= h \Big(\Big( \sum_{i = 1}^{n} \mathbb{E}\Big[ \Big(\int_{Q_\ell(x_i)}  b(x, X_{t}\Ssup{\underline{\omega}}) \d x\Big) \prod_{j \neq i} N_{Q_\ell(x_j)}(Y_{t}\Ssup{\underline{\omega}}) \Big] \Big) - n \,\mathbb{E}\big[f_\ell(Y_{t}\Ssup{\underline{\omega}}) \big] \Big) + O(h^2) \abs{Q_\ell}^n \\
        &=: h\, A_\ell + O(h^2) \abs{Q_\ell}^n.
    \end{align*}
    Using this, we can easily commute and compute limits as follows
    \begin{align*}
        \lim_{h \downarrow 0} \frac{\rho_n^{(t+h, \underline{\omega})}(x_1, \dots, x_n) - \rho_n^{(t, \underline{\omega})}(x_1, \dots, x_n)}{h} 
        &= \lim_{h \downarrow 0} \limsup_{\ell \downarrow 0} \frac{1}{h\abs{Q_\ell}^n } \big(\mathbb{E}\big[f_\ell(Y_{t+h}\Ssup{\underline{\omega}}) \big] - \mathbb{E}\big[f_\ell(Y_{t}\Ssup{\underline{\omega}}) \big] \big) \\
        &= \lim_{h \downarrow 0} \Big( \lim_{\ell \downarrow 0} \frac{A_l}{\abs{Q_\ell}^n} + O(h^2) \Big)
        = \lim_{\ell \downarrow 0} \frac{A_\ell}{\abs{Q_\ell}^n}.
    \end{align*} The existence of \(\lim_{\ell \downarrow 0} A_\ell/\abs{Q_\ell}^n\) can be proved very similar to \Cref{lemma:correlation_function_lebesgue_points}.

    We denote by $D_{Q_\ell(x_i)}(t)$ the set of times $s\leq t$ in which a death occurs in $Q_\ell(x_i)$ and have, as in \Cref{lem:domain},
    \begin{equation}\label{eq:difference t and t+h}
    \begin{split}
        \mathbb{E}\left[f_\ell(Y_{t+h}\Ssup{\underline{\omega}}) \right] - \mathbb{E}\left[f_\ell(Y_{t}\Ssup{\underline{\omega}}) \right] 
        &= \sum_{i = 1}^{n} \int_{t}^{t+h} \d s \,    \int_{Q_\ell(x_i)} \d x \, \mathbb{E}\left[\birth(x, X_{s}\Ssup{\underline{\omega}}) \{f_\ell(Y_{s}\Ssup{\underline{\omega}} + \delta_x) - f_\ell(Y_{s}\Ssup{\underline{\omega}}\} \right] \\
        &\qquad + \sum_{i = 1}^{n} \, \mathbb{E}\Big[\sum_{s\in D_{Q_\ell(x_i)}(t+h)\setminus D_{Q_\ell(x_i)}(t)} \{f_\ell(Y_{s}\Ssup{\underline{\omega}}) - f_\ell(Y_{s-}\Ssup{\underline{\omega}})\}\Big].
    \end{split}
    \end{equation}
    For the first summand in \Cref{eq:difference t and t+h}, setting 
    \begin{align*}
        G(s) :&= 
        \int_{Q_\ell(x_i)} \d x \, \mathbb{E}\left[\birth(x, X_{s}\Ssup{\underline{\omega}}) \{f_\ell(Y_{s}\Ssup{\underline{\omega}} + \delta_x) - f_\ell(Y_{s}\Ssup{\underline{\omega}}\} \right]= \int_{Q_\ell(x_i)} \d x \, \mathbb{E}\Big[\birth(x, X_{s}\Ssup{\underline{\omega}}) \prod_{j \neq i} N_{Q_\ell(x_j)}(Y_{s}\Ssup{\underline{\omega}}) \Big],
    \end{align*} 
    we see that
    \begin{equation}
    \begin{split}\label{eq:estimate_remainder_derivative_birth_part}
        &\Big\vert \int_{t}^{t+h}  \,  G(s) \d s - h G(t) \Big\vert
        \leq \int_{0}^{h} \, \d s \left\vert G(t+h) - G(t) \right\vert \\
        &\leq \int_{0}^{h} \, \d s \, \int_{Q_\ell(x_i)} \d x \, \Big\vert \mathbb{E}\Big[\birth(x, X_{t+h}\Ssup{\underline{\omega}}) \prod_{j \neq i} N_{Q_\ell(x_j)}(Y_{t+h}\Ssup{\underline{\omega}}) \Big] - \mathbb{E}\Big[\birth(x, X_{t}\Ssup{\underline{\omega}}) \prod_{j \neq i} N_{Q_\ell(x_j)}(Y_{t}\Ssup{\underline{\omega}}) \Big] \Big\vert \\
&\stackrel{\mathclap{(*)}}{\leq} h^2 \abs{Q_\ell}^{n} 2 \Vert b \Vert_\infty (\Vert \birth\Vert_\infty (t+h))^{n-1} \\
        &\quad \Big[(n-1)\Big( \frac{1 - \e^{-h}}{h} + \frac{1-\e^{-\abs{Q_\ell}(t+h)(1-\e^{-h})}}{h} + 1 + \Vert \birth\Vert_\infty t \abs{Q_\ell} \Big) + \abs{Q_\ell} \Vert \birth \Vert_\infty \Big((t+h) \frac{1-\e^{-h}}{h} +  1 \Big) \Big] \\
        &\lesssim_{n, t, \birth} h^2 \abs{Q_\ell}^{n},
    \end{split}
    \end{equation} 
    where $(*)$ is proved as follows.    For a fixed bounded region \(\Lambda\) and \(h > 0\) small enough, there will be no death in \(\underline{\omega}_{\Lambda}\), therefore only \(Y\Ssup{\underline{\omega}}\) is responsible for a possibly non-zero difference \(\birth(x, X_{t+h}\Ssup{\underline{\omega}}) - \birth(x, X_{t}\Ssup{\underline{\omega}})\).
    Therefore, the difference \(\birth(x, X_{t+h}\Ssup{\underline{\omega}}) \prod_{j \neq i} N_{Q_\ell(x_j)}(Y_{t+h}\Ssup{\underline{\omega}}) - \birth(x, X_{t}\Ssup{\underline{\omega}}) \prod_{j \neq i} N_{Q_\ell(x_j)}(Y_{t}\Ssup{\underline{\omega}})\) can only be non-zero if the dominating Poisson point process \(\zeta\) (coming from \(\mathcal{N}\) in the stochastic integral) on \(\mathbb{R}^d \times \mathbb{R}_+ \times \mathbb{R}_+\) (birth location, birth time, lifespan) with intensity measure \(\Vert \birth\Vert_\infty \d x \otimes \d s \otimes \e^{-r} \d r\) births points in \([t, t+h]\) or has born points in \([0, t+h]\) which die in \([t, t+h]\). Hence,
    \begin{align*}
        &\Big\vert \birth(x, X_{t+h}\Ssup{\underline{\omega}}) \prod_{j \neq i} N_{Q_\ell(x_j)}(Y_{t+h}\Ssup{\underline{\omega}}) - \birth(x, X_{t}\Ssup{\underline{\omega}}) \prod_{j \neq i} N_{Q_\ell(x_j)}(Y_{t}\Ssup{\underline{\omega}}) \Big\vert \\
        &\leq 2 \Vert \birth \Vert_\infty \Big(\prod_{j \neq i} N_{Q_\ell(x_j) \times [0, t+h] \times \mathbb{R}_+}(\zeta) \Big) \1_{N_{\Lambda \times [0, t+h] \times [0, h]} > 0 \text{ or } N_{\Lambda \times [t, t+h] \times \mathbb{R}_+} > 0}(\zeta),
    \end{align*} 
    with \(\Lambda = B_{\range}(x) \cup \bigcup_{j \neq i} Q_\ell(x_j)\), where we recall that \(\range\) is the range of the interaction.
    The product on the right-hand side can be bounded and expanded as
    \begin{align*}
        &\Big(\prod_{j \neq i} N_{Q_\ell(x_j) \times [0, t+h] \times \mathbb{R}_+}(\zeta) \Big) \1_{N_{\Lambda \times [0, t+h] \times [0, h]} > 0 \text{ or } N_{\Lambda \times [t, t+h] \times \mathbb{R}_+} > 0}(\zeta) \\
        &\leq \Big(\prod_{j \neq i} N_{Q_\ell(x_j) \times [0, t+h] \times \mathbb{R}_+}(\zeta) \Big) \left(\1_{N_{\Lambda \times [0, t+h] \times [0, h]} > 0}(\zeta) + \1_{N_{\Lambda \times [t, t+h] \times \mathbb{R}_+} > 0}(\zeta) \right) \\
        &\leq \sum_{k = 1}^{n} \Big(\prod_{j \neq i} N_{Q_\ell(x_j) \times [0, t+h] \times \mathbb{R}_+}(\zeta) \Big) \left(\1_{N_{\Lambda_k \times [0, t+h] \times [0, h]} > 0}(\zeta) + \1_{N_{\Lambda_k \times [t, t+h] \times \mathbb{R}_+} > 0}(\zeta) \right),
    \end{align*} 
    where \(\Lambda_k = Q_\ell(x_k)\) for \(k \neq i\) and \(\Lambda_i = B_{\range}(x) \setminus \bigcup_{k \neq i} Q_\ell(x_k)\).
    
    For \(k \neq i\), by independence, we have
    \begin{align*}
        &\mathbb{E}\Big[\prod_{j \neq i} N_{Q_\ell(x_j) \times [0, t+h] \times \mathbb{R}_+}(\zeta) \1_{N_{\Lambda_k \times [0, t+h] \times [0, h]} > 0}(\zeta)  \Big] 
        = \mathbb{E}\Big[\prod_{j \neq i,k} N_{Q_\ell(x_j) \times [0, t+h] \times \mathbb{R}_+}(\zeta) \Big]\\
        &\qquad\qquad \times\mathbb{E}\left[\left(N_{Q_\ell(x_j) \times [0, t+h] \times [0, h]}(\zeta) + N_{Q_\ell(x_j) \times [0, t+h] \times [0, h]^c}(\zeta)\right) \1_{N_{Q_\ell(x_k) \times [0, t+h] \times [0, h]} > 0}(\zeta)  \right] \\
        &= (\Vert \birth\Vert_\infty (t+h) \abs{Q_\ell})^{n-2}  \biggl( \mathbb{E}\left[N_{Q_\ell(x_j) \times [0, t+h] \times [0, h]}(\zeta) \1_{N_{Q_\ell(x_k) \times [0, t+h] \times [0, h]} > 0}(\zeta) \right] \\
        &\qquad\qquad + \mathbb{E}\left[N_{Q_\ell(x_j) \times [0, t+h] \times [0, h]^c}(\zeta)\right] \mathbb{P}(N_{Q_\ell(x_k) \times [0, t+h] \times [0, h]} (\zeta) > 0) \biggr) \\
        &\leq (\Vert \birth\Vert_\infty (t+h) \abs{Q_\ell})^{n-2}
        \left(\Vert \birth\Vert_\infty \abs{Q_\ell} (t+h)  (1 - \e^{-h}) + \Vert\birth\Vert_\infty \abs{Q_\ell} (t+h) (1-\e^{-\abs{Q_\ell}(t+h)(1-\e^{-h})}) \right) \\
        &= (\Vert \birth\Vert_\infty (t+h) \abs{Q_\ell})^{n-1} \Big( (1 - \e^{-h}) + (1-\e^{-\abs{Q_\ell}(t+h)(1-\e^{-h})})  \Big).
    \end{align*}
    In the case \(k = i\), we have
    \begin{align*}
        &\mathbb{E}\Big[\prod_{j \neq i} N_{Q_\ell(x_j) \times [0, t+h] \times \mathbb{R}_+}(\zeta) \1_{N_{\Lambda_i \times [0, t+h] \times [0, h]} > 0}(\zeta)  \Big] \\
        &\leq \mathbb{E}\Big[\prod_{j \neq i} N_{Q_\ell(x_j) \times [0, t+h] \times \mathbb{R}_+}(\zeta) \Big] \mathbb{E}\left[N_{\Lambda_i \times [0, t+h] \times [0, h]}(\zeta) \right] \\
        &\leq (\Vert \birth\Vert_\infty (t+h))^{n} \abs{Q_\ell}^{n-1} \abs{Q_\ell} (1-\e^{-h}).
    \end{align*}
    Analogously, we see that, for \(k\neq i\),
    \begin{align*}
        &\mathbb{E}\Big[\prod_{j \neq i} N_{Q_\ell(x_j) \times [0, t+h] \times \mathbb{R}_+}(\zeta) \1_{N_{\Lambda_k \times [t, t+h] \times \mathbb{R}_+} > 0}(\zeta)  \Big] \leq (\Vert \birth\Vert_\infty (t+h) \abs{Q_\ell})^{n-1} (1 + \Vert \birth\Vert_\infty t \abs{Q_\ell}) h,
    \end{align*} 
    and for \(k = i\),
    \begin{align*}
        &\mathbb{E}\Big[\prod_{j \neq i} N_{Q_\ell(x_j) \times [0, t+h] \times \mathbb{R}_+}(\zeta) \1_{N_{\Lambda_i \times [t, t+h] \times \mathbb{R}_+} > 0}(\zeta)  \Big]\leq (\Vert \birth\Vert_\infty (t+h) \abs{Q_\ell})^{n-1} \Vert \birth\Vert_\infty  \abs{B_L} h.
    \end{align*}
    It then follows that
    \begin{align*}
        &\Big\vert \mathbb{E}\Big[\birth(x, X_{t+h}\Ssup{\underline{\omega}}) \prod_{j \neq i} N_{Q_\ell(x_j)}(Y_{t+h}\Ssup{\underline{\omega}}) \Big] - \mathbb{E}\Big[\birth(x, X_{t}\Ssup{\underline{\omega}}) \prod_{j \neq i} N_{Q_\ell(x_j)}(Y_{t}\Ssup{\underline{\omega}}) \Big] \Big\vert \\
        &\leq \abs{Q_\ell}^{n-1} 2 \Vert b \Vert_\infty (\Vert \birth\Vert_\infty (t+h))^{n-1} \\
        &\qquad \Big[(n-1)\Big( (1 - \e^{-h}) + (1-\e^{-\abs{Q_\ell}(t+h)(1-\e^{-h})}) + (1 + \Vert \birth\Vert_\infty t \abs{Q_\ell}) h \Big) \\
        &\qquad\qquad\qquad + \abs{Q_\ell} \Vert \birth \Vert_\infty \Big((t+h) (1-\e^{-h}) +  h \Big) \Big],
    \end{align*} 
    proving $(*)$ in \Cref{eq:estimate_remainder_derivative_birth_part}.

    Let us now look at the death part of \Cref{eq:difference t and t+h}.
    We have
    \begin{align*}
        &\mathbb{E}\Big[\sum_{s\in D_{Q_\ell(x_i)}(t+h)\setminus D_{Q_\ell(x_i)}(t)} \{f_l(Y_{s}\Ssup{\underline{\omega}}) - f_l(Y_{s-}\Ssup{\underline{\omega}})\}\Big]
        = - \mathbb{E}\Big[\sum_{s\in D_{Q_\ell(x_i)}(t+h)\setminus D_{Q_\ell(x_i)}(t)} a \prod_{i \neq j} N_{Q_\ell(x_j)}(Y_{s-}\Ssup{\underline{\omega}}) \Big] \\
        &= - \mathbb{E}\Big[\Big(\sum_{s\in D_{Q_\ell(x_i)}(t+h)\setminus D_{Q_\ell(x_i)}(t)} \quad \prod_{i \neq j} N_{Q_\ell(x_j)}(Y_{s-}\Ssup{\underline{\omega}})\Big) \1_{\subalign{&\text{in $Q_\ell(x_i)$, during $(t,t+h]$,}\\ &\text{at most one particle dies,}\\ &\text{and it was born in $[0, t]$}}} \Big] \\
        &\qquad - \mathbb{E}\Big[\Big(\sum_{s\in D_{Q_\ell(x_i)}(t+h)\setminus D_{Q_\ell(x_i)}(t)} \quad \prod_{i \neq j} N_{Q_\ell(x_j)}(Y_{s-}\Ssup{\underline{\omega}})\Big) \1_{\subalign{&\text{in $Q_\ell(x_i)$, during $(t,t+h]$,}\\ &\text{more than one particle dies}  \\ &\text{or a particle born in $(t, t+h]$ dies}}}  \Big].
    \end{align*}
    For the first summand, we see as in \Cref{lem:domain} that
    \begin{align*}
        &\mathbb{E}\Big[\Big(\sum_{s\in D_{Q_\ell(x_i)}(t+h)\setminus D_{Q_\ell(x_i)}(t)} \quad \prod_{i \neq j} N_{Q_\ell(x_j)}(Y_{s-}\Ssup{\underline{\omega}})\Big) \1_{\subalign{&\text{in $Q_\ell(x_i)$, during $(t,t+h]$,}\\ &\text{at most one particle dies,}\\ &\text{and it was born in $[0, t]$}}} \Big] \\
        &= \mathbb{E}\Big[\sum_{x \in (Y_{t}\Ssup{\underline{\omega}})_{Q_\ell(x_i)}}  \Big(\prod_{j \neq i} N_{Q_\ell(x_j)}(Y_{t}\Ssup{\underline{\omega}})\Big) \1_{\inf\{s \geq 0 \colon x \not\in Y_t\Ssup{\underline{\omega}}\} \in (t, t+h]} \Big] \\
        &= (1-\e^{-h}) \mathbb{E}\Big[\sum_{x \in (Y_{t}\Ssup{\underline{\omega}})_{Q_\ell(x_i)}} \left(\prod_{j \neq i} N_{Q_\ell(x_j)}(Y_{t}\Ssup{\underline{\omega}})\right) \Big] = (1-\e^{-h}) \mathbb{E}\Big[\prod_{j = 1}^{n} N_{Q_\ell(x_j)}(Y_{t}\Ssup{\underline{\omega}}) \Big].
    \end{align*}
    For the remaining term, we have
    \begin{align*}
        &\mathbb{E}\Big[\Big(\sum_{s\in D_{Q_\ell(x_i)}(t+h)\setminus D_{Q_\ell(x_i)}(t)} \quad \prod_{i \neq j} N_{Q_\ell(x_j)}(Y_{s-}\Ssup{\underline{\omega}})\Big) \1_{\subalign{&\text{in $Q_\ell(x_i)$, during $(t,t+h]$,}\\ &\text{more than one particle dies}  \\ &\text{or a particle born in $(t, t+h]$ dies}}}\Big] \\
        &\leq \mathbb{E}\Big[\Big(\sum_{s\in D_{Q_\ell(x_i)}(t+h)\setminus D_{Q_\ell(x_i)}(t)} \quad \prod_{i \neq j} N_{Q_\ell(x_j)}(Y_{s-}\Ssup{\underline{\omega}})\Big) \1_{\subalign{&\text{in $Q_\ell(x_i)$, during $(t,t+h]$,}\\ &\text{more than one particle dies}}}\Big] \\
        &\qquad + \mathbb{E}\Big[\Big(\sum_{s\in D_{Q_\ell(x_i)}(t+h)\setminus D_{Q_\ell(x_i)}(t)} \quad \prod_{i \neq j} N_{Q_\ell(x_j)}(Y_{s-}\Ssup{\underline{\omega}})\Big) \1_{\subalign{&\text{in $Q_\ell(x_i)$, during $(t,t+h]$,}\\ &\text{a particle born in $(t, t+h]$ dies}}}\Big].
    \end{align*} 
    Herein, $
        \prod_{i \neq j} N_{Q_\ell(x_j)}(Y_{s-}\Ssup{\underline{\omega}})
        \leq \prod_{i \neq j} N_{Q_\ell(x_j) \times [0, t+h] \times \mathbb{R}_+}(\zeta)$.  It follows, by independence,
    \begin{align*}
        &\mathbb{E}\Big[\Big(\sum_{s\in D_{Q_\ell(x_i)}(t+h)\setminus D_{Q_\ell(x_i)}(t)} \quad \prod_{i \neq j} N_{Q_\ell(x_j)}(Y_{s-}\Ssup{\underline{\omega}})\Big) \1_{\subalign{&\text{in $Q_\ell(x_i)$, during $(t,t+h]$,}\\ &\text{a particle born in $(t, t+h]$ dies}}}\Big] \\
        &\leq \mathbb{E}\Big[\Big(\prod_{i \neq j} N_{Q_\ell(x_j) \times [0, t+h] \times \mathbb{R}_+}(\zeta)\Big) \mathbbm{1}_{N_{Q_\ell(x_i) \times (t, t+h] \times [0, h]} > 0}(\zeta) \Big] \leq \Vert b \Vert_\infty^n (t+h)^{n-1} \abs{Q_\ell}^n h (\e^h - 1).
    \end{align*}
    Similarly, 
    \begin{align*}
        &\mathbb{E}\Big[\Big(\sum_{s\in D_{Q_\ell(x_i)}(t+h)\setminus D_{Q_\ell(x_i)}(t)} \quad \prod_{i \neq j} N_{Q_\ell(x_j)}(Y_{s-}\Ssup{\underline{\omega}})\Big) \1_{\subalign{&\text{in $Q_\ell(x_i)$, during $(t,t+h]$,}\\ &\text{more than one particle dies}}}\Big] \\
        &\leq \mathbb{E}\Big[\Big(\prod_{i \neq j} N_{Q_\ell(x_j) \times [0, t+h] \times \mathbb{R}_+}(\zeta)\Big)  \1_{N_{Q_\ell(x_i) \times E} \geq 2}(\zeta)\Big] \\
        &\leq \Vert \birth \Vert_\infty^{n-1} (t+h)^{n-1} \abs{Q_\ell}^{n-1} \mathbb{P}(N_{Q_\ell(x_i) \times E}(\zeta) \geq 2)\leq \Vert \birth \Vert_\infty^{n-1} (t+h)^{n-1} \abs{Q_\ell}^{n-1} (h(1-\e^{-(t+1)}) \abs{Q_\ell})^2,
    \end{align*} 
    where \(E := \{(s, r) \in \mathbb{R}_+^2 \colon s+r \in [t, t+h] \}\) with \((\d s \otimes \e^{-r}\d r)(E) = h - \e^{-t}(1-\e^{-h}) \leq h(1-\e^{-(t+1)})\).
    Summing up,
    \begin{align*}
        &\mathbb{E}\Big[\sum_{s\in D_{Q_\ell(x_i)}(t+h)\setminus D_{Q_\ell(x_i)}(t)} \big(f_l(Y_{s}\Ssup{\underline{\omega}}) - f_l(Y_{s-}\Ssup{\underline{\omega}})\big)\Big] 
        = (1-\e^{-h}) \mathbb{E}\Big[\prod_{j = 1}^{n} N_{Q_\ell(x_j)}(Y_{t}\Ssup{\underline{\omega}}) \Big] + \mathrm{const}_{n, t, \birth} h^2 \abs{Q_\ell}^n.
    \end{align*}
    Putting it all together, we now see
    \begin{align*}
        &\lim_{h \downarrow 0} \frac{1}{h}\big(\rho_n^{(t+h)}(x_1, \dots, x_n) - \rho_n^{(t)}(x_1, \dots, x_n)\big) \\
        &= \sum_{i = 1}^{n} \lim_{\ell \downarrow 0} \abs{Q_\ell}^{-n}\mathbb{E}\Big[\Big(\int_{Q_\ell(x_i)} \, b(x, X_{t}\Ssup{\underline{\omega}}) \, \d x\Big) \prod_{j \neq i} N_{Q_\ell(x_j)} \Big]   - n \, \rho_n^{(t)}(x_1, \dots, x_n),
    \end{align*}
    as desired. 
\end{proof}
The regularity properties of the correlation functions easily transfer to properties of the Janossy densities, in the following precise sense. 
\begin{lemma}\label{le_differentiability_janossys}
    For every \(n \in \mathbb{N}\) there is a measurable version \((t, \underline{\omega}, x_1, \dots, x_n) \mapsto j^{(t, \underline{\omega})}_{n, \Lambda}(x_1, \dots, x_n)\) of the \(n\)-th Janossy density of \((Y_t\Ssup{\underline{\omega}})_\Lambda\), such that 
    \begin{align*}
        t \mapsto j^{(t, \underline{\omega})}_{n, \Lambda}(x_1, \dots, x_n)
    \end{align*} is right-differentiable in \(t \geq 0\) for distinct \(x_1, \dots, x_n\).
    Furthermore, 
    \begin{enumerate}
        \item the right-derivative \(\rightDerivative{t} j^{(t, \underline{\omega})}_{n, \Lambda}(x_1, \dots, x_n)\) is continuous at every time \(t \not\in \{r \colon \exists (x, r) \in \underline{\omega} \text{ with } x \in \Lambda \oplus B(0, \range)\}\) that is not one of the (finitely many) deterministic lifespans of points of \(\underline{\omega}\) in \(\Lambda \oplus B(0, \range)\),
        \item there exists a \(t_0 > 0\) such that for all \(t \leq t_0\) we have the bounds
    \begin{align*}
        j^{(t, \cdot)}_{n, \Lambda}
        \lesssim
        z_{\max}^n \, (1-p_t)^n
    \end{align*}
    and
    \begin{align*}
        \abs{\rightDerivative{t} \, j^{(t, \underline{\omega})}_{n, \Lambda}(x_1, \dots, x_n)}
        \lesssim 
        n z_{\max}^{n-1} \, (1-p_t)^{n-1}.
    \end{align*}
    \end{enumerate}
\end{lemma}
\begin{proof}
    Recalling that
    \begin{align*}
        j_{n, \Lambda}^{(t, \underline{\omega})}(x_1, \dots, x_n)
        = \sum_{k = 0}^{\infty} \frac{(-1)^k}{k!} \int_{\Lambda^k} \,\rho^{(t, \underline{\omega})}_{n+k}(x_1, \dots, x_{n+k}) \, \d x_{n+1} \dots \d x_{n+k},
    \end{align*} 
    the differentiability now follows from the differentiability of the \(\rho^{(t, \underline{\omega})}_m\), shown in \Cref{le_differentiability_corr_func}, and dominated convergence using the bound given in the same lemma. The bounds also follows using the bounds given for the corresponding correlation functions and their derivatives.
\end{proof}

\begin{proof}[Proof of \Cref{proposition:good_densities_of_the_birth_and_death_process}]
    Recalling that
    \begin{align*}
        &\psi_{t, \Lambda}(\zeta_\Lambda \,\vert\, \underline{\omega}) = \e^{\abs{\Lambda}} \Big(\sum_{n = 1}^{\infty} j_{n, \Lambda}^{(t, \underline{\omega})}(\zeta_\Lambda) \mathbbm{1}_{N_{\Lambda}(\zeta_\Lambda) = n} + \Big(1 - \sum_{n = 1}^{\infty} \frac{1}{n!} \int_{\Lambda^n} \, j_{n, \Lambda}^{(t, \underline{\omega})}(x_1, \dots, x_n) \, \d x_1 \dots \d x_n \Big) \mathbbm{1}_{N_{\Lambda}(\zeta_\Lambda) = 0}\Big),
    \end{align*} 
    the first part of this lemma follows from \Cref{le_differentiability_janossys}, together with dominated convergence, while the second part, concerning \(\psi_{\Lambda, t}(\emptyset \,\vert\, \underline{\omega}) = \e^{\abs{\Lambda}} \mathbb{P}\big((Y_{t}\Ssup{\underline{\omega}})_\Lambda = \emptyset\big) \sim 1\), follows by considering the driving Poisson noise, similarly to \Cref{lemma:birth_death_process_expected_number_of_points_continuity_in_box_centers}.
\end{proof}

\section{Entropy dissipation}\label{section:entropy_dissipation}

We now finally discuss the time-evolution of the specific relative entropy 
$
    \SpecEnt(\mu \,\vert\, \nu) = \lim_{\Lambda \uparrow \mathbb{R}^d} \RelEnt_\Lambda(\mu \,\vert\, \nu)
$
under our dynamics \((T_t)_{t \geq 0}\).
In \Cref{section:entropy_dissipation_local_equality} we justify the local computation 
\begin{equation*}
    \frac{\d}{\d t}\Big\vert_{t=0} \RelEnt_\Lambda(\mu T_t \,\vert\, \nu) = \mu\Big[\mathscr{L} \log\frac{\d\mu_\Lambda}{\d \nu_\Lambda}\Big],
\end{equation*}
which will be used in conjunction with the corresponding thermodynamic limit 
established in \Cref{section:entropy_dissipation_thermodynamic_limit} and a limit-interchange to prove the de~Bruijn-type identity \Cref{eq:de_bruijn} in our main \Cref{thm:main_thm} in \Cref{sec:de_bruijn_proof}.

\subsection{Entropy dissipation in finite volumes} \label{section:entropy_dissipation_local_equality}

In this section, we will carry out the proof of \Cref{theorem_derivative_of_the_local_entropy}. 
We start with the rigorous analysis of the behavior, in particular existence and differentiability properties, of the local densities \(\d(\mu T_t)_\Lambda/\d \nu_\Lambda\), which appear in the definition of \(\RelEnt_{\Lambda}(\mu T_t \,\vert\, \nu)\).

\subsubsection{Evolution of Local Densities}\label{section:evolution_of_local_densities}

The proof of \Cref{prop:differentiability-of-densities} is split into the three lemmas below.
First, we need a small probabilistic input to later separate the contributions of the dying initial condition \(\underline{\omega}_t\) and the newly-born-but-maybe-already-dead points \(Y_t\Ssup{\underline{\omega}}\) to the birth-and-death-process \(X_t\Ssup{\underline{\omega}}\).
In the next lemma and correspondingly in all other formulas of this section, one always has to do the following interpretations:
\begin{align*}
    \sum_{x \in \emptyset} \dots = 0 \quad\text{and}\quad \int (\d s_x)_{x \in \emptyset} \dots
    = 0.
\end{align*}

\begin{lemma}[Poisson variable change]\label{le_variable_change_poisson_point_process}
    Let \(t > 0\) be fixed and \(\Lambda \Subset \mathbb{R}^d\).
    Consider two independent Poisson point processes \(\omega_\Lambda, \zeta_\Lambda\) of intensity \(1\) and the iid marked version \(\underline{\omega}_\Lambda\) of \(\omega_\Lambda\), where the marks are \(\mathrm{Exp}(1)\)-distributed random variables. Let
    \begin{align*}
        \chi_\Lambda = \underline{\omega}^{(t)}_\Lambda + \zeta_{\Lambda},\quad \text{with }\ 
        \underline{\omega}^{(t)}_\Lambda
        = \sum_{(x, R_x) \in \underline{\omega}_\Lambda} \1_{R_x > t} \, \delta_{x}.
    \end{align*}
    Now consider the following parallel construction. Let \(\widetilde{\chi}_\Lambda\) be a Poisson point process of intensity \(1 + p_t\), with \(p_t := \e^{-t}\). 
    Define
    \begin{align*}
        \widetilde{\zeta}_\Lambda
        &:= \widetilde{\chi}_\Lambda^{((\frac{1}{1+p_t})-\mathrm{Thinned})}
    \end{align*} to be the resulting point process if we independently delete points from \(\chi_\Lambda\) with probability \(1 - 1/(1+p_t)\) and denote
    \begin{align*}
        \widehat{\omega}_\Lambda &:= \widetilde{\chi}_\Lambda \setminus  \widetilde{\zeta}_\Lambda.
    \end{align*}
    Now, let \(\underline{\widehat{\omega}}_\Lambda\) be the marked point process obtained by iid marking \(\widehat{\omega}_\Lambda\) with \(\mathrm{Exp}(1)\)-random variables conditioned to be  \(> t\) and \(\underline{\widehat{\beta}}_{\Lambda}\) be an auxiliary Poisson point process of intensity \(1 - p_t\) iid marked with  \(\mathrm{Exp}(1)\)-random variables conditioned to be \(\le t\).
    Finally, let
    \begin{align*}
        \underline{\widetilde{\omega}}_\Lambda 
        = \underline{\widehat{\omega}}_\Lambda + \underline{\widehat{\beta}}_\Lambda.
    \end{align*}
    Then, the following equality in distribution holds:
    \begin{align*}
    \left(\underline{\omega}_\Lambda, \, \zeta_\Lambda, \, \chi_\Lambda \right)
        \stackrel{\text{d}}{=}
        \big(\underline{\widetilde{\omega}}_{\Lambda}, \, \widetilde{\zeta}_\Lambda, \, \widetilde{\chi}_\Lambda\big).
    \end{align*} 
    Equivalently, for all non-negative measurable functions \(F\), we have
    \begin{align*}
        \int \pi_\Lambda(\d \omega_\Lambda) \int &M(\d\underline{\omega}_\Lambda \,\vert\, \omega_\Lambda) \int \pi_\Lambda(\d \zeta_\Lambda) \, F(\underline{\omega}_\Lambda, \zeta_\Lambda, \underline{\omega}_\Lambda^{(t)} \cup \zeta_\Lambda) \\
        &= \int \pi^{1 + p_t}(\d \chi_\Lambda) \int K_t(\d \underline{\omega}_{\Lambda}, \d \zeta_\Lambda \,\vert\, \chi_\Lambda)  \, F(\underline{\omega}_\Lambda, \zeta_\Lambda, \chi_\Lambda),
    \end{align*} 
    where \(M\) is the iid marking with  \(\mathrm{Exp}(1)\)-distributed random variables and the kernel \(K_t\) is given by
    \begin{align*}
        &\int K_t(\d \underline{\omega}_{\Lambda}, \d \zeta_\Lambda \,\vert\, \chi_\Lambda) \, f(\underline{\omega}_\Lambda, \zeta_\Lambda)
        := \int \mathrm{Thin}^{\frac{1}{1+p_t}}(\d \zeta_\Lambda \,\vert\, \chi_\Lambda) \int \pi^{1-p_t}(\d \beta_\Lambda) \int M^{> t}(\d \underline{\chi_\Lambda \setminus \zeta_\Lambda} \,\vert\, \chi_\Lambda \setminus \zeta_\Lambda) \\
        &\hspace{7cm}
        \int M^{\leq t}(\underline{\beta}_\Lambda \,\vert\, \beta_\Lambda) \, f(\underline{\chi_\Lambda \setminus \zeta_\Lambda} \cup \underline{\beta}_\Lambda , \zeta_\Lambda) \\
        &:= \sum_{\zeta_\Lambda \subseteq \chi_\Lambda} \Big(\frac{1}{1+p_t}\Big)^{\abs{\zeta_\Lambda}} \Big(1 - \frac{1}{1+p_t}\Big)^{\abs{\chi_\Lambda \setminus \zeta_\Lambda}} 
        \int \pi(\d \beta_\Lambda) \e^{p_t \abs{\Lambda}} (1-p_t)^{\abs{\beta_\Lambda}} \\
        &\qquad
        \int_{(t, \infty)^{\chi_\Lambda \setminus \zeta_\Lambda}}  \d (s_x)_{x \in \chi_\Lambda \setminus \zeta_\Lambda} \,  \e^{- \sum_{x \in \chi_\Lambda \setminus \zeta_\Lambda} s_x} p_t^{- \abs{\chi_\Lambda \setminus \zeta_\Lambda}}\int_{[0,t]^{\beta_\Lambda}} \d (s_x)_{x \in \beta_\Lambda} \,  \e^{- \sum_{x \in \beta_\Lambda} s_x}(1-p_t)^{-\abs{\beta_\Lambda}} \\ 
        &\qquad \qquad f\Big(\sum_{x \in \chi_\Lambda \setminus \zeta_\Lambda} \delta_{(x, s_x)} + \sum_{y \in \beta_\Lambda} \delta_{(y, s_y)}, \, \zeta_\Lambda \Big),\qquad f\ge 0.
    \end{align*}
\end{lemma}
\begin{proof}
    Note that, in order to prove the above statement, one can drop \(\chi_\Lambda, \widetilde{\chi}_\Lambda\), as both are the same (measurable) function of \(\underline{\omega}_\Lambda, \zeta_\Lambda\), resp.\ of \(\underline{\widetilde{\omega}}_\Lambda, \widetilde{\zeta}_\Lambda\). This means that we only have to show that
        \begin{align*}
        (\underline{\omega}_\Lambda, \, \zeta_\Lambda )
        \stackrel{\text{d}}{=}
        \big(\underline{\widetilde{\omega}}_{\Lambda}, \, \widetilde{\zeta}_\Lambda \big).
    \end{align*} 
    Now,
    \begin{align*}
        (\underline{\omega}_\Lambda, \, \zeta_\Lambda)
        = \Big(\underline{\omega}_\Lambda^{(\text{retained})} + \underline{\omega}_\Lambda^{(\text{discarded})}, \, \zeta_\Lambda\Big),
    \end{align*} 
    where \(\underline{\omega}_\Lambda^{(\text{retained})}, \underline{\omega}_\Lambda^{(\text{discarded})}, \zeta_\Lambda\) are independent, \(\underline{\omega}_\Lambda^{(\text{retained})}\) is a Poisson point process of intensity \(p_t\) with independent marks coming from an \(\mathrm{Exp}(1)\)-distribution conditioned to be \(> t\), \(\underline{\omega}_\Lambda^{(\text{discarded})}\) is a Poisson point process of intensity \(1 - p_t\) with independent marks coming from an \(\mathrm{Exp}(1)\)-distribution conditioned to be \(\leq t\), and \(\zeta_\Lambda\) is a Poisson point process of intensity \(1\).
    On the other hand,
    \begin{align*}
        \big(\underline{\widetilde{\omega}}_{\Lambda}, \, \widetilde{\zeta}_\Lambda \big)
        = \big(\underline{\widehat{\omega}}_\Lambda + \underline{\widehat{\beta}}_\Lambda, \, \widetilde{\zeta}_\Lambda \big),
    \end{align*} with completely analogous statements for \(\underline{\widehat{\omega}}_\Lambda, \underline{\widehat{\beta}}_\Lambda, \widetilde{\zeta}_\Lambda\).
\end{proof}

Using the above lemma, we can now semi-explicitly compute \(\d (\mu T_t)_\Lambda/\d \nu_\Lambda\). 
\begin{lemma}\label{lemma:local_densities_of_evolved_measure}
    Fix a version \(\widetilde{g}_\Lambda\) of \(\d \mu_\Lambda/\d \pi_\Lambda\).
    Then, the following function \(g_{\Lambda, t}\) 
    is a version of 
    \(\d (\mu T_t)_\Lambda/\d \nu_\Lambda\):
    \begin{align*}
        &g_{\Lambda, t}(\chi_\Lambda) := \frac{\d \pi_\Lambda}{\d \nu_\Lambda}(\chi_\Lambda) \, \widetilde{g}_{\Lambda, t}(\chi_\Lambda),
    \end{align*} 
    where
\begin{equation}\label{eq:gtilde}
\begin{split}
        \widetilde{g}_{\Lambda, t}(\chi_\Lambda)
        &:= \sum_{\zeta_\Lambda \subseteq \chi_\Lambda} 
        \int \pi(\d \beta_\Lambda) \, \widetilde{g}_\Lambda(\chi_\Lambda\setminus\zeta_\Lambda \cup \beta_\Lambda) \int_{(t, \infty)^{\chi_\Lambda \setminus \zeta_\Lambda}}  \d (s_x)_{x \in \chi_\Lambda \setminus \zeta_\Lambda} \,  \e^{- \sum_{x \in \chi_\Lambda \setminus \zeta_\Lambda} s_x} 
        \\
        &\qquad \qquad \int_{[0,t]^{\beta_\Lambda}} \d (s_x)_{x \in \beta_\Lambda} \,  \e^{- \sum_{x \in \beta_\Lambda} s_x}  \widetilde{\psi}_{\Lambda, t}\Big(\zeta_\Lambda \,\Big\vert\, \sum_{x \in \chi_\Lambda\setminus\zeta_\Lambda} \delta_{(x, s_x)} + \sum_{x \in \beta_\Lambda} \delta_{(x, s_x)} \Big),
    \end{split}
    \end{equation}
    and \(\widetilde{\psi}_{\Lambda, t}(\zeta_\Lambda \,\vert\, \omega_\Lambda) = \int \mu_{\Lambda^c \,\vert\, \Lambda}(\d \omega_{\Lambda^c} \,\vert\, \omega_\Lambda) \int M(\d\underline{\omega}_{\Lambda^c} \,\vert\, \omega_{\Lambda^c})  \, \psi_{\Lambda, t}(\zeta_\Lambda \,\vert\, \underline{\omega}_\Lambda \underline{\omega}_{\Lambda^c} )\).
\end{lemma}

Before we carry on to prove this lemma, let us show immediately how helpful it can be.
\begin{corollary}\label{corollary:regular_measures_stay_regular_under_evolution}
    Suppose \(\mu\) is a regular starting measure.
    Then, \(\mu T_t\) is regular too.
\end{corollary}
\begin{proof}
    This follows immediately from the formula for \(\widetilde{g}_{\Lambda, t} = \d (T_t \mu)_\Lambda/\d \pi_\Lambda\) given in \Cref{lemma:local_densities_of_evolved_measure} and the bounds from \Cref{proposition:good_densities_of_the_birth_and_death_process} for \(\psi_{\Lambda, t}\), and hence \(\widetilde{\psi}_{\Lambda, t}\).
\end{proof}

\begin{proof}[Proof of \Cref{lemma:local_densities_of_evolved_measure}]

    Let \(f\) be \(\Lambda\)-measurable and \(M(\d\underline\omega\vert\omega)\) the iid marking of a configuration $\omega$ by Exp(1)-distributed random variables.
    Recall that, by \Cref{proposition:good_densities_of_the_birth_and_death_process}, there exists a measurable map \((t, \underline{\omega}, \zeta_\Lambda) \mapsto \psi_{\Lambda, t}(\zeta_\Lambda \,\vert\, \underline{\omega})\) such that \(\psi_{\Lambda, t}(\cdot \,\vert\, \underline{\omega})\) is a version of the density of \((Y_t\Ssup{\underline{\omega}})_\Lambda\) w.r.t.\ \(\pi_\Lambda\). Then, we can write
    \begin{align*}
        (T_t f)(\omega)
        = \int M(\d\underline{\omega} \,\vert\, \omega) \int \pi_\Lambda(\d\zeta_\Lambda) \, \psi_{\Lambda, t}(\zeta_\Lambda \,\vert\, \underline{\omega} ) \, f(\underline{\omega}^{(t)}_\Lambda + \zeta_\Lambda).
    \end{align*}
    Hence,
    \begin{align*}
        \int \mu_t(\d \omega) \, f(\omega)
       & := \int \mu(\d \omega) (T_t f)(\omega)\\
        &= \int \mu(\d \omega) \int M(\d\underline{\omega} \,\vert\, \omega) \int \pi_\Lambda(\d\zeta_\Lambda) \, \psi_{\Lambda, t}(\zeta_\Lambda \,\vert\, \underline{\omega} ) \, f(\underline{\omega}^{(t)}_\Lambda + \zeta_\Lambda) \\
        &= \int \pi(\d \omega_\Lambda) \widetilde{g}_\Lambda(\omega_\Lambda) \int \mu_{\Lambda^c \,\vert\, \Lambda}(\d \omega_{\Lambda^c} \,\vert\, \omega_\Lambda) \int M(\d\underline{\omega} \,\vert\, \omega) \int \pi_\Lambda(\d\zeta_\Lambda) \, \psi_{\Lambda, t}(\zeta_\Lambda \,\vert\, \underline{\omega} ) \, f(\underline{\omega}^{(t)}_\Lambda + \zeta_\Lambda) \\
        &= \int \pi(\d \omega_\Lambda) \int M(\d\underline{\omega}_\Lambda \,\vert\, \omega_\Lambda) \int \pi_\Lambda(\d\zeta_\Lambda)  \, F(\underline{\omega}_\Lambda, \zeta_\Lambda, \underline{\omega}_{\Lambda}^{(t)} + \zeta_\Lambda),
    \end{align*} with
    \begin{align*}
        F(\underline{\omega}_\Lambda, \zeta_\Lambda, \chi_\Lambda)
        := \widetilde{g}_\Lambda(\omega_\Lambda) \int \mu_{\Lambda^c \,\vert\, \Lambda}(\d \omega_{\Lambda^c} \,\vert\, \omega_\Lambda) \int M(\d\underline{\omega}_{\Lambda^c} \,\vert\, \omega_{\Lambda^c})  \, \psi_{\Lambda, t}(\zeta_\Lambda \,\vert\, \underline{\omega}_\Lambda \underline{\omega}_{\Lambda^c} ) \, f(\chi_\Lambda). 
    \end{align*}
    Now \Cref{le_variable_change_poisson_point_process} implies
    \begin{align*}
        &\int \pi(\d \omega_\Lambda) \int M(\d\underline{\omega}_\Lambda \,\vert\, \omega_\Lambda) \int \pi_\Lambda(\d\zeta_\Lambda)  \, F(\underline{\omega}_\Lambda, \zeta_\Lambda, \underline{\omega}_{\Lambda}^{(t)} + \zeta_\Lambda) \\
        &= \int \pi^{1 + p_t}(\d \chi_\Lambda) \int K_t(\d \underline{\omega}_{\Lambda}, \d \zeta_\Lambda \,\vert\, \chi_\Lambda)  \, F(\underline{\omega}_\Lambda, \zeta_\Lambda, \chi_\Lambda) = \int \nu(\d \chi_\Lambda) \, g_{\Lambda, t}(\chi_\Lambda) f(\chi_\Lambda),
    \end{align*} with
    \begin{align*}
        &g_{\Lambda, t}(\chi_\Lambda):= \e^{-p_t \abs{\Lambda}} (1+p_t)^{N_\Lambda(\chi_\Lambda)} \frac{\d \pi_\Lambda}{\d \nu_\Lambda}(\chi_\Lambda) \int K_t(\d \underline{\omega}_{\Lambda}, \d \zeta_\Lambda \,\vert\, \chi_\Lambda) \, \widetilde{g}_\Lambda(\omega_\Lambda) \widetilde{\psi}_{\Lambda, t}(\zeta_\Lambda \,\vert\, \underline{\omega}_\Lambda ), 
    \end{align*} where
    \begin{align*}
        \widetilde{\psi}_{\Lambda, t}(\zeta_\Lambda \,\vert\, \underline{\omega}_\Lambda ) 
        :=
        \int \mu_{\Lambda^c \,\vert\, \Lambda}(\d \omega_{\Lambda^c} \,\vert\, \omega_\Lambda) \int M(\d\underline{\omega}_{\Lambda^c} \,\vert\, \omega_{\Lambda^c})  \, \psi_{\Lambda, t}(\zeta_\Lambda \,\vert\, \underline{\omega}_\Lambda \underline{\omega}_{\Lambda^c} ).
    \end{align*}
    By the explicit form of the kernel \(K_t\), given in \Cref{le_variable_change_poisson_point_process}, we get
    \begin{align*}
        g_{\Lambda, t}(\chi_\Lambda)
        &= \e^{-p_t \abs{\Lambda}} (1+p_t)^{N_\Lambda(\chi_\Lambda)} \frac{\d \pi_\Lambda}{\d \nu_\Lambda}(\chi_\Lambda) \int K_t(\d \underline{\omega}_{\Lambda}, \d \zeta_\Lambda \,\vert\, \chi_\Lambda) \, \widetilde{g}_\Lambda(\omega_\Lambda) \widetilde{\psi}_{\Lambda, t}(\zeta_\Lambda \,\vert\, \underline{\omega}_\Lambda ) \\
        &= \frac{\d \pi_\Lambda}{\d \nu_\Lambda}(\chi_\Lambda) \sum_{\zeta_\Lambda \subseteq \chi_\Lambda} 
        \int \pi(\d \beta_\Lambda) \, \widetilde{g}_\Lambda(\chi_\Lambda\setminus\zeta_\Lambda \cup \beta_\Lambda) \int_{(t, \infty)^{\chi_\Lambda \setminus \zeta_\Lambda}}  \d (s_x)_{x \in \chi_\Lambda \setminus \zeta_\Lambda} \,  \e^{- \sum_{x \in \chi_\Lambda \setminus \zeta_\Lambda} s_x} 
        \\
        &\qquad\int_{[0,t]^{\beta_\Lambda}} \d (s_x)_{x \in \beta_\Lambda} \,  \e^{- \sum_{x \in \beta_\Lambda} s_x} \widetilde{\psi}_{\Lambda, t}\Big(\zeta_\Lambda \,\bigg\vert\, \sum_{x \in \chi_\Lambda\setminus\zeta_\Lambda} \delta_{(x, s_x)} + \sum_{x \in \beta_\Lambda} \delta_{(x, s_x)} \Big),
    \end{align*}
    as desired. 
\end{proof}

Finally, we compute the time-\(t\)-derivative of \(g_{\Lambda, t} =\d (\mu T_t)_\Lambda/\d \nu_\Lambda\).
\begin{lemma}\label{lemma:derivative_of_g_t}
    The derivative  in \(t > 0\) of \(g_{\Lambda, t}\) from \Cref{lemma:local_densities_of_evolved_measure} is given by
    \begin{align*}
        \frac{\d}{\d t} g_{\Lambda, t}(\chi_\Lambda) 
        &= \frac{\d \pi_\Lambda}{\d \nu_\Lambda}(\chi_\Lambda) \sum_{\zeta_\Lambda \subseteq \chi_\Lambda} 
        \int \pi(\d \beta_\Lambda) \, \widetilde{g}_\Lambda(\chi_\Lambda\setminus\zeta_\Lambda \cup \beta_\Lambda) \Bigg\{ 
        \int_{(t, \infty)^{\chi_\Lambda \setminus \zeta_\Lambda}} \d (s_x)_{x \in \chi_\Lambda \setminus \zeta_\Lambda} 
        \int_{[0, t]^{\beta_\Lambda}} \d (s_x)_{x \in \beta_\Lambda}\\
        &\qquad\qquad 
        \e^{-\sum_{x \in \chi_\Lambda\setminus\zeta_\Lambda} s_x - \sum_{x \in \beta_\Lambda} s_x}  
        \Big(\frac{\d}{\d t} \widetilde{\psi}_{\Lambda, t}\Big) \Big(\zeta_\Lambda \,\Big\vert\, \sum_{x \in \chi_\Lambda\setminus\zeta_\Lambda} \delta_{(x, s_x)} + \sum_{x \in \beta_\Lambda} \delta_{(x, s_x)} \Big)  \\
        &\qquad 
        - \sum_{x \in \chi_\Lambda \setminus \zeta_\Lambda} \int_{(t, \infty)^{(\chi_\Lambda \setminus \zeta_\Lambda)-\delta_{x}}} \d (s_y)_{y \in (\chi_\Lambda \setminus \zeta_\Lambda)-\delta_{x}} \int_{[0, t]^{\beta_\Lambda}} \d (s_y)_{y \in \beta_\Lambda}\\
        &\qquad\qquad
        \e^{-t -\sum_{y \in (\chi_\Lambda\setminus\zeta_\Lambda)-\delta_{x}} s_y - \sum_{y \in \beta_\Lambda} s_y} 
        \widetilde{\psi}_{\Lambda, t}\Big(\zeta_\Lambda \,\Big\vert\, \delta_{(x, t)} + \sum_{y \in (\chi_\Lambda\setminus\zeta_\Lambda)-\delta_{x}} \delta_{(y, s_y)} + \sum_{y \in \beta_\Lambda} \delta_{(y, s_y)} \Big) \\
        &\qquad
        + \sum_{x \in \beta_\Lambda} \int_{(t, \infty)^{\chi_\Lambda\setminus\zeta_\Lambda}} \d (s_y)_{y \in \chi_\Lambda \setminus \zeta_\Lambda} 
        \int_{[0, t]^{\beta_\Lambda-\delta_{x}}} \d (s_y)_{y \in \beta_\Lambda-\delta_{x}}  \\
        &\qquad\qquad  
        \e^{-t -\sum_{y \in \chi_\Lambda\setminus\zeta_\Lambda} s_y - \sum_{y \in \beta_\Lambda-\delta_{x}} s_y}
        \widetilde{\psi}_{\Lambda, t}\Big(\zeta_\Lambda \,\Big\vert\, \delta_{(x, t)} + \sum_{y \in \chi_\Lambda\setminus\zeta_\Lambda} \delta_{(y, s_y)} + \sum_{y \in \beta_\Lambda-\delta_{x}} \delta_{(y, s_y)} \Big)
        \Bigg\}.
    \end{align*} 
\end{lemma}
\begin{proof}
    Recalling that $g_{\Lambda, t}(\chi_\Lambda) := \frac{\d \pi_\Lambda}{\d \nu_\Lambda}(\chi_\Lambda) \, \widetilde{g}_{\Lambda, t}(\chi_\Lambda)$, we can formally calculate the derivative of the integrand in~\eqref{eq:gtilde} for $\widetilde g$, yielding
    \begin{align*}
        &\frac{\d}{\d t} \Big( \prod_{x \in \chi_\Lambda \setminus \zeta_\Lambda} \1_{(t, \infty)}(s_x) \prod_{x \in \beta_\Lambda} \1_{[0, t]}(s_x)  \, \widetilde{\psi}_{\Lambda, t}\Big(\zeta_\Lambda \,\Big\vert\, \sum_{x \in \chi_\Lambda\setminus\zeta_\Lambda} \delta_{(x, s_x)} + \sum_{x \in \beta_\Lambda} \delta_{(x, s_x)} \Big)  \Big) \\
        &= \prod_{x \in \chi_\Lambda \setminus \zeta_\Lambda} \1_{(t, \infty)}(s_x) \prod_{x \in \beta_\Lambda} \1_{[0, t]}(s_x)  \, \Big(\frac{\d}{\d t} \widetilde{\psi}_{\Lambda, t}\Big(\zeta_\Lambda \,\bigg\vert\, \sum_{x \in \chi_\Lambda\setminus\zeta_\Lambda} \delta_{(x, s_x)} + \sum_{x \in \beta_\Lambda} \delta_{(x, s_x)} \Big)  \Big) \\
        &\qquad + \Big( \sum_{x \in \chi_\Lambda \setminus \zeta_\Lambda} \delta_t(s_x)  \prod_{y \in (\chi_\Lambda \setminus \zeta_\Lambda) -\delta_{x}} \1_{(t, \infty)}(s_y) \prod_{y \in \beta_\Lambda} \1_{[0, t]}(s_y) \\
        &\qquad\qquad  - \sum_{x \in \beta_\Lambda} \delta_t(s_x)  \prod_{y \in \chi_\Lambda \setminus \zeta_\Lambda} \1_{(t, \infty)}(s_y) \prod_{y \in \beta_\Lambda-\delta_{x}} \1_{[0, t]}(s_y) \Big) \widetilde{\psi}_{\Lambda, t}\Big(\zeta_\Lambda \,\Big\vert\, \sum_{x \in \chi_\Lambda\setminus\zeta_\Lambda} \delta_{(x, s_x)} + \sum_{x \in \beta_\Lambda} \delta_{(x, s_x)} \Big).
    \end{align*}
    By interchanging limit and integral we arrive at
    \begin{align*}
        &\frac{\d}{\d t} g_{\Lambda, t}(\chi_\Lambda)
        = \frac{\d \pi_\Lambda}{\d \nu_\Lambda}(\chi_\Lambda)\sum_{\zeta_\Lambda \subseteq \chi_\Lambda} 
        \int \pi(\d \beta_\Lambda) \, \widetilde{g}_\Lambda(\chi_\Lambda\setminus\zeta_\Lambda \cup \beta_\Lambda) 
        \int \d (s_x)_{x \in \chi_\Lambda \setminus \zeta_\Lambda} 
        \int \d (s_x)_{x \in \beta_\Lambda}\\
        &
        \e^{-\sum_{x \in \chi_\Lambda\setminus\zeta_\Lambda} s_x - \sum_{x \in \beta_\Lambda} s_x} 
        \Bigg\{ \prod_{x \in \chi_\Lambda \setminus \zeta_\Lambda} \1_{(t, \infty)}(s_x) \prod_{x \in \beta_\Lambda} \1_{[0, t]}(s_x)  \, \Big(\frac{\d}{\d t} \widetilde{\psi}_{\Lambda, t}\Big(\zeta_\Lambda \,\Big\vert\, \sum_{x \in \chi_\Lambda\setminus\zeta_\Lambda} \delta_{(x, s_x)} + \sum_{x \in \beta_\Lambda} \delta_{(x, s_x)} \Big)  \Big) \\
        &\quad
        + \Big( - \sum_{x \in \chi_\Lambda \setminus \zeta_\Lambda} \delta_t(s_x)  \prod_{y \in (\chi_\Lambda \setminus \zeta_\Lambda) -\delta_{x}} \1_{(t, \infty)}(s_y) \prod_{y \in \beta_\Lambda} \1_{[0, t]}(s_y) \\
        &\qquad\quad  + \sum_{x \in \beta_\Lambda} \delta_t(s_x)  \prod_{y \in \chi_\Lambda \setminus \zeta_\Lambda} \1_{(t, \infty)}(s_y) \prod_{y \in \beta_\Lambda-\delta_{x}} \1_{[0, t]}(s_y) \Big) \widetilde{\psi}_{\Lambda, t}\Big(\zeta_\Lambda \,\Big\vert\, \sum_{x \in \chi_\Lambda\setminus\zeta_\Lambda} \delta_{(x, s_x)} + \sum_{x \in \beta_\Lambda} \delta_{(x, s_x)} \Big) \Bigg\} \\
        &=: A_t(\chi_\Lambda) + B_t(\chi_\Lambda)  + C_t(\chi_\Lambda), 
    \end{align*} 
    where
    \begin{align*}
    A_t(\chi_\Lambda) 
    &    := \frac{\d \pi_\Lambda}{\d \nu_\Lambda}(\chi_\Lambda) \sum_{\zeta_\Lambda \subseteq \chi_\Lambda} 
        \int \pi(\d \beta_\Lambda) \, \widetilde{g}_\Lambda(\chi_\Lambda\setminus\zeta_\Lambda \cup \beta_\Lambda) \int_{(t, \infty)^{\chi_\Lambda \setminus \zeta_\Lambda}} \d (s_x)_{x \in \chi_\Lambda \setminus \zeta_\Lambda} 
        \int_{[0, t]^{\beta_\Lambda}} \d (s_x)_{x \in \beta_\Lambda}\\
        &\qquad\e^{-\sum_{x \in \chi_\Lambda\setminus\zeta_\Lambda} s_x - \sum_{x \in \beta_\Lambda} s_x} \Big(\frac{\d}{\d t} \widetilde{\psi}_{\Lambda, t}\Big) \Big(\zeta_\Lambda \,\Big\vert\, \sum_{x \in \chi_\Lambda\setminus\zeta_\Lambda} \delta_{(x, s_x)} + \sum_{x \in \beta_\Lambda} \delta_{(x, s_x)} \Big),
    \end{align*} 
    and
    \begin{align*}
        &B_t(\chi_\Lambda)  
        := - \frac{\d \pi_\Lambda}{\d \nu_\Lambda}(\chi_\Lambda) \sum_{\zeta_\Lambda \subseteq \chi_\Lambda} 
        \int \pi(\d \beta_\Lambda) \, \widetilde{g}_\Lambda(\chi_\Lambda\setminus\zeta_\Lambda \cup \beta_\Lambda) \sum_{x \in \chi_\Lambda \setminus \zeta_\Lambda} \int_{(t, \infty)^{(\chi_\Lambda \setminus \zeta_\Lambda)-\delta_{x}}} \d (s_y)_{y \in (\chi_\Lambda \setminus \zeta_\Lambda)-\delta_{x}}\\ 
        &\quad\int_{[0, t]^{\beta_\Lambda}} \d (s_y)_{y \in \beta_\Lambda} \e^{-t -\sum_{y \in (\chi_\Lambda\setminus\zeta_\Lambda)-\delta_{x}} s_y - \sum_{y \in \beta_\Lambda} s_y}  \widetilde{\psi}_{\Lambda, t}\Big(\zeta_\Lambda \,\Big\vert\, \delta_{(x, t)} + \sum_{y \in (\chi_\Lambda\setminus\zeta_\Lambda)-\delta_{x}} \delta_{(y, s_y)} + \sum_{y \in \beta_\Lambda} \delta_{(y, s_y)} \Big),
    \end{align*} 
    and
    \begin{align*}
        &C_t(\chi_\Lambda) 
        := \frac{\d \pi_\Lambda}{\d \nu_\Lambda}(\chi_\Lambda) \sum_{\zeta_\Lambda \subseteq \chi_\Lambda} 
        \int \pi(\d \beta_\Lambda) \, \widetilde{g}_\Lambda(\chi_\Lambda\setminus\zeta_\Lambda \cup \beta_\Lambda)\sum_{x \in \beta_\Lambda} \int_{(t, \infty)^{\chi_\Lambda\setminus\zeta_\Lambda}} \d (s_y)_{y \in \chi_\Lambda \setminus \zeta_\Lambda}  \\ 
        &
        \int_{[0, t]^{\beta_\Lambda-\delta_{x}}} \d (s_y)_{y \in \beta_\Lambda-\delta_{x}} \e^{-t -\sum_{y \in \chi_\Lambda\setminus\zeta_\Lambda} s_y - \sum_{y \in \beta_\Lambda-\delta_{x}} s_y}  \widetilde{\psi}_{\Lambda, t}\Big(\zeta_\Lambda \,\Big\vert\, \delta_{(x, t)} + \sum_{y \in \chi_\Lambda\setminus\zeta_\Lambda} \delta_{(y, s_y)} + \sum_{y \in \beta_\Lambda-\delta_{x}} \delta_{(y, s_y)} \Big).
    \end{align*}
    These formal calculations can be made rigorous without further difficulties except for some bookkeeping, achieved by looking at the corresponding difference quotients and some routine justifications of limit interchanging, using the fact that there is some joint continuity of \(\widetilde{\psi}_{\Lambda, t}(\zeta_\Lambda \,\vert\, \underline{\omega}_\Lambda )\) in \(t\) and the lifespans of points in \(\underline{\omega}_\Lambda\), which can again be proven very similarly to the technical input for the existence of correlation functions \Cref{lemma:birth_death_process_expected_number_of_points_continuity_in_box_centers} and the continuity in spatial translation of initial conditions in \Cref{le_continuity_wrt_translation_of_initial_conditions}.
    Note that the right-differentiability in \(t\) of \(\psi_{\Lambda, t}\) provided by \Cref{proposition:good_densities_of_the_birth_and_death_process} is here upgraded to differentiability at any \(t > 0\) by averaging over the lifespans in the initial configuration.
\end{proof}

\subsubsection{Proof of \Cref{theorem_derivative_of_the_local_entropy}}\label{sec:proof-theorem-derivative-local-entropy}

Using one last lemma on uniform integrability, the proof of which we postpone until immediately after, we are now in the position to give a proof of \Cref{theorem_derivative_of_the_local_entropy}.
\begin{lemma}\label{lemma:uniform_integrability_estimates}
    If the moments
    \begin{align*}
        \pi_\Lambda\big[N_\Lambda \, \widetilde{g}_\Lambda \log \widetilde{g}_\Lambda \big]\quad\text{and}\quad \pi_\Lambda\big[1/\widetilde{g}_\Lambda\big]
    \end{align*}
    exist, then the three families 
    \begin{equation*}
        \frac{\d }{\d t}\big(\widetilde{g}_{\Lambda, t} \log \widetilde{g}_{\Lambda, t} \big),\quad 
        \Big(\frac{\d }{\d t}\widetilde{g}_{\Lambda, t}\Big) \log \widetilde{g}_\Lambda,\quad
        \frac{\d }{\d t}\widetilde{g}_{\Lambda, t},
    \end{equation*}
    derived from \(\widetilde{g}_{\Lambda, t}\) in \Cref{lemma:local_densities_of_evolved_measure}, are uniformly integrable w.r.t.\ \(\pi_\Lambda\), for \(t \in (0, T)\).
\end{lemma}

As mentioned, we now provide the following proof.
\begin{proof}[Proof of \Cref{theorem_derivative_of_the_local_entropy}]
    Using the fact that $g_{\L,t} = \d(\mu T_t)_\Lambda/\d\nu_\Lambda$ thanks to \Cref{proposition:good_densities_of_the_birth_and_death_process}, and $g_\L = \d\mu_\Lambda/\d\nu_\Lambda$, we have
    \begin{equation*}
    \begin{split}
        \frac{1}{t} \big(\RelEnt_{\Lambda}(\mu T_t \,\vert\, \nu) - \RelEnt_{\Lambda}(\mu \,\vert\, \nu) \big)
        &= \nu_\L\Big[ \, \frac{g_{\Lambda, t} \log g_{\Lambda, t} - g_{\Lambda} \log g_{\Lambda}}{t} \bigg] \\
        &= \nu_\L\Big[ \, \frac{g_{\Lambda, t} - g_{\Lambda}}{t} \log g_{\Lambda} \Big] + \nu_\L\Big[ \, g_{\Lambda, t} \frac{\log g_{\Lambda, t} - \log g_{\Lambda}}{t} \Big] \\
        &= \nu_\L\Big[ \, \frac{g_{\Lambda, t} - g_{\Lambda}}{t} \log g_{\Lambda} \Big] + \pi_\L\Big[\widetilde{g}_{\Lambda, t} \frac{\log \widetilde{g}_{\Lambda, t} - \log \widetilde{g}_{\Lambda}}{t} \Big],
    \end{split}
    \end{equation*}
    since $\d\nu_\L\, g_{\Lambda, t} = \d\pi_\L\, \widetilde{g}_{\Lambda, t}$.
    For the first summand, we have, by \Cref{lem:log_generator}, that
    \begin{align*}
        \nu_\L\Big[  \frac{g_{\Lambda, t} - g_{\Lambda}}{t} \log g_{\Lambda} \Big]
        = \frac{\mu_t[\log g_{\Lambda}] - \mu[\log g_{\Lambda}]}{t}
        \xrightarrow[t \downarrow 0]{}
        \mu\big[\mathscr{L} \log g_{\Lambda}\big].
    \end{align*}
    For the second summand,
    using the pointwise convergences
    \begin{equation*}
       \frac{\log \widetilde{g}_{\Lambda, t} - \log \widetilde{g}_{\Lambda}}{\widetilde{g}_{\Lambda, t} - \widetilde{g}_{\Lambda}} \xrightarrow[t \downarrow 0]{} \frac{1}{\widetilde{g}_{\Lambda}} \quad \text{ and }\quad 
       \widetilde{g}_{\Lambda, t} \xrightarrow[t \downarrow 0]{}
       \widetilde{g}_\Lambda,
    \end{equation*}
    together with the fact that
    \begin{align*}
        &\Big\vert \frac{\widetilde{g}_{\Lambda, t} - \widetilde{g}_{\Lambda}}{t} \Big( \frac{\log \widetilde{g}_{\Lambda, t} - \log \widetilde{g}_{\Lambda}}{\widetilde{g}_{\Lambda, t} - \widetilde{g}_{\Lambda}} \widetilde{g}_{\Lambda, t} - 1 \Big)  \Big\vert
        \leq \Big\vert \frac{\widetilde{g}_{\Lambda, t} \log \widetilde{g}_{\Lambda, t} - \widetilde{g}_{\Lambda} \log \widetilde{g}_{\Lambda}}{t} \Big\vert + \Big\vert \frac{\widetilde{g}_{\Lambda, t} - \widetilde{g}_{\Lambda}}{t} \log \widetilde{g}_{\Lambda} \Big\vert + \bigg\vert \frac{\widetilde{g}_{\Lambda, t} - \widetilde{g}_{\Lambda}}{t} \Big\vert,
    \end{align*}
    the mean-value theorem and the uniform integrability provided by \Cref{lemma:uniform_integrability_estimates} yield
    \begin{align*}
        \pi\Big[ \, \widetilde{g}_{\Lambda, t} \frac{\log \widetilde{g}_{\Lambda, t} - \log \widetilde{g}_{\Lambda}}{t} \Big]
       & = \pi\Big[ \, \widetilde{g}_{\Lambda, t} \frac{\log \widetilde{g}_{\Lambda, t} - \log \widetilde{g}_{\Lambda}}{t} \Big] + \pi\Big[ \frac{\widetilde{g}_{\Lambda, t} - \widetilde{g}_{\Lambda}}{t} \Big] \\
        &= \pi\Big[ \frac{\widetilde{g}_{\Lambda, t} - \widetilde{g}_{\Lambda}}{t} \Big( \frac{\log \widetilde{g}_{\Lambda, t} - \log \widetilde{g}_{\Lambda}}{\widetilde{g}_{\Lambda, t} - \widetilde{g}_{\Lambda}} \widetilde{g}_{\Lambda, t} - 1 \Big) \Big]
        \xrightarrow[t \downarrow 0]{} 0,
    \end{align*} 
    as desired. 
\end{proof}

Eventually we provide the postponed proof.
\begin{proof}[Proof of \Cref{lemma:uniform_integrability_estimates}]
    We have \(\frac{\d }{\d t}\big(\widetilde{g}_{\Lambda, t} \log \widetilde{g}_{\Lambda, t} \big) = \big(\frac{\d}{\d t} \widetilde{g}_{\Lambda, t}) \big(\log \widetilde{g}_{\Lambda, t}  + 1\big)\), so we have to show uniform integrability of the families 
    \begin{equation*}
       \Big(\frac{\d}{\d t} \widetilde{g}_{\Lambda, t}\Big) \log \widetilde{g}_{\Lambda, t},\quad \Big(\frac{\d }{\d t}\widetilde{g}_{\Lambda, t}\Big) \log \widetilde{g}_\Lambda, \quad \text{ and }\quad \frac{\d }{\d t}\widetilde{g}_{\Lambda, t}. 
    \end{equation*} 
    We do this in three separate steps. 
    \medbreak
    
    \noindent
    \textit{1. Reduction}: 
    Since, by \Cref{{proposition:good_densities_of_the_birth_and_death_process}}, for small \(t\), \(\psi_{\Lambda, t}(\emptyset\,\vert\, \cdot) \sim 1\), for some fixed \(c_1 \in (0, 1)\), we have that \(\widetilde{g}_{\Lambda, t} \geq c_1 \widetilde{g}_{\Lambda}\) from the definition of \(\widetilde{g}_{\Lambda, t}\) in \Cref{lemma:local_densities_of_evolved_measure} and hence
    \begin{align*}
        - \log_{-}(\widetilde{g}_{\Lambda, t})
        \geq \log(c_1) - \log_{-}(\widetilde{g}_\Lambda).
    \end{align*}
    Using the same argument, we have for a \(c_2 \geq 1\), that
    \begin{align*}
        \log_{+}(\widetilde{g}_{\Lambda})
        \leq \log_{+}(\widetilde{g}_{\Lambda, t}) + \log(c_2).
    \end{align*}
    It follows that
    \begin{align*}
        \abs{\Big(\frac{\d}{\d t} \widetilde{g}_{\Lambda, t}\Big)\log \widetilde{g}_{\Lambda, t}} 
        \leq \abs{\frac{\d}{\d t} \widetilde{g}_{\Lambda, t}} \Big\{ \log_{-}(\widetilde{g}_{\Lambda}) +  \log_{+}(\widetilde{g}_{\Lambda, t}) + \log(1/c_1) \Big\}
    \end{align*} and
    \begin{align*}
        \abs{\Big(\frac{\d}{\d t} \widetilde{g}_{\Lambda, t}\Big)\log \widetilde{g}_{\Lambda}} 
        \leq \abs{\frac{\d}{\d t} \widetilde{g}_{\Lambda, t}} \Big\{ \log_{-}(\widetilde{g}_{\Lambda}) +  \log_{+}(\widetilde{g}_{\Lambda, t}) + \log(c_2) \Big\}.
    \end{align*}
    This means we only have to show uniform integrability of the families \(\frac{\d}{\d t} \widetilde{g}_{\Lambda, t}\), \(\abs{\frac{\d}{\d t} \widetilde{g}_{\Lambda, t}} \log_{-}(\widetilde{g}_{\Lambda})\) and \(\abs{\frac{\d}{\d t} \widetilde{g}_{\Lambda, t}} \log_{+}(\widetilde{g}_{\Lambda, t})\). We will not explicitly show the uniform integrability of  \(\frac{\d}{\d t} \widetilde{g}_{\Lambda, t}\), as it can be shown, in a very similar but easier way to the uniform integrability of \(\abs{\frac{\d}{\d t} \widetilde{g}_{\Lambda, t}} \log_{+}(\widetilde{g}_{\Lambda, t})\), that even \(\abs{\frac{\d}{\d t} \widetilde{g}_{\Lambda, t}} \log(\abs{\frac{\d}{\d t} \widetilde{g}_{\Lambda, t}})\) is uniformly integrable.
    \medbreak
    
    \noindent 
    \textit{2. Uniform integrability of \(\abs{\frac{\d}{\d t} \widetilde{g}_{\Lambda, t}} \log_{-}(\widetilde{g}_{\Lambda})\)}: 
    By the Fenchel--Young inequality, we have
    \begin{align*}
        \abs{\frac{\d}{\d t} \widetilde{g}_{\Lambda, t}} \log_{-}(\widetilde{g}_{\Lambda})
        \leq f\bigg(\abs{\frac{\d}{\d t} \widetilde{g}_{\Lambda, t}}\bigg) + f^*\big(\log_{-}(\widetilde{g}_{\Lambda})\big),
    \end{align*} 
    pointwise, for any convex function $f$ and its convex conjugate $f^*$. Choosing $f(x) = x \log(x) - x$ so that \(f^*(x) = \e^x\) and \(f(x) \leq x \log x\), we can use the uniform integrability of \(\abs{\frac{\d}{\d t} \widetilde{g}_{\Lambda, t}} \log(\abs{\frac{\d}{\d t} \widetilde{g}_{\Lambda, t}})\) mentioned above and the moment assumptions for \(f^*\big(\log_{-}(\widetilde{g}_{\Lambda})\big) \leq 1/\widetilde{g}_{\Lambda}\).
    \medbreak 

    \noindent
    \textit{3. Uniform integrability of \(\abs{\frac{\d}{\d t} \widetilde{g}_{\Lambda, t}}\log_{+}(\widetilde{g}_{\Lambda, t})\)}: 
    Remember the decomposition
    \begin{align*}
        \frac{\d}{\d t} \widetilde{g}_{\Lambda, t}
        = A_t + B_t + C_t
    \end{align*} from the proof of \Cref{lemma:derivative_of_g_t}.
    We will only show the uniform integrability of \(\abs{A_t} \log_{+}(\widetilde{g}_{\Lambda, t})\), because the uniform integrability of \(\abs{B_t} \log_{+}(\widetilde{g}_{\Lambda, t})\) and  \(\abs{C_t} \log_{+}(\widetilde{g}_{\Lambda, t})\) follows similarly but slightly easier.
    Note that if we can upper bound
    \begin{align*}
        \abs{A_t} \log(\widetilde{g}_{\Lambda, t})
        \leq x \log x + y
    \end{align*} for some \(x,y > 0\), then
    \begin{align*}
        \abs{A_t} \log_+(\widetilde{g}_{\Lambda, t})
        \leq x \log x + 1/\e + y,
    \end{align*} so we can ignore the positive part on the logarithm if this is the case.
    We have, by the bounds in \Cref{proposition:good_densities_of_the_birth_and_death_process},
    \begin{align*}
        &\abs{A_t(\chi_\Lambda)}
        \leq 
        \sum_{\zeta_\Lambda \subseteq \chi_\Lambda} 
        \int \pi(\d \beta_\Lambda) \, \widetilde{g}_\Lambda(\chi_\Lambda\setminus\zeta_\Lambda \cup \beta_\Lambda) \int_{(t, \infty)^{\chi_\Lambda \setminus \zeta_\Lambda}} \d (s_x)_{x \in \chi_\Lambda \setminus \zeta_\Lambda} \int_{[0, t]^{\beta_\Lambda}} \d (s_x)_{x \in \beta_\Lambda}  
        \\&\qquad\qquad\qquad\qquad
        \e^{-\sum_{x \in \chi_\Lambda\setminus\zeta_\Lambda} s_x - \sum_{x \in \beta_\Lambda} s_x}
        \Big\lvert \Big(\frac{\d}{\d t} \widetilde{\psi}_{\Lambda, t}\Big) \Big(\zeta_\Lambda \,\big\vert\, \sum_{x \in \chi_\Lambda\setminus\zeta_\Lambda} \delta_{(x, s_x)} + \sum_{x \in \beta_\Lambda} \delta_{(x, s_x)} \Big)\Big\rvert \\ 
        &\lesssim  \sum_{\zeta_\Lambda \subseteq \chi_\Lambda} \int \pi(\d \beta_\Lambda) \, \Big\{ \widetilde{g}_\Lambda(\chi_\Lambda\setminus\zeta_\Lambda \cup \beta_\Lambda) p_t^{\abs{\chi_\Lambda\setminus\zeta_\Lambda}} (1-p_t)^{\abs{\beta_\Lambda}}
        \big( \abs{\zeta_\Lambda} z_{\max}^{\abs{\zeta_\Lambda}} (1-p_t)^{\abs{\zeta_\Lambda}-1}\1_{\zeta_\Lambda \neq \emptyset}  + \1_{\zeta_\Lambda = \emptyset} \big) \Big\} \\
        &\lesssim \frac{1}{1-p_t} \sum_{\zeta_\Lambda \subseteq \chi_\Lambda} p_t^{\abs{\chi_\Lambda\setminus\zeta_\Lambda}} (1-p_t)^{\abs{\zeta_\Lambda}} \big(\abs{\zeta_\Lambda} + (1-p_t)\1_{\zeta_\Lambda =\emptyset}\big)
        \int \pi^{1-p_t}(\d \beta_\Lambda) \, \Big\{ \widetilde{g}_\Lambda(\chi_\Lambda\setminus\zeta_\Lambda \cup \beta_\Lambda)   z_{\max}^{\abs{\zeta_\Lambda}} \Big\}.
    \end{align*}
    We also have
    \begin{align*}
        \log(\widetilde{g}_{\Lambda, t}(\chi_\Lambda))
        &\leq \log(\mathrm{const}) + \log\Big(\sum_{\zeta_\Lambda \subseteq \chi_\Lambda} p_t^{\abs{\chi_\Lambda\setminus\zeta_\Lambda}} (1-p_t)^{\abs{\zeta_\Lambda}} \int \pi^{1-p_t}(\d \beta_\Lambda) \, \Big\{ \widetilde{g}_\Lambda(\chi_\Lambda\setminus\zeta_\Lambda \cup \beta_\Lambda)   z_{\max}^{\abs{\zeta_\Lambda}} \Big\}\Big) \\
        &\leq \log(\mathrm{const}) + \log\big(G_t(\chi_\Lambda)\big),
    \end{align*} 
    with
    \begin{align*}
        G_t(\chi_\Lambda)
        &:= \frac{1}{(1-p_t)\abs{\chi_\Lambda} + (1-p_t) p_t^{\abs{\chi_\Lambda}}} \sum_{\zeta_\Lambda \subseteq \chi_\Lambda} p_t^{\abs{\chi_\Lambda\setminus\zeta_\Lambda}} (1-p_t)^{\abs{\zeta_\Lambda}} \big(\abs{\zeta_\Lambda} + (1-p_t)\1_{\zeta_\Lambda =\emptyset}\big)\\
        &\qquad \int \pi^{1-p_t}(\d \beta_\Lambda) \, \Big\{ \widetilde{g}_\Lambda(\chi_\Lambda\setminus\zeta_\Lambda \cup \beta_\Lambda)   z_{\max}^{\abs{\zeta_\Lambda}} (\abs{\chi_\Lambda} + 1) \Big\}.
    \end{align*} 
    Hence, $\abs{A_t(\chi_\Lambda)}\log(\widetilde{g}_{\Lambda, t}(\chi_\Lambda))
        \leq  G_t(\chi_\Lambda) \big(\log(\mathrm{const}) + \log \big(G_t(\chi_\Lambda)\big)  \big)$. 
    We only show the uniform integrability of \(G_t \log (G_t)\), as that of \(G_t\) is easier to show in a similar way.
    Denote
    \begin{align*}
        C_t(\chi_\Lambda)
        := \frac{1}{(1-p_t)\abs{\chi_\Lambda} + (1-p_t) p_t^{\abs{\chi_\Lambda}}}.
    \end{align*}
    Now consider \(\phi(x) = x \log x\) and an increasing convex function \(\Phi\) such that 
    \begin{itemize}
        \item \(\pi_\Lambda[\Phi( (\abs{\omega}+1) \phi(\widetilde{g}_\Lambda)) ] < +\infty\), 
        \item \((\Phi\circ \phi)(x) \lesssim x^2\),
        \item \(\Phi(x)/x \to \infty\) as \(x\to\infty\)
        \item \((\Phi\circ \phi)(xy) \lesssim  (\Phi\circ \phi)(x) (\Phi\circ \phi)(y) \) for large enough \(x, y\).
    \end{itemize}
    Such a function exists by \Cref{lemma:special_de_la_vallee_poussin}. In applying the lemma, the random variable \(X\) is \(N_\Lambda \phi(\widetilde{g}_\Lambda)\), the function \(f\) is \(\Phi\) and the function \(g\) is \(g(x) = (\phi^{-1}(x))^2\).
    Then, by convexity and the above properties of \(\Phi\),
    \begin{align*}
        &\Phi\big( G_t(\chi_\Lambda)  \log G_t(\chi_\Lambda)  \big) 
        = (\Phi \circ \phi)\big( G_t(\chi_\Lambda) \big) \\
        &\leq 
        \sum_{\zeta_\Lambda \subseteq \chi_\Lambda} p_t^{\abs{\chi_\Lambda\setminus\zeta_\Lambda}} (1-p_t)^{\abs{\zeta_\Lambda}} \tfrac{\abs{\zeta_\Lambda} \lor (1-p_t)}{C_t(\chi_\Lambda)^{-1}}
        \int \pi^{1-p_t}(\d \beta_\Lambda) \, \big\{ (\Phi \circ \phi)\big( \widetilde{g}_\Lambda(\chi_\Lambda\setminus\zeta_\Lambda \cup \beta_\Lambda)  z_{\max}^{\abs{\zeta_\Lambda}} (\abs{\chi_\Lambda} + 1) \big) \big\} \\
        &\lesssim
        \sum_{\zeta_\Lambda \subseteq \chi_\Lambda} p_t^{\abs{\chi_\Lambda\setminus\zeta_\Lambda}} (1-p_t)^{\abs{\zeta_\Lambda}} \tfrac{\abs{\zeta_\Lambda} \lor (1-p_t)}{C_t(\chi_\Lambda)^{-1}}  
        \int \pi^{1-p_t}(\d \beta_\Lambda) \, \big\{ (\Phi \circ \phi)\big( \widetilde{g}_\Lambda(\chi_\Lambda\setminus\zeta_\Lambda \cup \beta_\Lambda)\big) (\Phi \circ \phi)\big(z_{\max}^{\abs{\zeta_\Lambda}} (\abs{\chi_\Lambda} + 1) \big) \big\}.
    \end{align*}
    We also have \(C_t(\chi_\Lambda) (\Phi \circ \phi)\big( \abs{\chi_\Lambda} + 1 \big) \lesssim (\abs{\chi_\Lambda} + 1)/(1-p_t) \), since \((\Phi \circ \phi)(x) \lesssim x^2\).
    It follows that
    \begin{align*}
        &\Phi\big( G_t(\chi_\Lambda)  \log G_t(\chi_\Lambda)  \big) 
        = (\Phi \circ \phi)\big( G_t(\chi_\Lambda) \big) \\
        &\lesssim 
        \sum_{\zeta_\Lambda \subseteq \chi_\Lambda} p_t^{\abs{\chi_\Lambda\setminus\zeta_\Lambda}} (1-p_t)^{\abs{\zeta_\Lambda}}
        \int \pi^{1-p_t}(\d \beta_\Lambda) \, \Big\{ (\Phi \circ \phi)\Big( \widetilde{g}_\Lambda(\chi_\Lambda\setminus\zeta_\Lambda \cup \beta_\Lambda)\Big) \tfrac{\abs{\zeta_\Lambda} \lor (1-p_t)}{1-p_t} z_{\max}^{2\abs{\zeta_\Lambda}} (\abs{\chi_\Lambda} + 1) \Big) \Big\}.
    \end{align*}
    Taking expectations,
    \begin{align*}
        \pi_\Lambda\Big[ \Phi\big( G_t\log G_t \big)  \Big]
        &\lesssim \int \pi_\Lambda(\d \omega_\Lambda) \, (\Phi \circ \phi)\big( \widetilde{g}_\Lambda(\omega_\Lambda)\big) \Big\{ \int \pi^{1-p_t}(\d\zeta_\Lambda) \tfrac{\abs{\zeta_\Lambda} \lor (1-p_t)}{1-p_t} z^{2\abs{\zeta_\Lambda}} (\abs{\omega_\Lambda} + \abs{\zeta_\Lambda} + 1) \Big\} \\
        &\lesssim \int \pi_\Lambda(\d \omega_\Lambda) \,  (\abs{\omega_\Lambda}  + 1)\, (\Phi \circ \phi)\big( \widetilde{g}_\Lambda(\omega_\Lambda)\big) \\
        &\leq \int \pi_\Lambda(\d \omega_\Lambda) \, (\Phi \circ \phi)\big( (\abs{\omega_\Lambda}  + 1)\widetilde{g}_\Lambda(\omega_\Lambda)\big)
        < \infty.
    \end{align*}
    Putting everything together, we have that \(\abs{A_t} \log_{+}(\widetilde{g}_{\Lambda, t})\) is uniformly integrable.
\end{proof}

\begin{remark}
    By using similar techniques as in the above proofs, we can show that, if \(\mu = \delta_\emptyset\), then \(\frac{\d}{\d t}\big\vert_{t = 0} \RelEnt_{\Lambda}(\mu T_t \,\vert\, \nu) = -\infty\).
\end{remark}

\subsection{Pointwise and averaged thermodynamic limit of entropy dissipation}\label{section:entropy_dissipation_thermodynamic_limit}

In order to prove \Cref{theorem:main_theorem_HANDA} we first need to establish the pointwise identity and estimates on the error term as stated in \Cref{lemma:rest_term}
The following observation is the key ingredient for the computation in the proof of \Cref{lemma:rest_term}.

\begin{lemma}\label{lemma:existence_of_nice_local_papangelou_intensities}
    Let \(\mu\in \Pcal_\theta\) with finite first moment and such that \(\mu_\Lambda\) and \(\pi_\Lambda\) are equivalent for any \(\Lambda \Subset \mathbb{R}^d\).
    Then, for any \(\Lambda \Subset \mathbb{R}^d\), there is a strictly positive measurable function \((x, \eta) \mapsto \widetilde{\varrho}_\Lambda(x, \eta)\) such that, for every version of \(\d\mu_\Lambda/\d\pi_\Lambda\),
    \begin{align*}
        \widetilde{\varrho}_\Lambda(x, \theta_{x}\eta) = \frac{\frac{\d\mu_\Lambda} {\d\pi_\Lambda}(\eta +\delta_{x})}{\frac{\d\mu_\Lambda} {\d\pi_\Lambda}(\eta)}
    \end{align*} 
    for \((\d x \otimes \mu)\)-a.e \((x, \eta)\) and
    \begin{align*}
        \sum_{x \in \eta} \log \widetilde{\varrho}_\Lambda(x, \theta_{x}(\eta -\delta_{x})) = \sum_{x \in \eta} \log \frac{\frac{\d\mu_\Lambda}{\d\pi_\Lambda}(\eta)}{\frac{\d\mu_\Lambda}{\d\pi_\Lambda}(\eta -\delta_{x})}
    \end{align*} for \(\mu\)-a.e.\ \(\eta\)
    and such that
    \begin{align*}
        \widetilde{\varrho}_\Lambda(x, \cdot) = \frac{\d(\mu_0^!)_{\Lambda -x}}{\d \mu_{\Lambda -x}}\qquad \text{ and }\qquad
        \frac{1}{\widetilde{\varrho}_\Lambda(x, \cdot)} = \frac{\d \mu_{\Lambda -x}}{\d(\mu_0^!)_{\Lambda -x}}
    \end{align*} 
    (as in ''is a version of'') for \(\mathrm{d}x\)-a.e.\ \(x\).
\end{lemma}
\begin{proof}
    Let us first acknowledge that it is at least intuitively plausible from the raison d'être of Palm measure, noting that \(0 \in \Lambda - x\) for \(x \in \Lambda\), we have
    \begin{align*}
        \frac{\d(\mu_0^!)_{\Lambda -x}}{\d \mu_{\Lambda -x}}(\eta)
        &= \frac{\mu_{\Lambda -x}(\eta +\delta_{0})}{\mu_{\Lambda -x}(\eta)}
        = \frac{\mu_{\Lambda -x}(\eta +\delta_{0})}{\mu_{\Lambda -x}(\eta)} \frac{\pi_{\Lambda -x}(\eta)}{\pi_{\Lambda -x}(\eta  +\delta_{0} )}
        = \frac{\widetilde{g}_{\Lambda -x}(\eta +\delta_{0})}{\widetilde{g}_{\Lambda -x}(\eta)}
        = \frac{\widetilde{g}_\Lambda(\theta_{-x}\eta +\delta_{x})}{\widetilde{g}_{\Lambda}(\theta_{-x}\eta)}.
    \end{align*}

    Now fix a strictly positive version \(\widetilde{g}_\Lambda\) of \(\d\mu_\Lambda/\d \pi_\Lambda\) and define
    \begin{align*}
         \widetilde{\varrho}_\Lambda(x, \eta)
         := \frac{\widetilde{g}_\Lambda(\theta_{-x}\eta +\delta_{x})}{\widetilde{g}_\Lambda(\theta_{-x}\eta)} > 0.
    \end{align*}

    We have to show that, for \(\d x\)-almost all \(x\), for every non-negative \(\mathcal{F}_{\Lambda-x}\)-measurable function \(f\), 
    \begin{align*}
        (\mu_0^!)_{\Lambda - x}\big[f \big]
        = \mu_{\Lambda-x}\Big[f  \frac{\widetilde{g}_\Lambda(\theta_{-x}\cdot +\delta_{x})}{\widetilde{g}_\Lambda(\theta_{-x}\cdot)} \Big].
    \end{align*}
    By general measure theoretic arguments, it suffices to check this for countably many \(f\) of the form \(f = \1_{N_{[a_1, b_1] \times \cdots \times [a_d, b_d]} = n}\).
    Furthermore, by monotone convergence, using that \(\mu_{\Lambda-x}\) and \((\mu_0^!)_{\Lambda -x}\) both cannot put mass on \(\partial (\Lambda-x)\), we can even restrict ourselves to those \(f\) with \(\supp f \subseteq \{y \in \Lambda \,\vert\, \dist(y, \partial(\Lambda - x)) > \epsilon \}\) for some \(\epsilon = \epsilon(f) > 0\).
    Hence, it is also enough to show that for all \(f\) in this collection, the equality holds for \(\mathrm{d}x\)-a.a.\ \(x\) such that the support of \(f\) lies strictly inside \(\Lambda - x\).

    Now, fix one such \(f\) such that the support of \(f\) lies strictly inside \(\Lambda - x\) and therefore also inside \(\Lambda - z\) for \(z\) from some small ball around \(x\).
    Then, by definition of the Palm measure,
    \begin{equation*}
    \begin{split}
        (\mu_0^!)_{\Lambda - x}\big[f \big]
        = \mu\Big[\sum_{y \in \eta} f(\theta_y (\eta -\delta_{y})) \varphi(y) \Big]
    \end{split} 
    \end{equation*} for every non-negative function \(\varphi\) that integrates to \(1\).
    Now fix some non-negative radial test function \(\varphi\) with support in a \(1\)-ball around \(0\) that integrates to \(1\).
    Define \(\varphi_\epsilon = \epsilon^{-d} \varphi(\epsilon^{-d} \cdot)\) for \(\epsilon > 0\). Then, the \((\varphi_\epsilon)_{\epsilon > 0}\) still integrate to \(1\), \(\varphi_\epsilon\) has support inside an \(\epsilon\)-ball around \(0\) and \(\varphi_\epsilon \to\delta_0\), as $\epsilon \to 0$.
    
    Note that the support of \(f\), lying strictly inside \(\Lambda - x\), implies that the random variable
    \begin{align*}
        \sum_{y \in \eta} f(\theta_y (\eta -\delta_{y})) \varphi_\epsilon(y)
    \end{align*} is \((\Lambda - x)\)-measurable for small enough \(\epsilon\).
    Hence, by the Mecke formula,
    \begin{equation*}
    \begin{split}
        (\mu_0^!)_{\Lambda - x}\big[f \big]
        &= \mu\Big[\sum_{y \in \eta} f(\theta_y (\eta -\delta_{y})) \varphi_\epsilon(y) \Big]\\
        &= \pi\Big[\widetilde{g}_{\Lambda - x}(\eta) \sum_{y \in \eta} f(\theta_y (\eta -\delta_{y})) \varphi_\epsilon(y) \Big] \\
        &= \pi\Big[\widetilde{g}_{\Lambda}(\theta_{-x}\eta) \sum_{y \in \eta} f(\theta_y (\eta -\delta_{y})) \varphi_\epsilon(y) \Big] \\
        &= \int \, \d y \, \varphi_\epsilon(y) \int \pi(\d\eta) \,  \widetilde{g}_\Lambda(\theta_{-x}\eta + \delta_{x+y}) \, f(\theta_y \eta) \\
        &= \int \, \d y \, \varphi_\epsilon(y) \int \pi(\d\eta) \,  \widetilde{g}_\Lambda(\theta_{-(x+y)}\eta + \delta_{x+y}) \, f(\eta) \\
        &= \int \, \d z \, \varphi_\epsilon(z-x) \int \pi(\d\eta) \,  \widetilde{g}_\Lambda(\theta_{-z}\eta + \delta_z) \, f(\eta).
    \end{split}
    \end{equation*}
    
Next, by the finiteness of the left-hand side, we see that
    \begin{align*}
        z \mapsto \int \pi(\d\eta) \,\widetilde{g}_\Lambda(\theta_{-z}\eta + \delta_z) \, f(\eta) 
    \end{align*} 
    is locally integrable (w.r.t.\ Lebesgue measure).
    By Lebesgue's differentiation theorem, we get
    \begin{align*}
        (\mu_0^!)_{\Lambda - x}\big[f \big]
        &= \lim_{\epsilon \downarrow 0} \int \, \d z \, \varphi_\epsilon(z-x) \int \pi(\d\eta) \,  \widetilde{g}_\Lambda(\theta_{-z}\eta + \delta_z) \, f(\eta)  \\
        &= \int \pi(\d\eta) \, f(\eta) \,\widetilde{g}_\Lambda(\theta_{-x}\eta +\delta_{x}) 
        = \mu_{\Lambda-x}\Big[f  \frac{\widetilde{g}_\Lambda(\theta_{-x}\cdot +\delta_{x})}{\widetilde{g}_\Lambda(\theta_{-x}\cdot)} \Big]
    \end{align*} for \(\d x\)-almost all \(x\).
\end{proof}

\begin{proof}[Proof of \Cref{lemma:rest_term}]\label{proof:proof_of_lemma_rest_term}
    For any $P\in \Pcal(\Omega)$, denote
    \begin{align*}
        b\Ssup{P}_\Lambda(x, \eta)
        := \int P(\d \zeta_{\Lambda^c} \vert \eta_\Lambda) \, b(x, \eta_\Lambda \zeta_{\Lambda^c}).
    \end{align*}
    Set \(g_\Lambda := \d \mu_\Lambda/\d \nu_\Lambda\) and, with \(\widetilde{\varrho}_\Lambda\) as in \Cref{lemma:existence_of_nice_local_papangelou_intensities},
    \begin{align*}
        \varrho_{\Lambda}(x, \cdot) := \frac{\widetilde{\varrho}_\Lambda(x, \cdot)}{b\Ssup{\mu}_\Lambda(x, \theta_{-x} \cdot)}.
    \end{align*}  
    Using that the Papangelou intensity of \(\nu\) is given by \(b\) (GNZ equations), \Cref{lemma:existence_of_nice_local_papangelou_intensities} implies
    \begin{align*}
             \frac{g_\Lambda(\eta)}{g_\Lambda(\eta +\delta_{x})}
             = \frac{b\Ssup{\nu}_\Lambda(x, \eta)}{\widetilde{\varrho}_\Lambda(x, \theta_{x}\eta)}
             =  \frac{b\Ssup{\nu}_\Lambda(x, \eta)}{b\Ssup{\mu}_\Lambda(x, \eta)} \frac{1}{\varrho_\Lambda(x, \theta_x \eta)}
    \end{align*} for \((\d x \otimes \mu)\)-a.e.\ \((x, \eta)\) and
    \begin{align*}
        \sum_{x \in \eta} \log \frac{g_\Lambda(\eta)}{g_\Lambda(\eta -\delta_{x})}
        = \sum_{x \in \eta} \log \frac{\widetilde{\varrho}_\Lambda(x, \theta_{x}(\eta -\delta_{x}))}{b\Ssup{\nu}(x, \eta-\delta_{x}) }
        = \sum_{x \in \eta} \log \Big( \frac{b\Ssup{\mu}(x, \eta-\delta_{x})} {b\Ssup{\nu}(x, \eta-\delta_{x})}\varrho_\Lambda(x, \theta_x (\eta -\delta_{x})) \Big)
    \end{align*} 
    for \(\mu\)-a.e.\ \(\eta\)
    and such that
    \begin{align*}
        \frac{\d(\mu_0^!)_{\Lambda -x}}{\d (b(0, \cdot) \mu)_{\Lambda -x}}
        = \frac{\widetilde{\varrho}_\Lambda(x, \cdot)}{b\Ssup{\mu}_\Lambda(x, \theta_{-x} \cdot)}
        = \varrho_\Lambda(x, \cdot)
    \end{align*} and
    \begin{align*}
        \frac{\d (b(0, \cdot) \mu)_{\Lambda -x}}{\d(\mu_0^!)_{\Lambda -x}}
        = \frac{b\Ssup{\mu}_\Lambda(x, \theta_{-x} \cdot)}{\widetilde{\varrho}_\Lambda(x, \cdot)}
        = \frac{1}{\varrho_\Lambda(x, \cdot)}
    \end{align*} (as in ``is a version of'') for \(\mathrm{d}x\)-a.e.\ \(x\).
    Hence,
    \begin{align*}
         &\Big\vert (-\mathscr{L}) \log\frac{\d \mu_\Lambda}{\d \nu_\Lambda} (\eta)
        - \Big( \int_{\Lambda} \, b(0, \theta_x \eta) \, \log \frac{1}{\varrho_{\Lambda}(x, \theta_x \eta)}
        + \sum_{x \in \eta_\Lambda} \log \varrho_\Lambda(x, \theta_x(\eta -\delta_{x})) \Big) \Big\vert\\
        &\quad\leq \int_{\Lambda} \, b(x, \eta) \Big\vert\log \frac{b\Ssup{\nu}_\Lambda(x, \eta)}{b\Ssup{\mu}_\Lambda(x, \eta)} \Big\vert \, \d x + \sum_{x \in \eta_{\Lambda}} \Big\vert\log \frac{b\Ssup{\mu}_\Lambda(x, \eta-\delta_{x})}{b\Ssup{\nu}_\Lambda(x, \eta-\delta_{x})} \Big\vert.
    \end{align*} 
    Note that, due to the finite-range property and boundedness from above and away from zero of \(b\), 
    \begin{align*}
        &\mu\Big[\int_{\Lambda} \, b(x, \eta) \Big\vert\log \frac{b\Ssup{\nu}_\Lambda(x, \eta)}{b\Ssup{\mu}_\Lambda(x, \eta)} \Big\vert \, \d x + \sum_{x \in \eta_{\Lambda}} \Big\vert\log \frac{b\Ssup{\mu}_\Lambda(x, \eta-\delta_{x})}{b\Ssup{\nu}_\Lambda(x, \eta-\delta_{x})} \Big\vert\Big] \\
        &\leq \big\vert \{x \in \Lambda \,\vert\, d(x, \partial \Lambda) \leq \range \}\big\vert \cdot \log\big(\Vert b\Vert_\infty \Vert 1/b\Vert_\infty\big) \Big(\Vert b\Vert_\infty + \mu\big[N_{[0,1 ]^d}\big] \Big).
    \end{align*} The right-hand side multiplied with \(1/\abs{\Lambda}\) goes to zero in \(L^1(\mu)\) as \(\Lambda \uparrow \mathbb{R}^d\).
\end{proof}

\subsection{Proof of the ergodic theorem}
As a last preparation for the proof of \Cref{theorem:main_theorem_HANDA} we now provide the proof of the ergodic theorem we rely on. 
\begin{proof}[Proof of \Cref{lemma:ergodic_theorem}]
    The existence of \(\zeta^\mu\) is the content of~\cite[Lemma~13.4.II]{DaleyVereJones2008} and the rest of our statement here follows completely analogous to~\cite[Theorem~13.4.III, Theorem~12.2.IV]{DaleyVereJones2008}, but the \(L^1(\mu)\)-convergence is for some reason not explicitly stated there and we check the details.

    Let \(h \in L^1(\mu_0^!)\) and without loss of generality assume that \(h \geq 0\).
    Fix some \(\epsilon > 0\) and let $\varphi_\epsilon$ be a continuous function on $\R^d$ such that $\varphi_\epsilon\geq 0$, $\int\varphi_\epsilon(u)\d u=1$, and whose support is contained in $B(0,\epsilon)$.
    
    Defining
    \begin{align*}
        f(\eta) = \sum_{u \in \eta} h(\theta_u (\eta -\delta_{u})) \varphi_\epsilon(u),
    \end{align*}
    we have
    \begin{align*}
        f(\theta_x \eta)
        = \sum_{y \in \eta} h(\theta_{y}(\eta -\delta_{y})) \,  \varphi_\epsilon(y-x).
    \end{align*} 
    Denoting \(\Lambda_n^{-\epsilon}:= \{x\in\R^d\colon  B(0,\epsilon)\subset \Lambda_n\}\) and \(\Lambda_n^{\epsilon}:= \bigcup_{x\in\Lambda_n}B(x,\epsilon)\), and letting \(L_n := \int_{\Lambda_n^{-\epsilon}} \, f(\theta_x \eta) \, \d x\) and \(U_n := \int_{\Lambda_n^{\epsilon}} \, f(\theta_x \eta) \, \d x \), we obtain
    \begin{align*}
        L_n
        \leq \sum_{y \in \eta_{\Lambda_n}} h(\theta_{y}(\eta -\delta_{y})) 
        \leq  U_n,
    \end{align*} 
    thanks to the following pointwise sandwiching
    \begin{align*}
        \int_{\Lambda_n^{-\epsilon}} \,  \varphi_\epsilon(y-x) \, \d x 
        \leq \1_{\Lambda_n}(y)
        \leq \int_{\Lambda_n^{\epsilon}} \,  \varphi_\epsilon(y-x) \, \d x.
    \end{align*}

    By~\cite[Proposition~12.2.II]{DaleyVereJones2008}, we have \(L_n/ \abs{\Lambda_n^{-\epsilon}} \to \mathbb{E}^{\mu}[f \,\vert\, \mathscr{S}]\) and \( U_n / \abs{\Lambda_n^{\epsilon}}\to \mathbb{E}^{\mu}[f \,\vert\, \mathscr{S}]\), as $n \to \infty$, \(\mu\)-a.s.\ and in \(L^1(\mu)\).
    Since, denoting \(Y := \mathbb{E}^{\mu}[f \,\vert\, \mathscr{S}] \), 
    \begin{align*}
        &\Big\vert \frac{1}{\abs{\Lambda_n}}\sum_{y \in \eta_{\Lambda_n}} h(\theta_{y}(\eta -\delta_{y})) -  \mathbb{E}^{\mu}[f \,\vert\, \mathscr{S}] \Big\vert \\
        &\leq \frac{\abs{\Lambda_n^{-\epsilon}}}{\abs{\Lambda_n}} \Big\vert \frac{L_n}{\abs{\Lambda_n^{-\epsilon}}} - Y \Big\vert + \Big(1 - \frac{\abs{\Lambda_n^{-\epsilon}}}{\abs{\Lambda_n}}\Big) \vert Y\vert + \frac{\abs{\Lambda_n^{\epsilon}}}{\abs{\Lambda_n}} \Big\vert \frac{U_n}{\abs{\Lambda_n^{\epsilon}}} - Y \Big\vert + \Big(1 - \frac{\abs{\Lambda_n^{\epsilon}}}{\abs{\Lambda_n}}\Big) \vert Y\vert,
    \end{align*} the convergence
    \begin{align*}
        \frac{1}{\abs{\Lambda_n}}\sum_{y \in \eta_{\Lambda_n}} h(\theta_{y}(\eta -\delta_{y})) 
        \xrightarrow[n \to \infty]{}
        \mathbb{E}^{\mu}[f \,\vert\, \mathscr{S}] 
    \end{align*} \(\mu\)-a.s.\ and in \(L^1(\mu)\) follows.
    Since we have
    \begin{align*}
        \mathbb{E}^{\mu}[f \,\vert\, \mathscr{S}] 
        = \int \d x \int \zeta^\mu(\d \eta \,\vert\, \cdot) \, h(\eta) \varphi_\epsilon(x) 
        = \int \zeta^\mu(\d \eta \,\vert\, \cdot) \, h(\eta),
    \end{align*} the second statement of this lemma is proved.
\end{proof}

\subsection{Proof of the thermodynamic limit}\label{sec:proof-main-result-handa}

We can now establish the thermodynamic limit of the relative entropy loss. 
\begin{proof}[Proof of \Cref{theorem:main_theorem_HANDA}]\label{proof:proof_of_main_theorem_HANDA}
    Let us first show that
    \begin{align*}
        \sigma(\mu)
        = \RelEnt\big(b(0, \cdot) \mu \,\vert\, \mu_0^!\big) + \RelEnt\big(\mu_0^!\,\vert\, b(0, \cdot) \mu \big).
    \end{align*}
    Thanks to \Cref{lemma:rest_term}, it is enough to consider the map $\L\mapsto \widetilde{\sigma}_\Lambda(\mu)$, with
    \begin{equation*}\label{eq:sigma_tilde_lambda_mu}
    \begin{split}
        \widetilde{\sigma}_\Lambda(\mu)
        &:= \mu\Big[ \int_{\Lambda} \, b(0, \theta_x \eta) \, \log\frac{\d(b(0, \cdot) \mu)_{\Lambda - x}}{\d (\mu_0^!)_{\Lambda - x}} \, \d x 
        + \sum_{x \in \eta_\Lambda} \log\frac{\d (\mu_0^!)_{\Lambda - x}}{\d(b(0, \cdot) \mu)_{\Lambda - x}}(\theta_x(\eta -\delta_{x}))  \Big]\\
        &= \int_\Lambda \, \big\{ \RelEnt_{\Lambda - x}\big(b(0, \cdot) \mu \,\big\vert\, \mu_0^! \big) +  \RelEnt_{\Lambda - x}\big(\mu_0^! \,\big\vert\, b(0, \cdot) \mu \big)  \big\} \, \d x.
    \end{split}
    \end{equation*}
    Since it is super-additive, its thermodynamic limit exists and is given by
\begin{align}\label{eq:sigma_as_sup}
        \lim_{\Lambda \uparrow \mathbb{R}^d}\frac{\widetilde{\sigma}_\Lambda(\mu)}{\vert\Lambda\vert}= \sup_{\Lambda \text{ cube}}\frac{\widetilde{\sigma}_\Lambda(\mu)}{\vert\Lambda\vert} = \lim_{\Lambda \uparrow \mathbb{R}^d} \frac{1}{\vert\Lambda\vert} \mu\Big[(-\mathscr{L}) \log\frac{\d \mu_\Lambda}{\d \nu_\Lambda}\Big] = \sigma(\mu).
    \end{align}
    Moreover, since
    \begin{align*}
        \widetilde{\sigma}_\Lambda(\mu)
        \leq \abs{\Lambda} \Big(  \RelEnt\big(b(0, \cdot) \mu \,\vert\, \mu_0^!\big) + \RelEnt\big(\mu_0^! \,\vert\, b(0, \cdot) \mu \big) \Big),
    \end{align*}
    it holds that
    \begin{align*}
        \sigma(\mu)
        \leq \RelEnt\big(b(0, \cdot) \mu \,\vert\, \mu_0^!\big) + \RelEnt\big(\mu_0^! \,\vert\, b(0, \cdot) \mu \big),
    \end{align*}
    and \(\RelEnt\big(b(0, \cdot) \mu \,\vert\, \mu_0^! \big) + \RelEnt\big(\mu_0^! \,\vert\, b(0, \cdot) \mu \big) = \infty\) if \(\sigma(\mu) = \infty\).

    Let us now assume \(\sigma(\mu) < \infty\). Then, since the relative entropy is increasing in the volume, we have that, for every \(\Delta \Subset \mathbb{R}^d\),
    \begin{equation*}
    \begin{split}
        \RelEnt_\Delta\big(b(0, \cdot) \mu \,\vert\, \mu_0^!\big) + \RelEnt_\Delta\big(\mu_0^! \,\vert\, b(0, \cdot) \mu \big)
        &= \Big(\limsup_{\Lambda \uparrow \mathbb{R}^d} \frac{1}{\abs{\Lambda}} \int_{\Lambda} \, \1_{\Delta \subseteq \Lambda -x} \, \d x \Big) \Big(\RelEnt_\Delta\big(b(0, \cdot) \mu \,\vert\, \mu_0^!\big) + \RelEnt_\Delta\big(\mu_0^! \,\vert\, b(0, \cdot) \mu \big) \Big) \\
        &\leq \limsup_{\Lambda \uparrow \mathbb{R}^d} \frac{1}{\abs{\Lambda}} \int_{\Lambda} \,  \Big(\RelEnt_{\Lambda - x}\big(b(0, \cdot) \mu \,\vert\, \mu_0^!\big) + \RelEnt_{\Lambda - x}\big(\mu_0^! \,\vert\, b(0, \cdot) \mu \big) \Big)= \sigma(\mu),
    \end{split}
    \end{equation*}
    by \Cref{eq:sigma_as_sup}.
    We can then apply~\cite[Theorem~6.6]{nguyen1979ergodic}, to get that \(b(0, \cdot)\mu \ll \mu_0^!\) and \(\mu_0^! \ll b(0, \cdot)\mu\),  and
    \begin{equation*}
    \begin{split}
        &\RelEnt\big(b(0, \cdot) \mu \,\vert\, \mu_0^!\big) + \RelEnt\big(\mu_0^!\,\vert\, b(0, \cdot) \mu \big)= \sup_{\Lambda} \Big( \RelEnt_\Lambda\big(b(0, \cdot) \mu \,\vert\, \mu_0^!\big) + \RelEnt_\Lambda\big(\mu_0^! \,\vert\, b(0, \cdot) \mu \big) \Big),
    \end{split}
    \end{equation*} 
    and that the convergences 
    \begin{align*}
        \log \frac{\d (b(0, \cdot)\mu)_\Lambda}{\d (\mu_0^!)_\Lambda}
        \xrightarrow[\Lambda \uparrow \mathbb{R}^d]{}
        \log \frac{\d (b(0, \cdot)\mu)}{\d \mu_0^!} \quad\text{in } L^1(b(0, \cdot)\mu)
    \end{align*} 
    and
\begin{align}\label{eq:log_RN_derivative_convergence}
        \log \frac{\d (\mu_0^!)_\Lambda}{\d (b(0, \cdot)\mu)_\Lambda}
        \xrightarrow[\Lambda \uparrow \mathbb{R}^d]{}
        \log \frac{\d \mu_0^!}{\d (b(0, \cdot)\mu)}\quad\text{in } L^1(\mu_0^!)
    \end{align} 
    hold. 
    In particular, it follows that \(\RelEnt\big(b(0, \cdot) \mu \,\vert\, \mu_0^!\big) + \RelEnt\big(\mu_0^! \,\vert\, b(0, \cdot) \mu \big) \leq \sigma(\mu) \), and hence \(\RelEnt\big(b(0, \cdot) \mu \,\vert\, \mu_0^!\big) + \RelEnt\big(\mu_0^! \,\vert\, b(0, \cdot) \mu \big) = \sigma(\mu)\). 

    Let us now establish the stated \(L^1(\mu)\)-convergence of \(\abs{\Lambda}^{-1}(-\mathscr{L})\log(\d \mu_\Lambda/\d \nu_\Lambda)\).
    By the ergodic theorem of \Cref{lemma:ergodic_theorem}, below, we have that
    \begin{equation}\label{eq:convergence}
    \begin{split}
        &\frac{1}{\abs{\Lambda }} \int_{\Lambda} \, b(0, \theta_x \eta) \, \log \frac{\d(b(0, \cdot) \mu)}{\d \mu_0^!}(\theta_x \eta) \, \d x 
        + \frac{1}{\abs{\Lambda }}\sum_{x \in \eta_\Lambda} \log \frac{\d \mu_0^!}{\d(b(0, \cdot) \mu)}(\theta_x(\eta -\delta_{x})) \\
        &\xrightarrow[\Lambda \uparrow \mathbb{R}^d]{}
        \mathbb{E}^{\mu}\Big[b(0, \cdot)\,\log \frac{\d(b(0, \cdot) \mu)}{\d \mu_0^!} \,\Big\vert\, \mathscr{S} \Big] 
        + \mathbb{E}^{\mu_0}\Big[\log \frac{\d \mu_0^!}{\d(b(0, \cdot) \mu)} \circ \gamma_0 \,\Big\vert\, \mathscr{S} \Big].
    \end{split}
    \end{equation}
    Recall that, from \Cref{lemma:rest_term} above, for the measurable and strictly positive version \(\varrho_\Lambda\) of \((x, \eta) \mapsto \frac{\d(\mu_0^!)_{\Lambda -x}}{\d (b(0, \cdot) \mu)_{\Lambda -x}}(\eta)\), we have that
    \begin{align*}
        (-\mathscr{L}) \log\frac{\d \mu_\Lambda}{\d \nu_\Lambda} (\eta)
        = \int_{\Lambda} \, b(0, \theta_x \eta) \, \log \frac{1}{\varrho_{\Lambda}(x, \theta_x \eta)}
        + \sum_{x \in \eta_\Lambda} \log \varrho_\Lambda(x, \theta_x(\eta -\delta_{x}))
    + R_{\Lambda}(\eta),
    \end{align*} 
    with \(R_{\Lambda}/\abs{\Lambda} \to 0\), in $L^1(\mu)$, as \(\Lambda \uparrow \mathbb{R}^d\). It is then enough to see that we can replace the global Radon--Nikodym derivatives by the right local ones in \Cref{eq:convergence}.
    This is indeed the case, since
    \begin{equation*}
    \begin{split}
        &\mu\Big[\Big\vert \frac{1}{\abs{\Lambda}} \sum_{x \in \cdot_\Lambda} \log \varrho_\Lambda(x, \theta_x(\cdot -\delta_{x})) - \frac{1}{\abs{\Lambda}} \sum_{x \in \cdot_\Lambda} \log\frac{\d \mu_0^!}{\d(b(0, \cdot) \mu)}(\theta_x(\cdot -\delta_{x}))  \Big\vert \Big] \\
        &\leq \frac{1}{\abs{\Lambda}} \mu\Big[\sum_{x \in \cdot_\Lambda} \Big\vert \log \varrho_\Lambda(x, \theta_x(\cdot -\delta_{x})) - \log\frac{\d \mu_0^!}{\d(b(0, \cdot) \mu)}(\theta_x(\cdot -\delta_{x}))  \Big\vert \Big] \\
        &= \frac{1}{\abs{\Lambda}} \int_{\Lambda} \, \mu_0^!\Big[\Big\vert \log \frac{\d (\mu_0^!)_{\Lambda - x}}{\d(b(0, \cdot) \mu)_{\Lambda - x}}  - \log \frac{\d \mu_0^!}{\d(b(0, \cdot) \mu)}  \Big\vert \Big]  \, \d x
        \xrightarrow[\Lambda \uparrow \mathbb{R}^d]{} 0
    \end{split}
    \end{equation*} 
    by \Cref{eq:log_RN_derivative_convergence}, and since \(\mu_0^![\vert \log \frac{\d (\mu_0^!)_{\Lambda - x}}{\d(b(0, \cdot) \mu)_{\Lambda - x}}\vert]  + \mu_0^![\vert  \log \frac{\d \mu_0^!}{\d(b(0, \cdot) \mu)}  \vert]\) is bounded thanks to the general inequality of \Cref{lemma:abs_log_entropy_inequality}.
    By the same argument,
    \begin{align*}
       &\mu\Big[\Big\vert \frac{1}{\abs{\Lambda}} \int_{\Lambda} \, b(0, \theta_x \cdot) \, \log \frac{1}{\varrho_{\Lambda}(x, \theta_x \cdot)} \, \d x  - \frac{1}{\abs{\Lambda}} \int_{\Lambda} \, b(0, \theta_x \cdot) \, \log \frac{\d(b(0, \cdot) \mu)}{\d \mu_0^!}(\theta_x \cdot) \, \d x  \Big\vert  \Big]
        \xrightarrow[\Lambda \uparrow \mathbb{R}^d]{} 0,
    \end{align*}
    concluding the proof.
\end{proof}

\section{Gibbs measures and reversible measures}\label{section:gibbs_and_reversible_measures}

In~\cite{JKSZ25} we used the following definition of a {\em reversible} measure, which we will here now call {\em infinitesimally reversible} to avoid confusion.
\begin{definition}
    \(\mu \in \Pcal(\Omega)\) is called {\em infinitesimally reversible} if 
    \begin{equation*}
        \mu\big[(\mathscr{L} f) \, g\big] = \mu\big[(f \, (\mathscr{L} g)\big],\qquad  \forall f, g \in \Lcal.
    \end{equation*} 
   
\end{definition}

In~\cite{JKSZ25}, we prove the following statement.
\begin{proposition}[\cite{JKSZ25}, Proposition~3.1]
    \(\mu \in \Pcal_\theta\) is infinitesimally reversible iff \(\mu \in \GG_\theta\).
\end{proposition}

The right definition of reversibility for our studies here is the following.
\begin{definition}
    \(\mu \in \Pcal(\Omega)\) is called {\em reversible} if \(\mu[(T_t f) \, g] = \mu[(f \, (T_t g)]\) for every \(f, g \in \Lcal\). The set of translation-invariant reversible probability measures is denoted by \(\RR_\theta\).
\end{definition}

That these two definitions actually coincide here and that we still have $\mathbf{\GG_\theta = \RR_\theta}$ as promised in \Cref{prop:reversible_equals_gibbs_with_the_right_definition} will be proved in \Cref{section:reversible_measures_are_gibbs_measures}.

\subsection{Small-time exponential series expansion for local and bounded functions under Poisson-like moment conditions}\label{sec:small_time_exponential_series_expansion}

For the proof of the series expansion \Cref{prop:small_time_exponential_series_expansion}, we need the following technical bounds on remainders in the corresponding exponential series.
\begin{lemma}[Remainder estimate]\label{lemma:estimate_powers_of_generator}
    Let \(\mu\in \Pcal(\Omega)\) with
    \begin{align*}
        \mu\big[N_{\Delta}^k\big]^{1/k}
        \leq \frac{c_{\mu, 3}  \, c_{\mu, 2}^{\abs{\Delta}/k} \, k}{\log(1 + k/(c_{\mu, 1} \abs{\Delta}))}
    \end{align*} for some \(c_{\mu, 1}, c_{\mu, 2}, c_{\mu, 3} \geq 1\), all \(\Delta \Subset \mathbb{R}^d\) and \(k \in \mathbb{N}\).
    Let \(\Lambda\) be a bounded box in \(\mathbb{R}^d\) with \(\abs{\Lambda} \geq \abs{B_\range}\).
    Then, for all bounded and \(\Lambda\)-measurable \(f \colon \Omega \to \mathbb{R}\),
    \begin{align*}
        \mu\big[\big\vert \mathscr{L}^k f\big\vert \big]
        &\leq \Vert f \Vert_{\infty} \big(12 \Vert b \Vert_{\infty}c_{\mu, 1} c_{\mu, 3} c_{\mu, 2}^{\abs{\Lambda}} \abs{\Lambda} (k+1)\big)^{k}. 
    \end{align*}
\end{lemma}
Before providing the proof of the above statement, we wish to give a heuristic idea of the situation. If we want the exponential series 
\begin{align*}
    \sum_{k = 0}^{\infty} \frac{t^k}{k!} \mu[\lvert\mathscr{L}^k f\rvert]
\end{align*}
to converge for small \(t > 0\), the terms \(\mu[\lvert\mathscr{L}^k f\rvert]\) are not allowed to grow much faster than \(k^k \sim k! e^k/\sqrt{2\pi k}\).
Assuming, without loss of generality that \(\Vert f \Vert_\infty \leq 1\), we have that
\begin{align*}
    \vert \mathscr{L} f\vert \leq 2 \Vert b\Vert_\infty \abs{\Lambda} + 2 N_\Lambda
\end{align*} 
and the ``support" of \(\mathscr{L}f\) is now contained in \(\Lambda \oplus B(0, \range)\).
Similarly,
\begin{align*}
    \vert \mathscr{L}^2 f\vert 
    &\leq 2\Vert b\Vert_\infty \abs{\Lambda^{\oplus 1}} \big\{ 2 \Vert b\Vert_\infty \abs{\Lambda} + 2 N_\Lambda \big\}  + 2 N_{\Lambda^{\oplus 1}} \big\{ 2 \Vert b\Vert_\infty \abs{\Lambda} + 2 N_\Lambda  \big\} \\
    &= 2^2 \Vert b \Vert_\infty^2 \abs{\Lambda^{\oplus 1}} \abs{\Lambda} +  2^2 \Vert b \Vert_\infty \abs{\Lambda^{\oplus 1}} N_\Lambda + 2^2 \Vert b \Vert_\infty \abs{\Lambda} N_{\Lambda^{\oplus 1}} + 2^2 N_{\Lambda^{\oplus 1}} N_{\Lambda},
\end{align*} 
and now \(\mathscr{L}^2f\) has a ``support'' of \(\Lambda^{\oplus 1} := \Lambda \oplus B(0, \range)\).
Continuing this rough estimates, we see that \(\mathscr{L}^k f\) now has a support of \(\Lambda^{\oplus k} := \Lambda^{\oplus (k-1)} \oplus B(0, \range)\) and an upper bound involving exponentially many (in \(k\)) summands of the form \(\prod_{i = 1}^{k} A_i\), where \(A_i \in \{\lvert\Lambda^{\oplus i}\rvert, N_{\Lambda^{\oplus i}}\}\).
The number of such summands is not an issue, because small \(t > 0\) beats the corresponding exponential growth, but the products are troublesome. Indeed, recall that for the Poisson point process \(\pi\) we have, rather sharply~\cite{Ahle2022}, that
\begin{align*}
    \pi\big[N_{\Delta}^k\big]
    \leq \Big(\frac{k}{\log(1 + k/\abs{\Delta})}\Big)^{k}
\end{align*} for \(\Delta \Subset \mathbb{R}^d\), and in particular
\begin{align*}
    \pi\Big[\prod_{i = 1}^{k} N_{\Lambda^{\oplus i}} \Big] \leq
    \pi\big[N_{\Lambda^{\oplus (k/2)}}^{k/2}\big]
    \leq \Big(\frac{k/2}{\log\big(1 + k/(2\abs{\Lambda^{\oplus (k/2))}})\big)}\Big)^{k / 2},
\end{align*} 
where the right-hand side, for large even \(k\), grows too fast for our intended application:
\begin{align*}
    k^{- k} \Big(\frac{k/2}{\log\big(1 + k/(2\abs{\Lambda^{\oplus (k/2)}})\big)}\Big)^{k / 2}
    &= \Big(\frac{1/2}{k \log\big(1 + k/(2\abs{\Lambda^{\oplus (k/2)}})\big)}\Big)^{k / 2}\\
    &\geq \bigg(\frac{1/2}{k \log\big(1 + 1/k!\big)}\bigg)^{k / 2}
    \geq \big((k-1)!/2\big)^{k / 2}.
\end{align*}
Hence, we need to estimate \(\vert \mathscr{L}^k f\vert\) with more care. The key point here is that, in the above calculations, we wrongly incorporated many fictitious non-zero contributions by always looking at the whole ``support'' of \(\mathscr{L}^{k-1} f\), whereas they actually have to come only from ``chains of balls'' \(B(0, \range)\) due to the finite-range property of \(b\).

\begin{proof}
    We have
    \begin{align*}
        \mathscr{L}^k
        = (\mathscr{L}_{\text{birth}} + \mathscr{L}_{\text{death}})^k
        = \sum_{I \in \{\text{birth}, \text{death}\}^k} \mathscr{L}_I
    \end{align*} with \(\mathscr{L}_I := \prod_{r = 1}^{k} \mathscr{L}_{I_r}\)
    and the birth- and death-parts of the formal generators
    \begin{align*}
        (\mathscr{L}_{\text{birth}} g)(\eta)
        = \int_{\mathbb{R}^d} b(x, \eta) \, (g(\eta + \delta_x) - g(\eta)) \, \d x\quad \text{and}\quad (\mathscr{L}_{\text{death}} g)(\eta)
        = \sum_{x \in \eta} \, (g(\eta - \delta_x) - g(\eta)).
    \end{align*}
    For any given \(I \in \{\text{birth}, \text{death}\}^k\) we can write
    \begin{align*}
        (\mathscr{L}_I f)(\eta) 
        = \int \, \kappa^{(I)}(\mathrm{d}x_1 \dots \mathrm{d}x_k \,\vert\, \eta) \, F(x_1, \dots, x_k, \eta)
    \end{align*} with the measure
    \begin{align*}
        &\kappa^{(I)}(\mathrm{d}x_1 \dots \mathrm{d}x_k \,\vert\, \eta) =
        \kappa^{(I_1)}(\mathrm{d}x_1 \,\vert\, \eta) \, \kappa^{(I_2)}\Big(\mathrm{d}x_2 \,\vert\, \eta + \sum_{\substack{1 \leq i \leq 1 \\ I_i = \text{birth}}} \delta_{x_i} \Big) \dots \, \kappa^{(I_k)}\Big(\mathrm{d}x_k \,\vert\, \eta + \sum_{\substack{1 \leq i \leq k-1 \\ I_i = \text{birth}}} \delta_{x_i} \Big),
    \end{align*} where \(\kappa^{(\text{birth})}(\mathrm{d}x \, \vert \, \zeta)\) is just Lebesgue measure on \(\mathbb{R}^d\),  \(\kappa^{(\text{death})}(\mathrm{d}x \, \vert \, \zeta)\) is the counting measure on \(\zeta\), and
    \begin{align*}
        &F(x_1, \dots, x_k, \eta)= \sum_{J \in \{0, 1\}^k} (-1)^{\sum_{j = 1}^{k} J_j} \Big(\prod_{i = 1}^{k} b^{(I_i)}\bigl(x_i, \eta + \sum_{j = 1}^{i-1} J_j \sgn(I_j) \delta_{x_j}\bigr)\Big) f\Big(\eta + \sum_{j = 1}^{k} J_j \sgn(I_j) \delta_{x_j} \Big),
    \end{align*} 
    where \(b^{(\text{birth})} = b\) and \(b^{(\text{death})} \equiv 1\). 

    From \(f = f(\cdot_\Lambda)\), as well as the finite-range property of \(b\), and resulting symmetry, we see that \(F(x_1, \dots, x_k, \eta) = 0\) unless \(x_k \in \Lambda\) and \(x_i \in \Lambda \cup \bigcup_{\substack{i < j \leq k, I_j = \text{birth}}} B_\range(x_j)\) for all \(1 \leq i < k\).
    Therefore,
    \begin{align*}
        \mu\big[\big\vert \mathscr{L}_I f \big\vert \Big]
        &\leq \Vert F \Vert_{\infty} \, \mu\Big[ \int \, \kappa^{(I)}(\mathrm{d}x_1 \dots \mathrm{d}x_k \,\vert\, \eta) \, \1_{\forall 1 \leq i \leq k \colon x_i \in \Lambda \cup \bigcup_{\substack{i < j \leq k, I_j = \text{birth}}} B_R(x_j)} \Big] \\
        &\leq \Vert f \Vert_{\infty} \big(6 \Vert b \Vert_{\infty}c_{\mu, 1} c_{\mu, 3} c_{\mu, 2}^{\abs{\Lambda}} \abs{\Lambda} (k+1)\big)^{k},
    \end{align*} 
    where we used the fact that
 $
        \Vert F \Vert_{\infty}
        \leq 2^k \,\Vert b \Vert_{\infty}^k \, \Vert f \Vert_{\infty}
$
    and
    \begin{align*}
        \mu\Big[ \int \, \kappa^{(I)}(\mathrm{d}x_1 \dots \mathrm{d}x_k \,\vert\, \eta) \, \1_{\forall 1 \leq i \leq k \colon x_i \in \Lambda \cup \bigcup_{\substack{i < j \leq k, I_j = \text{birth}}} B_\range(x_j)} \Big] 
        \leq \big(3 c_{\mu, 1} c_{\mu, 3} c_{\mu, 2}^{\abs{\Lambda}} \abs{\Lambda} (k+1)\big)^{k}. 
    \end{align*} 
    Let us now show the latter bound to be true. First, let \(\#\text{birth}\), resp.\ \(\#\text{death}\), be the respective number of occurrences of \(\text{birth}\), resp.\ \(\text{death}\) in \(I\) and let \(i^{(\text{b})}_j, i^{(\text{d})}_l\) be the index of the \(j\)-th, resp. \(l\)-th, occurrence of \(\text{birth}\), resp.\ \(\text{death}\), in \(I\) (read from left to right). Then
    \begin{align*}
        &\mu\Big[ \int \, \kappa^{(I)}(\mathrm{d}x_1 \dots \mathrm{d}x_k \,\vert\, \eta) \, \1_{\forall 1 \leq i \leq k \colon x_i \in \Lambda \cup \bigcup_{\substack{i < j \leq k\colon I_j = \text{birth}}} B_R(x_j)} \Big] \\
        &= \int_{\Lambda} \d x_{i^{(\text{b})}_{\#\text{birth}}} \, \dots \, \int_{\Lambda \cup \bigcup_{1 < j \leq \#\text{birth}} B_R(x_{i^{(\text{b})}_j})} \, \d x_{i^{(\text{b})}_1}  \\
        &\quad
        \mu\Bigg[\sum_{x_{i^{(\text{d})}_1} \in \Big(\eta + \sum_{\substack{j < i_1^{(\text{d})}\colon \\ I_j = \text{birth}}} \delta_{x_j}\Big)_{\Lambda \cup \bigcup_{\substack{j > i_1^{(\text{d})}\colon \\ I_j = \text{birth}}} B_R(x_j)}}
        \dots \sum_{x_{i^{(\text{d})}_{\#\text{death}}} \in \Big(\eta + \sum_{\substack{j < i^{(\text{d})}_{\#\text{death}}\colon \\ I_j = \text{birth}}} \delta_{x_j}\Big)_{\Lambda \cup \bigcup_{\substack{j > i^{(\text{d})}_{\#\text{death}}\colon \\ I_j = \text{birth}}} B_\range(x_j)}} 1 \Bigg].
    \end{align*}
    The expectation can be bounded as
    \begin{align*}
        &\mu\Bigg[\sum_{x_{i^{(\text{d})}_1} \in \Big(\eta + \sum_{\substack{j < i_1^{(\text{d})}\colon \\ I_j = \text{birth}}} \delta_{x_j}\Big)_{\Lambda \cup \bigcup_{\substack{j > i_1^{(\text{d})}\colon \\ I_j = \text{birth}}} B_\range(x_j)}} \dots \sum_{x_{i^{(\text{d})}_{\#\text{death}}} \in \Big(\eta + \sum_{\substack{j < i^{(\text{d})}_{\#\text{death}}\colon \\ I_j = \text{birth}}} \delta_{x_j}\Big)_{\Lambda \cup \bigcup_{\substack{j > i^{(\text{d})}_{\#\text{death}}\colon \\ I_j = \text{birth}}} B_\range(x_j)}} 1 \Bigg] \\
        &\leq \mu\Big[\Big(N_{\Lambda \cup \bigcup_{\substack{j > i_1^{(\text{d})}\colon I_j = \text{birth}}} B_\range(x_j)}(\eta) + \#\{j < i^{(\text{d})}_{1} \,\vert\, I_j = \text{birth}\}\Big) \\ 
        &\qquad\quad \cdots \Big(N_{\Lambda \cup \bigcup_{\substack{j > i^{(\text{d})}_{\#\text{death}}\colon I_j = \text{birth}}} B_\range(x_j)}(\eta) + \#\{j < i^{(\text{d})}_{\#\text{death}} \,\vert\, I_j = \text{birth}\}\Big) \Big] \\
        &\leq \prod_{l = 1}^{\#\text{death}} \,\mu\Big[\Big(N_{\Lambda \cup \bigcup_{\substack{j > i_l^{(\text{d})}\colon I_j = \text{birth}}} B_\range(x_j)}(\eta) + \#\{j < i^{(\text{d})}_{l} \,\vert\, I_j = \text{birth}\}\Big)^{\#\text{death}} \Big]^{1/\#\text{death}}. 
    \end{align*}
    Looking now at fixed \(l\),
    \begin{align*}
        &\mu\Big[\Big(N_{\Lambda \cup \bigcup_{\substack{j > i_l^{(\text{d})}\colon I_j = \text{birth}}} B_\range(x_j)}(\eta) + \#\{j < i^{(\text{d})}_{l} \,\vert\, I_j = \text{birth}\}\Big)^{\#\text{death}} \Big]^{1/\#\text{death}}\\
        &\leq \mu\Big[\Big(N_{\Lambda \cup \bigcup_{\substack{j > i_l^{(\text{d})}\colon I_j = \text{birth}}} B_\range(x_j)}(\eta)\Big)^{\#\text{death}} \Big]^{1/\#\text{death}} + \, \#\{j < i^{(\text{d})}_{l} \,\vert\, I_j = \text{birth}\}.
    \end{align*} 
    By assumption on \(\mu\),
    \begin{align*}
        \mu\Big[\Big(N_{\Lambda \cup \bigcup_{\substack{j > i_l^{(\text{d})}\colon I_j = \text{birth}}} B_\range(x_j)}(\eta)\Big)^{\#\text{death}} \Big]^{1/\#\text{death}}
        \leq \frac{c_{\mu, 3} \, c_{\mu, 2}^{L / \#\text{death}} \, \#\text{death}}{\log\left(1 + \#\text{death}/L \right)}
    \end{align*} 
    where \(L = c_{\mu, 1} \big\vert \Lambda \cup \bigcup_{\substack{j > i_l^{(\text{d})}\colon  I_j = \text{birth}}} B_\range(x_j) \big\vert\).
    But
    \begin{align*}
        \frac{x}{\log(1 + x/y)}
        \leq \max\left\{ \frac{1}{\log(3/2)} x, \e^{1/4} y \right\}
        \leq 3(x+y)
    \end{align*} 
    for \(x, y > 0\), and therefore
    \begin{align*}
        &\mu\Big[\Big(N_{\Lambda \cup \bigcup_{\substack{j > i_l^{(\text{death})}\colon I_j = \text{birth}}} B_\range(x_j)}(\eta) + \#\{j < i^{(\text{d})}_{l} \,\vert\, I_j = \text{birth}\}\Big)^{\#\text{death}} \Big]^{1/\#\text{death}}\\
        &\leq 3 c_{\mu, 1}c_{\mu, 3} c_{\mu, 2}^{k \abs{\Lambda} / \#\text{death}} \abs{\Lambda} \Big(\#\text{death} + 1  + \#\{j > i^{(\text{d})}_{l} \,\vert\, I_j = \text{birth}\}
        + \#\{j < i^{(\text{d})}_{l} \,\vert\, I_j = \text{birth}\} \Big) \\
        &\leq 3 c_{\mu, 1}c_{\mu, 3} c_{\mu, 2}^{k \abs{\Lambda} / \#\text{death}}  \abs{\Lambda} (k+1).
    \end{align*}
    It now follows that
    \begin{align*}
        &\mu\Big[ \int \, \kappa^{(I)}(\mathrm{d}x_1 \dots \mathrm{d}x_k \,\vert\, \eta) \, \1_{\forall 1 \leq i \leq k \colon x_i \in \Lambda \cup \bigcup_{\substack{i < j \leq k,\colon I_j = \text{birth}}} B_\range(x_j)} \Big] \\
        &\leq \int_{\Lambda} \d x_{i^{(\text{b})}_{\#\text{birth}}} \, \dots \, \int_{\Lambda \cup \bigcup_{1 < j \leq \#\text{birth}} B_R(x_{i^{(\text{b})}_j})}  \d x_{i^{(\text{b})}_1} \,  c_{\mu, 2}^{k\abs{\Lambda}} \big(3 c_{\mu, 1} c_{\mu, 3} \abs{\Lambda} (k+1)\big)^{\#\text{death}} \\
        &\leq  \big(\#\text{birth} \cdot \abs{\Lambda}\big)^{\#\text{birth}} c_{\mu, 2}^{k\abs{\Lambda}} \big(3 c_{\mu, 1} c_{\mu, 3} \abs{\Lambda} (k+1)\big)^{\#\text{death}} \\
        &\leq \big(3 c_{\mu, 1} c_{\mu, 3} c_{\mu, 2}^{\abs{\Lambda}} \abs{\Lambda} (k+1)\big)^{k}, 
    \end{align*}
    as desired. 
\end{proof}

Let us check that the moment requirements for \Cref{lemma:estimate_powers_of_generator} stay fulfilled under time evolution.

\begin{lemma}\label{lemma:moments_stay_okay}
    Let \(\mu\in\Pcal(\Omega)\) with
    \begin{align*}
        \mu\big[N_{\Delta}^k\big]^{1/k}
        \leq \frac{c_{\mu, 3}  \, c_{\mu, 2}^{\abs{\Delta}/k} \, k}{\log(1 + k/(c_{\mu, 1} \abs{\Delta}))}
    \end{align*} for some \(c_{\mu, 1}, c_{\mu, 2}, c_{\mu, 3} \geq 1\), all \(\Delta \Subset \mathbb{R}^d\) and \(k \in \mathbb{N}\).
    Then \(\mu T_t\) fulfills the same bounds with \(c_{{\mu T_t}, 3} = 1 + c_{\mu, 3}\), \(c_{{\mu T_t}, 2} = c_{\mu, 2}\) and  \(c_{{\mu T_t}, 1} = \max\{c_{\mu, 1}, t \Vert \birth \Vert_{\infty}\}\).
\end{lemma}
\begin{proof}
    We have, using the graphical representation,    
    \begin{align*}
        (\mu T_t)[N_{\Delta}^k ]^{1/k}
        &= \Big(\int \mu(\d \omega) \, \mathbb{E}\big[N_{\Delta}^k(X_t\Ssup{\omega})\big]\Big) \\ 
        &\leq \Big(\int \mu(\d \omega) \, \mathbb{E}\Big[\Big(\int_{\Delta \times [0,t]\times[0,+\infty) \times [0, \Vert \birth \Vert_{\infty}]} N(\d x,\d s,\d r,\d u) + N_{\Delta}(\omega)\Big)^k\Big]\Big)^{1/k} \\
        &\leq \Big(\int \mu(\d \omega) N_{\Delta}(\omega)^k\Big)^{1/k} + \mathbb{E}\Big[\Big(\int_{\Delta \times [0,t]\times[0,+\infty) \times [0, \Vert \birth \Vert_{\infty}]} N(\d x,\d s,\d r,\d u)\Big)^k\Big]^{1/k} \\
        &\leq \frac{c_{\mu, 3}  \, c_{\mu, 2}^{\abs{\Delta}/k} \, k}{\log(1 + k/(c_{\mu, 1} \abs{\Delta}))} 
        + \frac{k}{\log(1 + k/(t \Vert \birth \Vert_{\infty} \abs{\Delta}))}\leq \frac{c_{{\mu T_t}, 3}  \, c_{{\mu T_t}, 2}^{\abs{\Delta}/k} \, k}{\log(1 + k/(c_{{\mu T_t}, 1} \abs{\Delta}))},
    \end{align*}
    with \(c_{{\mu T_t}, 3} = 1 + c_{\mu, 3}\), \(c_{{\mu T_t}, 2} = c_{\mu, 2}\) and  \(c_{{\mu T_t}, 1} = \max\{c_{\mu, 1}, t \Vert \birth \Vert_{\infty}\}\).
\end{proof}
Using the technical lemmas from above, we can establish the series expansion for $\mu[T_t f]$. 
\begin{proof}[Proof of \Cref{prop:small_time_exponential_series_expansion}] 
    We have
    \begin{align*}
        (T_t f)(\eta)
        = \sum_{k = 0}^{N} \frac{t^k}{k!} (\mathscr{L}^k f)(\eta) + \int_{0}^{t} \d s \, \frac{s^N}{N!} T_s(\mathscr{L}^{N+1} f)(\eta),
    \end{align*} 
    by iteration of the identity 
    \begin{align*}
        T_t g
        = g + \int_{0}^{t} \d s \, T_s\mathscr{L} g
    \end{align*} 
    for local \(g\) with good moments. Note herein that \(g = \mathscr{L}^{N} f\) stays local (but not \(\Lambda\)-local).
    It follows that
    \begin{align*}
        \mu\big[T_t f \big]
        = \sum_{k = 0}^{N} \frac{t^k}{k!} \mu\big[\mathscr{L}^k f \big]
        + \mu\Big[  \int_{0}^{t} \d s \, \frac{s^N}{N!} T_s(\mathscr{L}^{N+1} f) \Big].
    \end{align*} By \Cref{lemma:estimate_powers_of_generator} and \Cref{lemma:moments_stay_okay},
    \begin{align*}
        &\Big\vert \mu\Big[  \int_{0}^{t} \d s \, \frac{s^N}{N!} T_s(\mathscr{L}^{N+1} f) \Big] \Big\vert 
        \leq \int_{0}^{t} \d s \, \frac{s^N}{N!} (\mu T_s)\big[\big\vert \mathscr{L}^{N+1}f \big\vert \big] \\
        &\leq  \Vert f \Vert_{\infty} \int_{0}^{t} \d s \, \frac{s^N}{N!} \big(12 \Vert b \Vert_{\infty} \max\{c_{\mu, 1}, t \Vert b\Vert_\infty\} (c_{\mu, 3}+1) c_{\mu, 2}^{\abs{\Lambda}} \abs{\Lambda}\big)^{N+1} (N+2)^{N+1}.
    \end{align*} Stirling's approximation implies, for \(N\) large enough,
    \begin{align*}
        &\int_{0}^{t} \d s \, \frac{s^N}{N!} \big(12 \Vert b \Vert_{\infty}c_{\mu, 1} c_{\mu, 3} c_{\mu, 2}^{\abs{\Lambda}} \abs{\Lambda}\big)^{N+1} (N+2)^{N+1}\leq  (\e \widetilde{c} t)^{N+1},
    \end{align*} with \(\widetilde{c} = 12 \Vert b \Vert_{\infty} \max\{c_{\mu, 1}, t \Vert b\Vert_\infty\} (c_{\mu, 3}+1) c_{\mu, 2}^{\abs{\Lambda}} \abs{\Lambda}\).
    Hence,
    \begin{align*}
        \mu\big[T_t f\big]
        = \sum_{k = 0}^{\infty} \frac{t^k}{k!} \mu\big[\mathscr{L}^k f\big]
    \end{align*} if \(t < \frac{1}{c}\) for \(c = 12 \e \Vert b \Vert_{\infty} c_{\mu, 1} (c_{\mu, 3}+1) c_{\mu, 2}^{\abs{\Lambda}} \abs{\Lambda}\).
\end{proof}

\subsection{Non-Reversible (-Gibbs) to reversible (Gibbs) in finite time is impossible}\label{section:impossibility_of_finite_time_gibbs_without_gibbs_start}

Let us now check that the moments of Gibbs measures with respect to the considered interaction $H$ are good enough to apply the small-time exponential series expansion derived above.

\begin{proof}[Proof of \Cref{lemma:moments_of_Area_interaction_Gibbs_measures}]
    For a Poisson point process \(\pi^{z}\) with intensity \(z > 0\) the counting variables are simply Poisson variables, so it holds that
    \begin{align*}
        \pi^z\big[N_{\Delta}^k\big]^{1/k}
        \leq \frac{k}{\log(1 + k/(z \abs{\Delta}))},
    \end{align*}
    see, e.g.,~\cite{Ahle2022}.
    The result then follows by simply applying the crude bound
    \begin{align*}
        \frac{\d \nu_\Delta}{\d \pi_\Delta}
        \leq c_{\nu}^{\abs{\Delta}} z_{\nu}^{N_\Delta}.
    \end{align*}
    This proves the result. 
\end{proof}

Next, we show that if we started with an initial distribution \(\mu\) and the evolved measure \(\mu T_t\) satisfies some moment conditions, then the original measure \(\mu\) has to satisfy similar moment conditions.

\begin{proof}[Proof of \Cref{lemma:moments_cannot_be_much_worse_at_starting_point}]
    Note that, since the lifespan of points is exponentially distributed with parameter $1$, we have that the expectation of the number of initial points \(\omega\) that are still alive at time \(t\) is given by \(\E[N_\Delta(\underline\omega_t)] = \e^{-t} N_\Delta(\omega)\), and
    \begin{align*}
        \mu\big[N_{\Delta}^k\big]
        &= \e^{tk} \, \mu\big[(\e^{-t} N_{\Delta})^k\big]
        = \e^{tk} \mu\big[\mathbb{E}[N_{\Delta}(\underline{\omega}_t)]^k\big] \leq \e^{tk} \mu\big[\mathbb{E}[N_{\Delta}(\underline{\omega}_t)^k]\big]
       \\ & \leq \e^{tk} \mu\big[\mathbb{E}[N_{\Delta}(X_t\Ssup{\omega})^k]\big]= \e^{tk} (\mu T_t)\big[N_{\Delta}^k\big]
        = \e^{tk} \widetilde{\nu}\big[N_{\Delta}^k\big],
    \end{align*}
    as desired. 
\end{proof}

\subsection{Equivalence of reversible and Gibbs measures}\label{section:reversible_measures_are_gibbs_measures}

In this section we present the proof of \Cref{prop:reversible_equals_gibbs_with_the_right_definition}, split into the two inclusions below.
\begin{proposition}
    We have \(\GG_\theta \subseteq \RR_\theta\).
\end{proposition}

\begin{proof}
    Let \(\nu \in \GG_\theta\). Then, $\nu[f\LL g]=\nu[g\LL f]$ holds for any pair $f,g\in \Lcal$ by the GNZ equations.
    The (proof of the) series expansion from \Cref{prop:small_time_exponential_series_expansion} together with \Cref{lemma:moments_of_Area_interaction_Gibbs_measures} imply that the corresponding moment assumptions hold and thus gives at small \(t > 0\) that
    \begin{align*}
        \nu\big[ (T_t f) g \big]
        = \sum_{k = 0}^{\infty} \frac{t^k}{k!} \nu\big[ (\mathscr{L}^k f) g \big]
        = \sum_{k = 0}^{\infty} \frac{t^k}{k!} \nu\big[ f (\mathscr{L}^k g) \big]
        = \nu\big[ f (T_t g) \big]
    \end{align*} for any \(f, g \in \Lcal\). E.g. the monotone convergence theorem extends the identity of left- and right-hand side to all \(f, g\) of sufficient integrability. Hence, we can extend the identity to all \(t > 0\) inductively by virtue of
    \begin{align*}
        \nu\big[ (T_{k \delta} f) g \big]
        = \nu\big[ (T_{\delta} (T_{(k-1)\delta}f)) \, g \big]
        = \nu\big[ (T_{(k-1)\delta}f) \, (T_\delta g) \big]
        = \dots
        = \nu\big[ f (T_{k \delta} g) \big],
    \end{align*} with \(\delta > 0\) small and \(k \in \mathbb{N}\).
\end{proof}

An alternative proof is also possible via approximation with the local evolutions \((\Tl_t)_{t \geq 0}\) from \cite{JKSZ25}, which admit reversible measures \(\nu_\Lambda\).

\begin{proposition}
    We have \(\RR_\theta \subseteq \GG_\theta\).
\end{proposition}
\begin{proof}
    Let \(\mu \in \RR_\theta\), meaning it holds
    \begin{align*}
        \mu[(T_t f) g]
        = \mu[f (T_t g)]
    \end{align*} for all \(f, g\) with sufficient \(\mu\)-integrability.
    For bounded \(f, g \in \mathcal{L}\) we have
    \begin{align*}
        \frac{T_t f - f}{t}
        \xrightarrow[t \downarrow 0]{} \mathscr{L}f
    \end{align*} pointwise.
    Letting \(\phi\) denote any convex increasing function, we have by Jensen's inequality,
    \begin{align*}
        \mu\Big[\phi\Big( \Big\vert \frac{T_t f - f}{t}  \Big\vert\Big)\Big]
       & = \mu\Big[\phi\Big( \Big\vert \frac{1}{t} \int_{0}^{t} (T_s \mathscr{L} f)  \, \d s  \Big\vert\Big)\Big] \leq \mu\Big[\frac{1}{t} \int_{0}^{t} (T_s \phi(\vert \mathscr{L} f \vert ))  \, \d s  \Big] \\
      & = \frac{1}{t} \int_{0}^{t}\mu\big[ T_s \phi(\vert \mathscr{L} f) \vert\big] \d s = \,\mu\big[ \phi(\vert \mathscr{L} f \vert )\big]. 
    \end{align*} 
    Hence, if \(\mathscr{L} f\) and \(\mathscr{L} g\) are (uniformly) integrable, so are \(\big((T_t f - f)/t\big)_{t \geq 0}\) and \(\big((T_t g - f)/t\big)_{t \geq 0}\) and it holds by reversibility of \(\mu\) that
    \begin{align*}
        \mu\big[(\mathscr{L} f) \, g \big]
        = \lim_{t \downarrow 0} \mu\Big[\frac{T_t f - f}{t} \, g \Big]
        = \lim_{t \downarrow 0} \mu\Big[f \, \frac{T_t g - g}{t}  \Big]
        = \mu\big[ f \, (\mathscr{L} g) \big].
    \end{align*} According to the proof of~\cite[Proposition~3.1]{JKSZ25}, to now see that \(\mu \in \GG_\theta\), we only have to check the equality \(\mu\big[(\mathscr{L} f) \, g\big] = \mu\big[(f \, (\mathscr{L} g)\big]\) for a class of functions \(f, g \in \Lcal\) such that \(\mathscr{L} f\) and  \(\mathscr{L} g\) are bounded and hence automatically \(\mu\)-integrable too.
\end{proof}

\section{Fisher information and quantitative decay}
\subsection{Thermodynamic limit of the Fisher information}\label{section:fisher_info}
Recall that we denote the finite-volume Fisher information with boundary configurations sampled from \(\nu\) by
    \begin{align*}
        \mathcal{J}^{(\nu)}_{\L}(\mu\lvert \nu)
        &= \int\d x\int \nu_{\L}(\d\eta)\  b_\Lambda^{\nu}(x,\eta_{\L})\ D_x\frac{\d\mu_{\L}}{\d\nu_\L}(\eta_{\L}) \ D_x\log\frac{\d\mu_\L}{\d\nu_\L}(\eta_{\L}) 
    \end{align*}
and the finite-volume Fisher information with free boundary conditions by
    \begin{align*}
        \mathcal{J}^{(\emptyset)}_{\L}(\mu\lvert \nu)
        &= \mathcal{J}_{\L}(\mu\lvert \nu).
    \end{align*}

The following result states that the thermodynamic limits of these two quantities exist and, moreover, that the choice of these two different boundary conditions does not matter macroscopically. 
\begin{lemma}\label{lemma:equivalence_fisher_infos}
    Under the assumptions of \Cref{theorem:main_theorem_HANDA}, we have
    \begin{align*}
        \xi^\mu(0)
        = \sigma(\mu)
        = \mathcal{J}(\mu \vert\nu) 
        =  \lim_{n \uparrow \infty} \frac{1}{\abs{\Lambda_n}}\mathcal{J}^{(\emptyset)}_{\L_n}(\mu\lvert \nu)
        =  \lim_{n \uparrow \infty} \frac{1}{\abs{\Lambda_n}}\mathcal{J}^{(\nu)}_{\L_n}(\mu\lvert \nu).
    \end{align*}
\end{lemma}
\begin{proof}
    Let us first recall that we have
    \begin{align*}
        \xi^\mu(0)
        = \sigma(\mu)
        = \lim_{n \uparrow \infty} \frac{1}{\abs{\Lambda_n}} \mu\Big[(-\mathscr{L}) \log\frac{\d\mu_{\L_n}}{\d \nu_{\L_n}}\Big]
    \end{align*}
    by \Cref{theorem:main_theorem_HANDA} and that
    \begin{align*}
        \mathcal{J}^{(\nu)}_{\L}(\mu\lvert \nu)
        &= \int\d x\int \nu_{\L}(\d\eta)\  b_\Lambda^{\nu}(x,\eta_{\L})\ D_x\frac{\d\mu_{\L}}{\d\nu_\L}(\eta_{\L}) \ D_x\log\frac{\d\mu_\L}{\d\nu_\L}(\eta_{\L}) \\
        &= \nu\Big[\frac{\d\mu_{\L}}{\d \nu_{\L}}\mathscr{L}_{\L}^{\nu} \log\frac{\d\mu_{\L}}{\d \nu_{\L}}\Big] 
        = \mu\Big[\mathscr{L}_{{\L}}^{\nu} \log\frac{\d\mu_{\L}}{\d \nu_{\L}}\Big].
    \end{align*}
    Let us assume \(\sigma(\mu) = \RelEnt\big(\mu_0^! \,\big\vert\, b(0, \cdot) \mu  \big) + \RelEnt\big(b(0, \cdot) \mu \,\big\vert\, \mu_0^! \big) < \infty\), the other case is handled similarly. Notice that these relative entropies individually, opposed to their sum, do not have to be non-negative, as both involved measures are not normalized to be probability measures. Still, the only non-finite value they can attain is \(+\infty\).
    We now have that
    \begin{align*}
        \Big\vert \mu\Big[(-\mathscr{L}) \log\frac{\d\mu_{\L}}{\d \nu_{\L}}\Big] -  \mathcal{J}^{(\nu)}_{\L}(\mu\lvert \nu) \Big\vert
        &= \Big\vert \mu\Big[(-\mathscr{L}) \log\frac{\d\mu_{\L}}{\d \nu_{\L}}\Big] - \mu\Big[(-\mathscr{L}_{{\L}}^{\nu}) \log\frac{\d\mu_{\L}}{\d \nu_{\L}}\Big] \Big\vert \\
        &\lesssim_{b, \mu} \int_{\{\dist(\cdot, \partial \Lambda) \leq \range\}} \d x \, \Big\vert \mu\Big[D_x\log\frac{\d\mu_{\L}}{\d \nu_{\L}}  \Big]\Big\vert \\
        &\lesssim_{b} \int_{\{\dist(\cdot, \partial \Lambda) \leq \range\}} \, \Big\vert \mu\Big[b(x, \cdot)D_x\log\frac{\d\mu_{\L}}{\d \nu_{\L}}  \Big]\Big\vert \, \d x  \\
        &\lesssim_{b, \mu} \int_{\{\dist(\cdot, \partial \Lambda) \leq \range\}}  \,  \Big( \Big\vert \RelEnt_{\Lambda - x}\big(b(0, \cdot) \mu \,\big\vert\, \mu_0^! \big) \Big\vert + 1\Big) \, \d x \\
        &\leq \abs{\{\dist(\cdot, \partial \Lambda) \leq \range\}} (\vert \RelEnt\big(b(0, \cdot) \mu \,\big\vert\, \mu_0^! \big)\vert  + 1),
    \end{align*} where we used the finite-range and boundedness properties of \(b\) and the arguments of \Cref{proof:proof_of_lemma_rest_term} as well as \Cref{proof:proof_of_main_theorem_HANDA}.
    This already shows 
    \begin{align*}
        \sigma(\mu)
        = \lim_{n \uparrow \infty} \frac{1}{\abs{\Lambda_n}} \mu\Big[(-\mathscr{L}) \log\frac{\d\mu_{\L_n}}{\d \nu_{\L_n}}\Big]
        = \lim_{n \uparrow \infty} \frac{1}{\abs{\Lambda_n}}\mathcal{J}^{(\nu)}_{\L_n}(\mu\lvert \nu).
    \end{align*}
    Let us now see that in the thermodynamic limit regarding the Fisher information we can replace the boundary conditions sampled from \(\nu\) by a free boundary.
    Here, denoting \[(\mathfrak{L}^{\emptyset} g)(\eta) = \int_{\{\dist(\cdot, \partial \Lambda) \leq \range\}} \d x \ b(x,\eta)\big( f(\eta+\delta_x) - f(\eta) \big) + \sum_{x\in\eta_{\{\dist(\cdot, \partial \Lambda) \leq \range\}}} \big( f(\eta-\delta_x) - f(\eta)\big) \] and similarly \(\mathfrak{L}^{\nu} g\) with \(b_\Lambda^\nu\) in place of \(b\), we have
    \begin{align*}
         \Big\vert \mathcal{J}^{(\emptyset)}_{\L}(\mu\lvert \nu) -  \mathcal{J}^{(\nu)}_{\L}(\mu\lvert \nu) \Big\vert 
         &\lesssim_{b} \int\d x \int \nu_{\L}(\d\eta)\  b_\Lambda^{\nu}(x,\eta_{\L}) \ D_x\frac{\d\mu_{\L}}{\d\nu_\L}(\eta_{\L}) \ D_x\log\frac{\d\mu_\L}{\d\nu_\L}(\eta_{\L}) \1_{\dist(\cdot, \partial \Lambda) \leq \range\}}(x) \\
         &=  \nu\Big[\frac{\d\mu_{\L}}{\d\nu_\L} (-\mathfrak{L}^\nu) \log \frac{\d\mu_{\L}}{\d\nu_\L}   \Big]  
         \lesssim_{b} \Big\vert \nu\Big[\frac{\d\mu_{\L}}{\d\nu_\L} (- \mathfrak{L}^\emptyset) \log \frac{\d\mu_{\L}}{\d\nu_\L}   \Big] \Big\vert \\
         &= \Big\vert \mu\Big[(-\mathfrak{L}^{\emptyset}) \log \frac{\d\mu_{\L}}{\d\nu_\L} \Big] \Big\vert,
    \end{align*} by the arguments used above and the GNZ equations.
    For the right-hand side we can again recycle the arguments from above to see that
    \begin{align*}
        \Big\vert \mu\Big[(-\mathfrak{L}^{\emptyset}) \log \frac{\d\mu_{\L}}{\d\nu_\L}   \Big] \Big\vert
        &\lesssim_{b, \mu} \int_{\{\dist(\cdot, \partial \Lambda) \leq \range\}} \d x  \,  \big( \vert \RelEnt_{\Lambda - x}\big(\mu_0^! \,\big\vert\, b(0, \cdot) \mu  \big)\vert + \vert \RelEnt_{\Lambda - x}\big(b(0, \cdot) \mu \,\big\vert\, \mu_0^! \big)\vert + 1\big)\\
        &\leq \abs{\{\dist(\cdot, \partial \Lambda) \leq \range\}} (\vert \RelEnt\big(\mu_0^! \,\big\vert\, b(0, \cdot) \mu  \big)\vert + \vert \RelEnt\big(b(0, \cdot) \mu \,\big\vert\, \mu_0^! \big)\vert + 1),
    \end{align*}
    as desired. 
\end{proof}

\subsection{Quantitative decay of specific relative entropy in high-temperature regimes}\label{section:proof-quantitative-decay}

\begin{proof}[Proof of \Cref{lemma:appendix_DaiPraPosta_exponential_decay_entropy_finite_volume}]
    Define
    \begin{align*}
        \varepsilon(\beta)
        &= \e^{\beta \log_+(\Vert b \Vert_\infty )} \sup_{\eta \in \Omega, \,x \in \mathbb{R}^d} \int_{\mathbb{R}^d} \, [1-\exp(-\beta \nabla_x^+ \nabla_y^+H(\eta) - 2 \beta \log_+(\Vert b\Vert_\infty))] \, \d y \\
        &= \e^{\beta \log_+(\Vert b \Vert_\infty )} \sup_{\eta \in \Omega} \int_{B(0, \range)} \, [1 - \exp(-\beta(\nabla_y^+ \nabla_0^+H(\eta) - 2\beta \log_+(\Vert b\Vert_\infty)))] \, \d y.
    \end{align*} 
    
    We verify that
    \begin{equation}\label{eq:DaiPra_Posta_strengthened_entropy_inequality}
        \kappa \,\mathcal{E}_{\Lambda}^{\nu}(f, \log f)
        \leq \nu_\Lambda\big[((\mathscr{L}_{\Lambda}^{\nu})^2 f) \log f\big] + \nu_\Lambda\big[(\mathscr{L}_\Lambda^{\nu} f~)^2/f\big],
    \end{equation} i.e., \cite[Inequality~(2.6)]{DaiPraPosta2013} holds for good \(f\) with
    \begin{align*}
        \kappa = \kappa(\beta) = 1 - \varepsilon(\beta)
    \end{align*} without repeating everything in the cited article. The statement of our lemma then follows like in~\cite[Proposition~2.1]{DaiPraPosta2013}.  It is also evident that
    \begin{align*}
        \kappa(\beta)
        &= 1 - \e^{\beta \log_+(\Vert b \Vert_\infty )} \sup_{\eta \in \Omega} \int_{B(0, \range)} \, [1 - \exp(-\beta(\nabla_y^+ \nabla_0^+H(\eta) - 2\beta \log_+(\Vert b\Vert_\infty)))] \, \d y \\
        &\geq 1 - 4 \abs{B(0, \range)} (\log(\Vert b\Vert_\infty) \lor 1) \beta \e^{\beta \log_+(\Vert b \Vert_\infty )}  
    \end{align*} converges to \(1\) as \(\beta \rightarrow 0\), in particular \(\kappa(\beta) > 0\) for \(\beta > 0\) small enough.
    
    We just follow the steps in \cite[Section~3.1]{DaiPraPosta2013}, where the authors show a corresponding statement for the example of non-negative (``purely repulsive'') pair potentials and indicate the necessary adaptions to fit the case on hand. Their notation is borrowed too for this proof to make it as easy as possible to check the details in their article.
    Define
    \begin{align*}
        &r(\eta, \gamma_x^+, \gamma_y^+)
        := \exp(-2 \beta \log_+(\Vert b \Vert_\infty)) \exp(-\beta \nabla_x^+ \nabla_x^- H(\eta)), \\
        &r(\eta, \gamma_x^-, \gamma_y^-)
        := \begin{cases}
            1, & \text{ for } x,y \in \eta \text{ and } x \neq y, \\
            0, &\text{ otherwise},
        \end{cases} \\
        &r(\eta, \gamma_x^-, \gamma_y^+)
        := r(\eta, \gamma_x^+, \gamma_y^-)
        = 1.
    \end{align*}
    Then, \(r\) is {\em admissible} like in~\cite[Lemma~3.1]{DaiPraPosta2013}. 
    With minor modifications, we can conclude with the proof of~\cite[Theorem~3.2]{DaiPraPosta2013}, noting that for our case
    \begin{align*}
        e^{-\beta \nabla_y^+H(\eta)}
        \leq (\Vert b \Vert_\infty \lor 1)^{\beta}
        = \e^{\beta \log_+(\Vert b \Vert_\infty )}.
    \end{align*} and also
$
        1 - r(\eta, \gamma, \delta)
        \geq 0.
$
\end{proof}

\subsection*{Acknowledgments}
The authors would like to thank Iosif Pinelis and Giorgio Metafune for helpful discussions related to \Cref{lemma:special_de_la_vallee_poussin} on MathOverflow. 
BJ and JK gratefully received support by the Leibniz Association within the Leibniz Junior Research Group on \textit{Probabilistic Methods for Dynamic Communication Networks} as part of the Leibniz Competition (grant no.\ J105/2020).
BJ gratefully received support from Deutsche Forschungsgemeinschaft through DFG Project no.\ P27 within the SPP 2265.

\section{Appendix}

\subsection{Measure theory}
We collect some useful statements about entropies and expectations. 
\begin{lemma}\label{lemma:abs_log_entropy_inequality}
    Suppose \(\mu_1, \mu_2\) are probability measures defined on a common measurable space.
    Then,
    \begin{align}
        \mu_1\Big[\Big\vert \log \frac{\d \mu_1}{\d \mu_2} \Big\vert\Big]
        \leq \RelEnt(\mu_1 \,\vert\, \mu_2) + \sqrt{2 \RelEnt(\mu_1 \,\vert\, \mu_2)}.
    \end{align}
\end{lemma}
\begin{proof}
    \cite{handa1996entropy} gives the reference \cite{barron_entropy_and_the_CLT_1986} and probably means the proof of the corollary in Section~3 therein, where the inequality is proved for probability measures on \(\mathbb{R}\) with Lebesgue densities.
    We reproduce the proof in our context for the convenience of our readership. We have that
    \begin{align}
        \mu_1\Big[\Big\vert \log \frac{\d \mu_1}{\d \mu_2} \Big\vert\Big]
        = \RelEnt(\mu_1 \,\vert\, \mu_2) + 2 \mu_1\Big[\Big( \log \frac{\d \mu_1}{\d \mu_2}\Big)^{-}\Big].
    \end{align}
    Letting \(A := \big\{0 < \d \mu_1/\d \mu_2 < 1\big\}\), we have
    \begin{align}
        \mu_1\Big[\Big( \log \frac{\d \mu_1}{\d \mu_2}\Big)^{-} \Big]
        = \mu_1\Big[\log \frac{1}{\d \mu_1/\d \mu_2}\1_A \Big]
        \leq \mu_1\Big[\Big(\frac{1}{\d \mu_1/\d \mu_2} - 1\Big)\1_A \Big]
        = \mu_2(A) - \mu_1(A)
        \leq \sqrt{\frac{1}{2} \RelEnt(\mu_1 \,\vert\, \mu_2)},
    \end{align} 
    by the elementary inequality \(\log(x) \leq x-1\) for \(x > 0\) and Pinsker's inequality.
\end{proof}

\begin{lemma}\label{lemma:special_de_la_vallee_poussin}
    Given a positive random variable \(X\) with \(\mathbb{E}[X]<\infty\) and an increasing function \(g\) with \(g(x)/x \to \infty\), as $x \to \infty$, we can find a positive, increasing, convex function \(f\) with \(f(x)/x \to \infty\), as $x \to \infty$, and \(\mathbb{E}[f(X)] < \infty\)
    such that
    \(f(x)/g(x) \to 0\), as $x \to \infty$, and
    \begin{align*}
        y f(x) \leq f(xy) \leq 3 f(x) f(y)
    \end{align*} for all \(x, y \geq 1\).
\end{lemma}
The proof is based on a MathOverflow discussion of YS with Giorgio Metafune~\cite{MO_Answer_Giorgio_Metafune}. A different proof was also provided by Iosif Pinelis~\cite{MO_Answer_Iosif_Pinelis}.
\begin{proof}
    Without loss of generality we can assume \(X \geq 1\). By the standard de~la~Vallée--Poussin criterion there is some increasing \(G \colon [0, \infty) \to [0, \infty)\) such that \(G(x)/x \to \infty\), as $x \to \infty$, and \(\mathbb{E}[X G(\log X)] < \infty\). Without loss of generality, we can assume that \(x G(\log x)/g(x) \to 0\), as $x \to \infty$, that \(G\) is smooth, concave and that \(G'(0) = 1\).
    
    Now, there exists a strictly decreasing function \(h \colon [0, \infty) \to [0, \infty)\) with  \(\int_{0}^{x} \, h(t) \, \d t \sim G(x)\) that fulfills the differential inequality \(h' \geq -h\).
    To put it shortly, the solution \(h\) to the ODE \(h + h' = G'\) with \(h(0) = 1\) works.
    Now set
    \begin{align*}
        f(x) = x \, \widetilde{G}(\log x)
    \end{align*} for
    \begin{align*}
        \widetilde{G}(x)
        = \int_{0}^{x} \, h(t)  \, \d t.
    \end{align*}
    Then, \(f\) fulfills the conditions of this lemma because \(\widetilde{G}\) is subadditive by concavity and \(\widetilde{G}(0) = 0\), but the function \(f\) is still convex by \(h' \geq -h\).
\end{proof}

\bibliographystyle{alpha}
\bibliography{references}

@article{Jessen1934NoteOT,
  title={Note on the differentiability of multiple integrals},
  author={B. Jessen and J. Marcinkiewicz and A. Zygmund},
  journal={Fundam. Math.},
  year={1934},
  volume={23},
  pages={217-234},
}

@article{vossboehme,
  title={Gibbsian Characterization for the
Reversible Measures of Interacting
Particle Systems},
  author={A. Vo{\ss}-B{\"o}hme},
  journal={Markov Processes Relat. Fields},
  year={2009},
  volume={15},
  pages={441-476},
}

@article{garcia2006spatial,
  title={Spatial birth and death processes as solutions of stochastic equations},
  author={Garcia, N. L. and Kurtz, T. G.},
  journal={ALEA},
  volume={1},
  pages={281--303},
  year={2006}
}

@article{JK25,
  title={On the long-time behaviour of reversible interacting particle systems in one and two dimensions},
  author={Jahnel, B. and K{\"o}ppl, J.},
  journal={Probab. Math. Phys.},
  volume={6},
  number={2},
  pages={479--503},
  year={2025}
}

@article{xanh1979integral,
  title={Integral and differential characterizations of the {G}ibbs process},
  author={Nguyen, X. X. and Zessin, H.},
  journal={Math. Nachr.},
  volume={88},
  number={1},
  pages={105--115},
  year={1979},
  publisher={Wiley Online Library}
}

@article{georgii1976canonical,
  title={Canonical and grand canonical {G}ibbs states for continuum systems},
  author={Georgii, H.-O.},
  journal={Comm. Math. Phys.},
  volume={48},
  pages={31--51},
  year={1976},
  publisher={Springer}
}

@article{handa1996entropy,
  title={Entropy production per site in (nonreversible) spin-flip processes},
  author={Handa, K.},
  journal={J. Stat. Phys.},
  volume={83},
  pages={555--571},
  year={1996},
  publisher={Springer}
}

@article{nguyen1979ergodic,
  title={Ergodic theorems for spatial processes},
  author={Nguyen, X. X. and Zessin, H.},
  journal={Z. Wahrscheinlichkeitstheorie Verw. Gebiete},
  volume={48},
  number={2},
  pages={133--158},
  year={1979},
  publisher={Springer}
}

@article{penrose2008spatial,
author = {M. D. Penrose},
title = {{Existence and spatial limit theorems for lattice and continuum particle systems}},
volume = {5},
journal = {Probab. Surv.},
publisher = {Institute of Mathematical Statistics and Bernoulli Society},
pages = {1 -- 36},
year = {2008},
doi = {10.1214/07-PS112},
URL = {https://doi.org/10.1214/07-PS112}
}

@article {Sullivan1976,
    AUTHOR = {Sullivan, W. G.},
     TITLE = {Specific information gain for interacting {M}arkov processes},
   JOURNAL = {Z. Wahrscheinlichkeitstheorie Verw. Gebiete},
  FJOURNAL = {Zeitschrift f\"ur Wahrscheinlichkeitstheorie und Verwandte
              Gebiete},
    VOLUME = {37},
      YEAR = {1976/77},
    NUMBER = {1},
     PAGES = {77--90},
       DOI = {10.1007/BF00536299},
       URL = {https://doi.org/10.1007/BF00536299},
}

@book {DaleyVereJones2008,
    AUTHOR = {Daley, D. J. and Vere-Jones, D.},
     TITLE = {An Introduction to the Theory of Point Processes. {V}ol. {II}},
    SERIES = {Probability and its Applications (New York)},
   EDITION = {Second},
 PUBLISHER = {Springer, New York},
      YEAR = {2008},
}

@book {DaleyVereJones2003,
    AUTHOR = {Daley, D. J. and Vere-Jones, D.},
     TITLE = {An Introduction to the Theory of Point Processes. {V}ol. {I}},
    SERIES = {Probability and its Applications (New York)},
   EDITION = {Second},
 PUBLISHER = {Springer, New York},
      YEAR = {2003},
}

@book{EK86,
	author = {Ethier, S. N. and Kurtz, T. G.},
	doi = {10.1002/9780470316658},
	isbn = {9780471081869},
	issn = {1940-6347},
	year = {1986},
	title = {Markov {Processes}},
	url = {http://dx.doi.org/10.1002/9780470316658},
        publisher = {John Wiley \& Sons Inc.},
        series = {Wiley Series in Probability and Mathematical Statistics: Probability and Mathematical Statistics},
}

@article{Dereudre_Vasseur_2020, title={Existence of {G}ibbs point processes with stable infinite range interaction}, volume={57}, 
DOI={10.1017/jpr.2020.39}, 
number={3}, 
journal={J. Appl. Probab.}, 
author={Dereudre, D. and Vasseur, T.}, 
year={2020}, 
pages={775–791}}

@article{HS78,
    author = {R. A. Holley and D. W. Stroock},
    title = {{Nearest neighbor birth and death processes on the real line}},
    volume = {140},
    journal = {Acta Math.},
    publisher = {Institut Mittag-Leffler},
    pages = {103 -- 154},
    year = {1978},
    doi = {10.1007/BF02392306},
    URL = {https://doi.org/10.1007/BF02392306}
}

@article{CCK95,
author = {J. T. Chayes and L. Chayes and R. Koteck{\'y}},
title = {{The analysis of the Widom--Rowlinson model by stochastic geometric methods}},
volume = {172},
journal = {Comm. Math. Phys.},
number = {3},
publisher = {Springer},
pages = {551 -- 569},
year = {1995},
}

@article{Ruelle1971WRPT,
  title = {Existence of a Phase Transition in a Continuous Classical System},
  author = {Ruelle, D.},
  journal = {Phys. Rev. Lett.},
  volume = {27},
  issue = {16},
  pages = {1040--1041},
  numpages = {0},
  year = {1971},
  publisher = {American Physical Society},
  doi = {10.1103/PhysRevLett.27.1040},
  url = {https://link.aps.org/doi/10.1103/PhysRevLett.27.1040}
}

@article{Giacomin1995Agreement,
	author = {Giacomin, G. and Lebowitz, J. L. and Maes, C.},
	journal = {J. Stat. Phys.},
	doi = {10.1007/bf02179875},
	issn = {0022-4715},
	number = {5-6},
	year = {1995},
	pages = {1379--1403},
	publisher = {{Springer Science and Business Media LLC}},
	title = {Agreement percolation and phase coexistence in some {Gibbs} systems},
	url = {http://dx.doi.org/10.1007/BF02179875},
	volume = {80},
}

@article{KonSko06Contact,
author = {Kondratiev, Y. and Skorokhod, A.},
title = {ON CONTACT PROCESSES IN CONTINUUM},
journal = {Infin.~Dimens.~Anal.~Quantum Probab.~Relat.~Top.},
volume = {09},
number = {02},
pages = {187-198},
year = {2006},
doi = {10.1142/S0219025706002305},
URL = {https://doi.org/10.1142/S0219025706002305},
eprint = {https://doi.org/10.1142/S0219025706002305}
}

@article{Finkelshtein2014Dynamical,
	author = {Finkelshtein, D. and Kondratiev, Y. and Kutoviy, O. and Oliveira, M. J.},
	journal = {J. Stat. Phys.},
	doi = {10.1007/s10955-014-1124-6},
	issn = {0022-4715},
	number = {1},
	year = {2014},
	pages = {57--86},
	publisher = {{Springer Science and Business Media LLC}},
	title = {Dynamical {Widom}--{Rowlinson} model and its mesoscopic limit},
	url = {http://dx.doi.org/10.1007/s10955-014-1124-6},
	volume = {158},
}

@book{Kallenberg1983Random,
	author = {Kallenberg, O.},
	year = {1983},
	publisher = {Academic Press},
	title = {Random {Measures}},
	url = {https://books.google.com/books/about/Random_Measures.html?hl=&id=McWmAAAAIAAJ},
}

@article {Preston1975,
    AUTHOR = {Preston, C.},
     TITLE = {Spatial birth-and-death processes},
   JOURNAL = {Bull. Inst. Internat. Statist.},
  FJOURNAL = {Bulletin de l'Institut International de Statistique},
    VOLUME = {46},
      YEAR = {1975},
    NUMBER = {2},
     PAGES = {371--391, 405--408},
   MRCLASS = {60J80 (60K35)},
  MRNUMBER = {474532},
MRREVIEWER = {T.\ M.\ Liggett},
}

@article{huesmann2025,
      title={Non-local {W}asserstein Geometry, Gradient Flows, and Functional Inequalities for Stationary Point Processes}, 
      author={M. Huesmann and H. Stange},
      year={2025},
      journal={Preprint arXiv:2504.12047},
      doi={arXiv:2504.12047},
      url={https://arxiv.org/abs/2504.12047}, 
}

@article{JKSZ24,
      title={The variational principle for a marked {G}ibss point process with infinite-range multibody interactions}, 
      author={Jahnel, B. and Köppl, J. and Steenbeck, Y. and Zass, A.},
      year={2024},
      journal={Preprint arXiv:2408.17170},
      doi={arXiv:2408.17170},
      url={https://arxiv.org/abs/2408.17170}, 
}

@article{Dello2024Wasserstein,
	author = {Dello Schiavo, L. and Herry, R. and Suzuki, K.},
	journal = {	J. {\'E}c. Polytech. Math.},
	doi = {10.5802/jep.270},
	issn = {2270-518X},
	year = {2024},
	pages = {957--1010},
	publisher = {Cellule MathDoc/Centre Mersenne},
	title = {Wasserstein geometry and {Ricci} curvature bounds for {Poisson} spaces},
	url = {http://dx.doi.org/10.5802/jep.270},
	volume = {11},
}

@article{Kond2008, 
    title={Correlation functions and invariant measures in continuous contact model}, 
    volume={11}, 
    DOI={10.1142/s0219025708003038}, 
    number={02}, 
    journal={Infin.~Dimens.~Anal.~Quantum Probab.~Relat.~Top.}, 
    author={Kondratiev, Y. and Kutoviy, O. and Pirogov, S.}, 
    year={2008}, 
    pages={231–258}
}

@article{Kurtz1980Representations,
	author = {Kurtz, T. G.},
	journal = {Ann. Probab.},
	doi = {10.1214/aop/1176994660},
	issn = {0091-1798},
	number = {4},
	year = {1980},
	publisher = {Institute of Mathematical Statistics},
	title = {Representations of {Markov} {Processes} as multiparameter time changes},
	url = {http://dx.doi.org/10.1214/aop/1176994660},
	volume = {8},
}

@article{Garcia1995Birth,
	author = {Garcia, N. L.},
	journal = {Adv. Appl. Probab.},
	doi = {10.2307/1427928},
	issn = {0001-8678},
	number = {4},
	year = {1995},
	pages = {911--930},
	publisher = {Cambridge University Press (CUP)},
	title = {Birth and death processes as projections of higher-dimensional {Poisson} processes},
	url = {http://dx.doi.org/10.2307/1427928},
	volume = {27},
}

@article{Etheridge2019Genealogical,
	author = {Etheridge, A. M. and Kurtz, T. G.},
	journal = {Ann. Probab.},
	doi = {10.1214/18-aop1266},
	issn = {0091-1798},
	number = {4},
	year = {2019},
	publisher = {Institute of Mathematical Statistics},
	title = {Genealogical constructions of population models},
	url = {http://dx.doi.org/10.1214/18-AOP1266},
	volume = {47},
}

@article{holley_free_1971,
    title = {Free energy in a {Markovian} model of a lattice spin system},
    volume = {23},
    issn = {0010-3616, 1432-0916},
    url = {http://link.springer.com/10.1007/BF01877751},
    doi = {10.1007/BF01877751},
    language = {en},
    number = {2},
    urldate = {2021-11-10},
    journal = {Comm. Math. Phys.},
    author = {Holley, R. A.},
    year = {1971},
    pages = {87--99},
}

@article{higuchi_results_1975,
    title = {Some results on {Markov} processes of infinite lattice spin systems},
    volume = {15},
    issn = {2156-2261},
    url = {https://projecteuclid.org/journals/kyoto-journal-of-mathematics/volume-15/issue-1/Some-results-on-Markov-processes-of-infinite-lattice-spin-systems/10.1215/kjm/1250523126.full},
    doi = {10.1215/kjm/1250523126},
    number = {1},
    urldate = {2021-11-10},
    journal = {Kyoto J. Math.},
    author = {Higuchi, Y. and Shiga, T.},
    year = {1975},
}

@article{kunsch_time_1984,
    title = {Non Reversible stationary measures for infinite interacting particle systems},
    volume = {66},
    journal = {Z. Wahrscheinlichkeitstheorie Verw. Gebiete},
    author = {Künsch, H.},
    year = {1984},
    pages = {407--424},
}

@article{jahnel_attractor_2019,
    title = {Attractor properties for irreversible and reversible interacting particle systems},
    volume = {366},
    issn = {0010-3616, 1432-0916},
    url = {http://link.springer.com/10.1007/s00220-019-03352-4},
    doi = {10.1007/s00220-019-03352-4},
    language = {en},
    number = {1},
    urldate = {2020-10-23},
    journal = {Comm. Math. Phys.},
    author = {Jahnel, B. and Külske, C.},
    year = {2019},
    pages = {139--172},
}

@article{jahnel_dynamical_2023,
    title = {Dynamical {Gibbs} variational principles for irreversible interacting particle systems with applications to attractor properties},
    volume = {33},
    issn = {1050-5164},
    url = {https://projecteuclid.org/journals/annals-of-applied-probability/volume-33/issue-6A/Dynamical-Gibbs-variational-principles-for-irreversible-interacting-particle-systems-with/10.1214/22-AAP1926.full},
    doi = {10.1214/22-AAP1926},
    number = {6A},
    urldate = {2024-05-20},
    journal = {Ann. Appl. Probab.},
    author = {Jahnel, B. and Köppl, J.},
    year = {2023},
}

@article{holley_one_1977,
    title = {In one and two dimensions, every stationary measure for a stochastic {Ising} model is a {Gibbs} state},
    volume = {55},
    issn = {0010-3616, 1432-0916},
    url = {http://link.springer.com/10.1007/BF01613147},
    doi = {10.1007/BF01613147},
    language = {en},
    number = {1},
    urldate = {2022-11-28},
    journal = {Comm. Math. Phys.},
    author = {Holley, R. A. and Stroock, D. W.},
    year = {1977},
    pages = {37--45},
}

@book{liggett_interacting_2005,
    address = {Berlin, Heidelberg},
    series = {Classics in {Mathematics}},
    title = {Interacting {Particle} {Systems}},
    isbn = {978-3-540-22617-8 978-3-540-26962-5},
    url = {http://link.springer.com/10.1007/b138374},
    urldate = {2022-10-06},
    publisher = {Springer Berlin Heidelberg},
    author = {Liggett, T. M.},
    year = {2005},
    doi = {10.1007/b138374},
}

@article{georgii_equivalence_1995,
    title = {The equivalence of ensembles for classical systems of particles},
    volume = {80},
    issn = {1572-9613},
    url = {https://doi.org/10.1007/BF02179874},
    doi = {10.1007/BF02179874},
    number = {5},
    journal = {J. Stat. Phys.},
    author = {Georgii, H.-O.},
    year = {1995},
    pages = {1341--1378},
}

@article{barron_entropy_and_the_CLT_1986,
author = {A. R. Barron},
title = {Entropy and the Central Limit Theorem},
volume = {14},
journal = {Ann. Probab.},
number = {1},
publisher = {Institute of Mathematical Statistics},
pages = {336 -- 342},
year = {1986},
}

@article{JKSZ25,
      title={Reversible birth-and-death dynamics in continuum: free-energy dissipation and attractor properties}, 
      author={B. Jahnel and J. Köppl and Y. Steenbeck and A. Zass},
      year={2025},
      journal={Preprint arXiv:2508.21196} ,
      doi={arXiv:2508.21196},
      url={https://arxiv.org/abs/2508.21196},
}

@article{Ahle2022,
title = {Sharp and simple bounds for the raw moments of the binomial and Poisson distributions},
journal = {Stat. Probab. Lett.},
volume = {182},
pages = {109306},
year = {2022},
author = {T. D. Ahle},
}

@article{Dereudre2025,
author = {D. Dereudre and C. Renaud-Chan},
title = {{Liquid-gas phase transition for Gibbs point process with Quermass interaction}},
volume = {30},
journal = {Electron. J. Probab.},
publisher = {Institute of Mathematical Statistics and Bernoulli Society},
pages = {1 -- 32},
year = {2025},
doi = {10.1214/25-EJP1361},
URL = {https://doi.org/10.1214/25-EJP1361}
}

@article{Georgii_1996,
author = {H.-O. Georgii and O. H{\aa}ggstr{\"o}m},
title = {{Phase transition in continuum Potts models}},
volume = {181},
journal = {Commun. Math. Phys.},
number = {2},
publisher = {Springer},
pages = {507 -- 528},
year = {1996},
}

@article{Roelly_Zass_2020, 
    title={Marked {G}ibbs Point processes with unbounded interaction: {A}n existence result}, 
    volume={179}, 
    DOI={10.1007/s10955-020-02559-3}, 
    number={4}, 
    journal={J. Stat. Phys.}, 
    author={R{\oe}lly, S. and Zass, A.}, 
    year={2020}, 
    month={May}, 
    pages={972--996}
}

@book{ruelle_1969, 
place={London}, 
title={Statistical Mechanics: Rigorous Results}, 
publisher={W. A. Benjamin, Inc., New York-Amsterdam}, 
author={Ruelle, D.}, 
year={1969}
}

@article{dereudre2009variational,
  title={Variational characterisation of {G}ibbs measures with {D}elaunay triangle interaction},
  author={Dereudre, D. and Georgii, H.-O.},
volume = {14},
journal = {Electron. J. Probab.},
  year={2009}
}

@book{Bakry2014,
  title = {Analysis and Geometry of Markov Diffusion Operators},
  ISBN = {9783319002279},
  ISSN = {2196-9701},
  url = {http://dx.doi.org/10.1007/978-3-319-00227-9},
  DOI = {10.1007/978-3-319-00227-9},
  journal = {Grundl. Math. Wiss.},
  publisher = {Springer International Publishing},
  author = {Bakry, D. and Gentil, I. and Ledoux, M.},
  year = {2014}
}

@incollection{dereudre_2019,
author={Dereudre, D.},
editor={Coupier, David},
title={Introduction to the theory of {G}ibbs Point Processes},
bookTitle={Stochastic {G}eometry: Modern Research Frontiers},
year={2019},
publisher={Springer International Publishing},
address={Cham},
pages={181--229},
isbn={978-3-030-13547-8},
doi={10.1007/978-3-030-13547-8_5},
url={https://doi.org/10.1007/978-3-030-13547-8_5}
}

@MISC{MO_Answer_Iosif_Pinelis,
    TITLE = {Is there a increasing, convex, superlinear $f$ with $c_1 f(x)y \leq f(xy)\leq c_2 f(x)f(y)$ such that \(\mathbb{E}[f(X)] < \infty\)? (answer)},
    AUTHOR = {I. Pinelis},
    HOWPUBLISHED = {MathOverflow},
    DATE = {2025-11-26},
    NOTE = {URL: \url{https://mathoverflow.net/a/504347} (visited on: 2026-02-10)},
    URL = {https://mathoverflow.net/a/504347},
    URLDATE = {2026-02-10},
}

@MISC{MO_Answer_Giorgio_Metafune,
    TITLE = {Can positive decreasing functions $f$ with $f' \geq -f$? be such that $\int_{0}^{x} f(t) \mathrm{d}t$ diverges arbitrarily slow for $x \to \infty$? (answer)},
    AUTHOR = {G. Metafune},
    HOWPUBLISHED = {MathOverflow},
    DATE = {2025-11-29},
    NOTE = {URL: \url{https://mathoverflow.net/a/504366} (visited on: 2026-02-10)},
    URL = {https://mathoverflow.net/a/504366},
    URLDATE = {2026-02-10},
}

@article{DaiPraPosta2013,
    title={Entropy decay for interacting systems via the {B}ochner–{B}akry–{É}mery approach},
    author={Dai Pra, P. and Posta, G.},
    volume = {18},
    journal = {Electron. J. Probab.},
    year={2013}
}

@article{Fernndez2007,
  title = {The Analyticity Region of the Hard Sphere Gas. {I}mproved Bounds},
  volume = {128},
  ISSN = {1572-9613},
  url = {http://dx.doi.org/10.1007/s10955-007-9352-7},
  DOI = {10.1007/s10955-007-9352-7},
  number = {5},
  journal = {J. Stat. Phys.},
  publisher = {Springer Science and Business Media LLC},
  author = {Fernández, R. and Procacci, A. and Scoppola, B.},
  year = {2007},
  month = jun,
  pages = {1139–1143}
}

@article{Betsch2023,
  title = {On the uniqueness of {G}ibbs distributions with a non-negative and subcritical pair potential},
  volume = {59},
  ISSN = {0246-0203},
  url = {http://dx.doi.org/10.1214/22-AIHP1265},
  DOI = {10.1214/22-aihp1265},
  number = {2},
  journal = {Ann. Inst. Henri Poincaré Probab. Stat.},
  publisher = {Institute of Mathematical Statistics},
  author = {Betsch, S. and Last, G.},
  year = {2023},
  month = may 
}

@article{Houdebert2022,
  title = {An explicit {D}obrushin uniqueness region for {G}ibbs point processes with repulsive interactions},
  volume = {59},
  ISSN = {1475-6072},
  url = {http://dx.doi.org/10.1017/jpr.2021.70},
  DOI = {10.1017/jpr.2021.70},
  number = {2},
  journal = {J. Appl. Probab.},
  publisher = {Cambridge University Press (CUP)},
  author = {Houdebert, P. and Zass, A.},
  year = {2022},
  month = mar,
  pages = {541–555}
}

@article{dereudre2026,
      title={First-order phase transition for Gibbs point processes with saturated interactions}, 
      author={D. Dereudre and C. Renaud-Chan},
      year={2026},
      journal={Preprint arXiv:2602.11078} ,
      doi={arXiv:2602.11078},
      url={https://arxiv.org/abs/2602.11078}
}

\end{document}